\documentclass[11pt,a4paper]{amsart}
\usepackage{a4wide}
\usepackage[latin1]{inputenc}
\usepackage{amsmath}
\usepackage{amsfonts}
\usepackage{amssymb}
\usepackage{graphicx}
\usepackage{amsthm}
\usepackage{amssymb}
\usepackage{cite}
\usepackage{mathrsfs}
\usepackage{hyperref}
\usepackage{paralist}
\usepackage{enumitem}
\usepackage{tikz}
\usetikzlibrary{matrix,chains}
\usepackage{dynkin-diagrams}

\begin{document}

	\newcommand{\n}{\mathbf{n}}
	\newcommand{\x}{\mathbf{x}}
	\newcommand{\h}{\mathbf{h}}
	\newcommand{\m}{\mathbf{m}}
		\newcommand{\Ph}{\Phi_}
	
		\newcommand{\bL}{\mathbf{L}}
	\newcommand{\bG}{\mathbf{G}}
	\newcommand{\bS}{\mathbf{S}}
	\newcommand{\bZ}{\mathbf{Z}}
	\newcommand{\bH}{\mathbf{H}}
	\newcommand{\B}{\mathbf{B}}
	\newcommand{\U}{\mathbf{U}}
	\newcommand{\V}{\mathbf{V}}
	\newcommand{\T}{\mathbf{T}}
	\newcommand{\G}{\mathbf{G}}
	\newcommand{\Para}{\mathbf{P}}
	\newcommand{\Levi}{\mathbf{L}}
		\newcommand{\bM}{\mathbf{M}}
	\newcommand{\Y}{\mathbf{Y}}
	\newcommand{\X}{\mathbf{X}}
	\newcommand{\M}{\mathbf{M}}
	\newcommand{\pro}{\mathbf{prod}}
	\renewcommand{\o}{\overline}
	
	\newcommand{\Gtilde}{\mathbf{\tilde{G}}}
	\newcommand{\Ttilde}{\mathbf{\tilde{T}}}
	\newcommand{\Btilde}{\mathbf{\tilde{B}}}
	\newcommand{\Ltilde}{\mathbf{\tilde{L}}}
	\newcommand{\C}{\operatorname{C}}
	
	\newcommand{\N}{\operatorname{N}}
	\newcommand{\bl}{\operatorname{bl}}
	\newcommand{\Z}{\operatorname{Z}}
	\newcommand{\Gal}{\operatorname{Gal}}
	\newcommand{\modulo}{\operatorname{mod}}
	\newcommand{\kernel}{\operatorname{ker}}
	\newcommand{\Irr}{\operatorname{Irr}}
	\newcommand{\D}{\operatorname{D}}
	\newcommand{\I}{\operatorname{I}}
	\newcommand{\GL}{\operatorname{GL}}
	\newcommand{\SL}{\operatorname{SL}}
	\newcommand{\W}{\operatorname{W}}
	\newcommand{\R}{\operatorname{R}}
	\newcommand{\Br}{\operatorname{Br}}
	\newcommand{\Aut}{\operatorname{Aut}}
	\newcommand{\End}{\operatorname{End}}
	\newcommand{\Ind}{\operatorname{Ind}}
	\newcommand{\Res}{\operatorname{Res}}
	\newcommand{\br}{\operatorname{br}}
	\newcommand{\Hom}{\operatorname{Hom}}
	\newcommand{\Endo}{\operatorname{End}}
	\newcommand{\Ho}{\operatorname{H}}
	\newcommand{\Tr}{\operatorname{Tr}}
	\newcommand{\opp}{\operatorname{opp}}
		\newcommand{\ssc}{\operatorname{sc}}
			\newcommand{\ad}{\operatorname{ad}}
			
				\newcommand{\tw}[1]{{}^#1\!}
	
		\theoremstyle{remark}
	
	\theoremstyle{definition}
	\newtheorem{definition}{Definition}[section]
	\newtheorem{construction}[definition]{Construction}
	\newtheorem{remark}[definition]{Remark}
	\newtheorem{example}[definition]{Example}
	\newtheorem{notation}[definition]{Notation}
	\newtheorem{question}[definition]{Question}
	
	\theoremstyle{plain}
	\newtheorem{theorem}[definition]{Theorem}
	\newtheorem{lemma}[definition]{Lemma}
	\newtheorem{proposition}[definition]{Proposition}
	\newtheorem{corollary}[definition]{Corollary}
	\newtheorem{conjecture}[definition]{Conjecture}
	\newtheorem{assumption}[definition]{Assumption}
	\newtheorem{main theorem}[definition]{Main Theorem}
	\newtheorem{hypothesis}[definition]{Hypothesis}
	
	\newtheorem*{theo*}{Theorem}
	\newtheorem*{conj*}{Conjecture}
	\newtheorem*{cor*}{Corollary}
	\newtheorem*{propo*}{Proposition}
	
	\newtheorem{theo}{Theorem}
	\newtheorem{conj}[theo]{Conjecture}
	\newtheorem{cor}[theo]{Corollary}
	\newtheorem{propo}[theo]{Proposition}
	
	\renewcommand{\thetheo}{\Alph{theo}}
	\renewcommand{\theconj}{\Alph{conj}}
	\renewcommand{\thecor}{\Alph{cor}}

	\title{The Alperin--McKay conjecture for the prime 2}

\date{\today}
\author{Lucas Ruhstorfer}
\address{Fachbereich Mathematik, TU Kaiserslautern, 67653 Kaiserslautern, Germany}
\email{ruhstorfer@mathematik.uni-kl.de}
\keywords{Alperin-McKay conjecture, exceptional groups of Lie type, quasi-isolated blocks}

\subjclass[2010]{20C33}

\begin{abstract}
	In this paper we consider the inductive Alperin--McKay condition for quasi-isolated $2$-blocks of exceptional groups of Lie type.
 Thereby, we complete the proof of the Alperin--McKay conjecture for the prime $2$.
\end{abstract}

\maketitle

	\section*{Introduction}

\subsection*{Alperin--McKay conjecture}
Let $\ell$ be a prime and $b$ an $\ell$-block of a finite group $G$ with defect group $D$. We denote by $\Irr_0(G,b)$ the set of height zero characters of $b$, i.e. the set of characters $\chi \in \Irr(G,b)$ associated to the block $b$ with $\chi(1)_\ell=|G:D|_\ell$. Then the Alperin--McKay conjecture \cite{Alperin} from 1976 predicts the following:

\begin{conj*}[Alperin--McKay]
	Let $b$ be an $\ell$-block of $G$ with defect group $D$ and $B$ its Brauer correspondent in $\mathrm{N}_G(D)$. Then $$|\Irr_0(G,b)|= |\Irr_0(\mathrm{N}_G(D),B)|.$$
\end{conj*}

Our main objective in this paper is to prove this conjecture for the prime $2$.

\begin{theo}\label{AM}
The Alperin--McKay conjecture holds for the prime $\ell=2$.
\end{theo}

Späth \cite[Theorem C]{IAM} showed that the Alperin--McKay conjecture holds for the prime $\ell$ if the so-called inductive Alperin--McKay (iAM) condition holds for all finite simple groups with respect to the prime $\ell$. Our proof of Theorem \ref{AM} therefore shows more precisely that the iAM condition holds for all finite simple groups with respect to the prime $2$.

\subsection*{Brauer's height zero conjecture}
This particular way of proving Alperin's conjecture has important consequences for Brauer's height zero (BHZ) conjecture \cite{Brauer} from 1955 which states that if $b$ is an $\ell$-block of $G$ with defect group $D$, then $\Irr(G,b)=\Irr_0(G,b)$ if and only if $b$ is abelian. The "if"-direction of this conjecture was fully established by Kessar--Malle \cite{KessarMalle} while on the other hand, Navarro--Späth \cite{JEMS} showed that the "only-if" direction of BHZ would follow from a proof of the Alperin--McKay conjecture via the inductive conditions. As a consequence of our proof of Theorem \ref{AM}, we therefore also obtain Brauer's height zero conjecture for the prime $2$.

\begin{cor}[BHZ for the prime $2$]
	Let $b$ be a $2$-block of a finite group $G$ with defect
	group $D$. Then $\Irr(b)=\Irr_0(b)$ if and only if $D$ is abelian.
\end{cor}

Our strategy towards the proof of Theorem \ref{AM} is now as follows.
By \cite{BroughRuhstorfer} the iAM-condition holds for all finite simple groups at the prime $\ell=2$ except possibly for groups of exceptional Lie type defined over a field of odd characteristic. Moreover, by the main result of \cite{Jordan2} it suffices to check the inductive condition only for the quasi-isolated blocks of these groups. Throughout the introduction, $\G$ therefore denotes a simple simply connected linear algebraic group of
exceptional Lie type in characteristic $p \neq 2$ with Frobenius endomorphism $F: \G \to \G$ defining an $\mathbb{F}_q$-structure on $\G$ and $G:=\G^F$ its associated finite group of Lie type. 

 To finish the proof of Theorem \ref{AM} we therefore show the following:

\begin{theo}\label{AM good}
 Let $G$ be an exceptional group of Lie type as above and $b$ a quasi-isolated $2$-block of $G$. Then the block $b$ satisfies the iAM-condition.
\end{theo}

\subsection*{Jordan decomposition for quasi-isolated blocks}

In a first step towards the proof of Theorem \ref{AM good} we describe the structure of the quasi-isolated blocks of groups of exceptional type. For this recall that a quasi-isolated $2$-block is a block associated to a quasi-isolated semisimple element $s \in (\G^\ast)^{F}$ of odd order in the group $(\G^\ast,F)$ in duality with $(\G,F)$. In \cite{KessarMalle} Malle--Kessar give a parametrization of these blocks in terms of quasi-central $e$-cuspidal pairs, where $e$ denotes the order of $q$ modulo $4$. More precisely, the set of $2$-blocks associated to the semisimple element $s$ is in bijection with the set of $G$-conjugacy classes of quasi-central $e$-cuspidal pairs $(\Levi,\lambda)$ with $\lambda \in \mathcal{E}(\Levi,s)$.

Motivated by the results of Enguehard \cite{EnguehardJordandecomp} and by using Malle--Kessar's parametrization we construct a Jordan decomposition for quasi-isolated blocks of exceptional groups. Following \cite{EnguehardJordandecomp} we construct for this an $F$-stable subgroup $\G(s) \leq \G$ in "duality" with $\C_{\G^\ast}(s)$. From a philosophical point of view the representation theory of a block associated to $s$ should be closely linked to the representation theory of a corresponding unipotent block of the group $G(s):=\G(s)^F$. Using Kessar--Malle's parametrization and comparing it with the parametrization of the involved unipotent blocks we can now prove the following:

\begin{propo}\label{jordan}
In the above setting, assume that $\mathrm{C}_{\G^\ast}(s)^F=\mathrm{C}^\circ_{\G^\ast}(s)^F$.
	\begin{enumerate}[label=(\alph*)]
		\item There exists a bijection between the set of $2$-blocks of $G$ associated to $s$ and the set of unipotent $2$-blocks of $G(s)$.
		\item There exists a bijection $\mathcal{E}_\ell(G,s) \to \mathcal{E}_\ell(G(s),s)$ which induces a bijection $\mathcal{J}:\Irr_0(G,b) \to  \Irr_0(G(s),c)$, where the blocks $b$ and $c$ correspond under the bijection of (a).
		\item Assume that $\mathrm{Out}(G)_b$ is cyclic. Then there exists a bijection $\mathcal{J}$ as in (b) which is $\mathrm{Aut}(G)_{G(s),c}$-equivariant.
		\end{enumerate}
\end{propo}



%


\subsection*{Determination of defect groups}

Our next aim is to relate the local structure of the blocks of $G$ associated to $s$ to the structure of the corresponding unipotent block of $G(s)$. In \cite{KessarMalle} the authors give a description of the defect group in terms of this $e$-cuspidal pair. Except for a few cases in type $E_8$ all quasi-isolated $2$-blocks have maximal possible defect, i.e. their defect group has order $|\C_{(\G^\ast)^{F}}(s)|_2$. For this reason, we call blocks associated to the semisimple element $s$ with defect $|\C_{(\G^\ast)^{F}}(s)|_2$ maximal blocks. Building on Kessar--Malle's description of defect groups we obtain the following result which shows that the local categories of Jordan correspondent blocks are related.
	

		\begin{propo}\label{defect}
		Let $b$ be a maximal quasi-isolated $2$-block of $G$ associated to the semisimple element $s$.
		\begin{enumerate}[label=(\alph*)]
			\item  The defect group $P$ is conjugate to a Sylow $2$-subgroup of $G(s)$.
			\item Suppose that the centralizer of $s$ is not of type $A_2(\pm q^3)$ or ${}^3 D_4(q).3$. Then we have $\N_G(P,b_P) = \N_{G(s)}(P)$.
		\end{enumerate}
	\end{propo}

As an interesting byproduct of the previous proposition we obtain the following theorem about Sylow $2$-subgroups, which appears to be new:

\begin{theo}
	Let $\mathbf{H}$ be a simple algebraic group of simply connected type $A_n$, $D_{2n+1}$ or $E_6$ with $n >1$ and $F: \bH \to \bH$ a Frobenius endomorphism defining an $\mathbb{F}_q$-structure for some odd prime power $q$. Then the Sylow $2$-subgroups of $\mathbf{H}^F$ are Cabanes (see Definition \ref{Cabanes def}).
\end{theo}

We also have a result for other groups $\mathbf{H}$ with different root system but under an additional assumption on $q$, see Theorem \ref{Cabanes}.

\subsection*{Verification of the inductive condition}

After having described the structure of these blocks, we turn towards proving the inductive AM-condition. As this condition is vastly simpler for blocks $b$ with cyclic stabilizer $\mathrm{Out}(G)_b$ we will focus on those first. It turns out that most quasi-isolated blocks actually have such a cyclic stabilizer in the outer automorphism group. As a consequence of Proposition \ref{jordan} and Proposition \ref{defect} we can therefore show that these blocks satisfy the iAM-condition.

%

%
%
	
	
	After having dealt with blocks $b$ with cyclic $\mathrm{Out}(G)_b$, we shift our attention to the remaining maximal blocks.
	Any height zero character of such a block corresponds under the Jordan decomposition constructed above to a $2'$-character of $G(s)$. Building on Malle--Späth's \cite{MS} parametrization of $2'$-characters via Harish-Chandra theory we are able to prove the following.
	
	
	\begin{theo}\label{HC intro}
		Any height zero character in $\mathcal{E}_\ell(G,s) \cap \Irr_0(G,b)$ of a maximal block $b$ lies in the Harish-Chandra series of a semisimple character in $\mathcal{E}_\ell(\Levi_0,s)$. 
	\end{theo}
	
	Here, $\Levi_0$ is a Levi subgroup dual to the minimal $1$-split Levi subgroup $\Levi_0^\ast$ of $\G^\ast$ which contains $s$, see Lemma \ref{minimal split}. The previous theorem yields in particular a parametrization of the height zero characters of maximal blocks in terms of Harish-Chandra series. Let $(\Levi,\lambda)$ be the $e$-cuspidal pair associated to the block $b$ and $B$ the unique block of $\N_{G}(\Levi)$ with $B^G=b$. We are able to parametrize the local characters $\Irr_0(\N_{G}(\Levi),B)$ and combining this with the parametrization obtained from Theorem \ref{HC intro} we provide an iAM-bijection $\Irr_0(\N_{G}(\Levi),B) \to \Irr_0(G,b)$ which shows that the remaining maximal blocks satisfy the inductive condition.
	


	
	We are therefore left to consider non-maximal unipotent blocks. In this case, Enguehard \cite{Enguehard} has provided an explicit description of characters in these blocks. Through an explicit case-by-case analysis of these blocks we are then able to verify the iAM-condition for these blocks, thereby completing the proof of Theorem \ref{AM good} and as a consequence the proof of Theorem~\ref{AM}.
	
%

\section*{Structure of the paper}

The organization of this paper is as follows.
In Section \ref{section 1} we recall the Jordan decomposition for groups with connected center. In Section \ref{section 2} we then introduce Kessar--Malle's parametrization of $2$-blocks. Using their parametrization we are then  able to provide a Jordan decomposition for quasi-isolated $2$-blocks of groups of exceptional type. In Section \ref{section 3} we use results of Malle--Späth and determine the Harish-Chandra series of height zero characters of maximal blocks.

The next sections are then devoted to computing the local structure of quasi-isolated blocks. In Section \ref{section 4} we show that many Sylow $2$-subgroups of group of Lie type have the so-called Cabanes property. In Sections \ref{section 5} to \ref{section 7} this is used to determine the structure of the defect groups of the quasi-isolated blocks. In Section \ref{section 8} we analyze the action of automorphisms on quasi-isolated blocks and deduce that the Jordan decomposition in Section \ref{section 2} can be chosen to be automorphism equivariant. 

 After having developed these structural results about quasi-isolated blocks we start discussing the inductive AM-condition. In Section \ref{section 9} we recall a criterion for checking this condition and some simplifications which arise in our situation. This simplified criterion will allow us later in Section \ref{section 13} to check the inductive condition for blocks with cyclic stabilizer in the outer automorphism group. In Sections \ref{section 10} and \ref{section 11} we then deal with the maximal blocks with non-cyclic stabilizer. This will in particular involve explicit calculations with the semisimple elements $s$ of $E_6(q)$ whose centralizer is of type $A_2(q)^3.3$. Finally we will deal in Section \ref{section 13} with the remaining non-maximal (unipotent) blocks and then prove our main theorem in Section \ref{section 14}.

\section*{Acknowledgement}

This paper is a contribution to SFB TRR 195. We thank Gunter Malle for helpful discussions and advice. In particular, we thank him for his careful reading of an earlier version of this paper.

	\section{Jordan decomposition for groups with connected center}\label{section 1}

\subsection{Representation theory}
For the representation theory of finite groups we mainly follow the notation in \cite[Section 2.1]{KessarMalle}. Let us recall the most important notation from there. For this let $G$ be a finite group, $\ell$ a prime and $(K, \mathcal{O}, k)$ an $\ell$-modular system large enough for $G$. We denote by $\Irr(G)$ the set of $K$-valued irreducible characters of $G$. By an $\ell$-block of $G$ we will mean a primitive idempotent of $\Z(kG)$ or $\Z(\mathcal{O} G)$. This defines a partition $$\Irr(G) = \displaystyle\bigsqcup_{b} \Irr(G,b)$$ of the irreducible characters into their $\ell$-blocks. For $\chi \in \Irr(G,b)$ let $\mathrm{bl}(\chi):=b_G(\chi):=b$.

A Brauer pair of $G$ with respect to $\ell$ is a pair $(Q, c)$ with $Q$ an $\ell$-subgroup of $G$ and $c$ an $\ell$-block of $\C_G(Q)$. The set of Brauer pairs has a structure of a $G$-poset by the order relation "$\leq$". A
Brauer pair $(Q, c)$ is a $b$-Brauer pair if and only if $\mathrm{br}_Q(b)c = c$,
where $$\mathrm{br}_Q : (kG)^Q \to k {\C_G(Q)}$$ denotes the Brauer homomorphism.

\subsection{Groups of Lie type}
Let $\G$ be a connected reductive linear algebraic group with Frobenius endomorphism $F: \G \to \G$ defined over $\overline{\mathbb{F}}_p$ for some prime $p \neq \ell$.  Given such a group $\G$ it is often convenient to consider it as a closed normal subgroup of a group whose center is connected. Therefore, we fix a regular embedding $\iota: \G \hookrightarrow \Gtilde$ of $\G$ as in \cite[Section 15.1]{MarcBook} and we identify $\G$ with its image in $\Gtilde$. For any closed subgroup $\mathbf{M}$ of $\G$ we define $\tilde{\mathbf{M}}:=\mathbf{M} \mathrm{Z}(\tilde \G)$. Moreover, if $\mathbf{H}$ is any closed $F$-stable subgroup of $\Gtilde$ then we denote by $H:=\mathbf{H}^F$ its subset of $F$-stable points.

In addition, we fix a group $(\G^\ast,F^\ast)$ in duality with $(\G,F)$ as in \cite[Definition 13.10]{DM} and we denote the Frobenius endomorphism $F^\ast$ with the same letter $F$ if no confusion can arise. We also extend our notation for fixed points to the dual group, that is we write $H:= \mathbf{H}^{F^\ast}$ for any $F^\ast$-stable subgroup $\bH$ of $\G^\ast$. The duality betwen $(\G,F)$ and $(\G^\ast,F)$ extends to a duality between the $G$-conjugacy classes of $F$-stable Levi subgroups of $\G$ and the $G^\ast$-conjugacy classes of $F$-stable Levi subgroups of $\G^\ast$. Hence, if $\Levi$ is an $F$-stable Levi subgroup of $\G$, then $\Levi^\ast$ denotes an $F$-stable Levi subgroup of $\G^\ast$ in duality with it.

\subsection{Character theory}

We now recall the character theory of the finite groups $G=\G^F$. By a fundamental result of Lusztig, the set of irreducible characters $\Irr(G)$ is a disjoint union of rational Lusztig series $\mathcal{E}(G,s)$, where $s$ runs over semisimple elements of $G^\ast$ up to conjugation. If $\Levi$ is an $F$-stable subgroup of $\G$ of a (non-necessarily $F$-stable) parabolic subgroup $\Para$ of $\G$ then we have a map
$$R_{\Levi \subset \Para}^{\G}: \mathbb{Z} \Irr(\Levi^F) \to \mathbb{Z} \Irr(\G^F)$$
called Lusztig induction. This map possibly depends on the choice of the parabolic subgroup used in its definition but in our case it will not, see the remarks following \cite[Theorem 2.8]{KessarMalle}. We will therefore write $R_{L}^G$ instead of $R_{\Levi \subset \Para}^{\G}$ in the following. An important consequence of the definition of Lusztig series (see \cite[Theorem 2.8(b)]{KessarMalle}) is that for any semsimple element $s \in L^\ast$ the map $R_L^G$ restricts to a map
$$R_L^G: \mathbb{Z} \mathcal{E}(L,s) \to \mathbb{Z} \mathcal{E}(G,s).$$

\subsection{Jordan decomposition}

We recall the following uniqueness statement about Jordan decomposition for groups with connected center by Digne--Michel. The explicit form of this theorem is taken from \cite[Theorem 4.7.1]{GeckMalle}.

\begin{theorem}\label{DigneMichel}
 There exists a unique collection of bijections
$J_{G,s}: \mathcal{E}(G, s) \to \mathcal{E}(\C_{G^\ast}(s),1)$
where $\G$ runs over connected reductive groups with connected center and with Frobenius
map $F$, and $s \in (\G^\ast)^F$ is semisimple, satisfying the following, where we write $\mathbf{H} :=
\C_{\G^\ast}(s)$:
\begin{enumerate}

\item For any $F^\ast$-stable maximal torus $\T^\ast \leq \mathbf{H}$,
$$\langle
R^\G_{\T^\ast}(s), \rho \rangle
= \varepsilon_G \varepsilon_H
\langle R^{\bH}_{\T^\ast}(1_{\T^\ast} ), J_{G,s}
(\rho)
\rangle
\text{ for all } \rho \in \mathcal{E} (G, s).$$
\item If $s = 1$ and $\rho \in \mathcal{E} (\G^F, 1)$ is unipotent then
\begin{enumerate}

\item  the Frobenius eigenvalues $\omega_\rho$ and $\omega_{J_{G,1}
}(\rho)$ are equal, and
\item if $\rho$ lies in the principal series then $\rho$ and $J_{G,1}
(\rho)$ correspond to the same
character of the Iwahori-Hecke algebra.
\end{enumerate}
\item If $z \in \Z(G^\ast)$ then $J_{G,{sz}}
(\rho \otimes \hat{z}) = J_{G,{s}}
(\rho)$ for $\rho \in \mathcal{E}(G, s)$
\item For any $F$-stable Levi subgroup $\Levi^\ast$ of $\G^\ast$ such that $\mathbf{H} \leq \Levi^\ast$, with dual $\Levi \leq \G$,
the following diagram commutes:
\begin{center}
	\begin{tikzpicture}
		\matrix (m) [matrix of math nodes,row sep=3em,column sep=4em,minimum width=2em] {
			
			\mathcal{E}(G,s)  & \mathcal{E}({H},1)  \\
			\mathcal{E}(L,s) & \mathcal{E}(H,1)
			\\};
		\path[-stealth]
		(m-2-2) edge node [right] {$\mathrm{id}$} (m-1-2)
		(m-1-1) edge node [above] {$J_{G,s}$} (m-1-2)
		(m-2-1) edge node [above] {$J_{L,s}$} (m-2-2)
		(m-2-1) edge node [left] {$R_{{\Levi}}^{{\G}}$} (m-1-1);
		
	\end{tikzpicture}
\end{center}
\item If $\G$ is of type $E_8$ and $\mathbf{H}$ is of type $E_7 A_1$ (resp. $E_6 A_2$) and $\Levi \leq \G$ is a Levi
subgroup of type $E_7$ (resp. $E_6$) with dual $\Levi^\ast \leq H$ then the following diagram
commutes:
\begin{center}
	\begin{tikzpicture}
		\matrix (m) [matrix of math nodes,row sep=3em,column sep=4em,minimum width=2em] {
			
			\mathbb{Z}\mathcal{E}(G,s)  &\mathbb{Z} \mathcal{E}({H},1)  \\
		\mathbb{Z}	\mathcal{E}(L,s)_c &\mathbb{Z} \mathcal{E}(L^\ast,1)_c
			\\};
		\path[-stealth]
		(m-2-2) edge node [right] {$R_{\Levi^\ast}^{\bH}$} (m-1-2)
		(m-1-1) edge node [above] {$J_{G,s}$} (m-1-2)
		(m-2-1) edge node [above] {$J_{L,s}$} (m-2-2)
		(m-2-1) edge node [left] {$R_{{\Levi}}^{{\G}}$} (m-1-1);
		
	\end{tikzpicture}
\end{center}
where the index $c$ denotes the subspace spanned by the cuspidal part of the
corresponding Lusztig series.

\item For any $F$-stable central torus $\T_1 \leq \Z(\G)$ with corresponding natural epimorphism
$\phi : \G \to \G_1 := \G/\T_1$ and for $s_1 \in \G^\ast_1$
with $s = \phi^\ast(s_1)$ the following
diagram commutes:
\begin{center}
	\begin{tikzpicture}
		\matrix (m) [matrix of math nodes,row sep=3em,column sep=4em,minimum width=2em] {
			
			\mathcal{E}(G,s)  &\mathcal{E}({H},1)  \\
		\mathcal{E}(G_1,s_1) & \mathcal{E}(H_1,1)
			\\};
		\path[-stealth]
		(m-1-2) edge node [right] {$ $} (m-2-2)
		(m-1-1) edge node [above] {$J_{G,s}$} (m-1-2)
		(m-2-1) edge node [above] {$J_{G_1,s_1}$} (m-2-2)
		(m-2-1) edge node [left] {$ $} (m-1-1);
		
	\end{tikzpicture}
\end{center}
where the vertical maps are just the inflation map along
$\G^F \to \G_1^F$
and the restriction along the embedding $\mathbf{H}^F_1 
\to \mathbf{H}^F$ respectively.
\item If $\G$ is a direct product
$\G=\prod_i \G_i$ of $F$-stable subgroups $\G_i$ then $J_{G,{\prod_i s_i}}
=
\prod_i
J_{G_i,s_i}$.
\end{enumerate}
\end{theorem}

In the following, we denote by $\mathcal{Z}_F$ the subset of $F$-covariants of the quotient group $\mathcal{Z}:=\Z(\G)/\Z(\G)^\circ$. Let $\mathcal{L}: \G \to \G, g \mapsto g^{-1} F(g),$ denote the Lang map. Then $\mathcal{Z}_F$ acts on $\Irr(\G^F)$ via the isomorphism $\mathcal{L}^{-1}(\Z(\G))/ \Z(\bG) \G^F \cong \mathcal{Z}_F$ induced by the Lang map.

Moreover, for a semisimple element $s \in (\G^\ast)^F$ we denote by
$A_{\G^\ast}(s)$ (or $A(s)$ if the ambient group is clear from the context) the component group of its centralizer $\C_{\G^\ast}(s)$.

Let $\sigma: \G \to \G$ be a bijective morphism commuting with the action of $F$. Then there exists a dual morphism $\sigma^\ast: \G^\ast \to \G^\ast$ which commutes with $F^\ast$ as in \cite[Definition 2.1]{Spaeth}. This dual morphism is unique up to automorphisms induced by the action of elements of $\mathcal{L}_{F^\ast}^{-1}(\Z(\G^\ast))$.

\begin{theorem}\label{equivariant sp}
Let $\G$ be a connected reductive group with connected center and $\sigma: \G \to \G$ a bijective morphism commuting with $F$. Then for $s \in (\G^\ast)^F$ the bijection $J_s: \mathcal{E}(G,s )  \to \mathcal{E}(\C_{\G^\ast}(s),1)$ from Theorem \ref{DigneMichel} satisfies $J_s(\lambda)^\sigma=J_{\sigma^\ast(s)}(\sigma^\ast(\lambda))$ for all $\lambda \in \mathcal{E}(G,s)$.  
\end{theorem}

\begin{proof}
This is \cite[Theorem 3.1]{Spaeth}.
\end{proof}

\begin{example}\label{trivial action}
Keep the notation and assumptions of Theorem \ref{equivariant sp}. Assume that $\sigma$ stabilizes all rational components of $\bH:=\C_{\G^\ast}(s)$. The action of automorphisms on $\mathcal{E}(H,1)$ can be read off from \cite[Theorem 2.5]{MalleUnip}. In particular, if $\bH$ has no rational component with the same type as the groups appearing in \cite[Theorem 2.5]{MalleUnip}, then $\sigma^\ast$ acts trivially on $\mathcal{E}(H,1)$ and hence $\sigma$ acts trivally on $\mathcal{E}(G,s)$ by Theorem \ref{equivariant sp}.
\end{example}

The following theorem is probably well-known but since we will need parts of its proof later we decided to include its proof as well.

\begin{theorem}\label{equivariant jordan}
Let $\G$ be a connected reductive group with Frobenius endomorphism $F$ and assume that $\mathcal{Z}_F$ is cyclic. Then for every semisimple element $s \in (\G^\ast)^F$ the following holds.
\begin{enumerate}[label=(\alph*)]
\item There exists a bijection $\overline{J}_s:\mathcal{E}(G,s)/\mathcal{Z}_F \to \mathcal{E}(\C^\circ_{\G^\ast}(s),1)/A_{G^\ast}(s)$
on orbits which are of the same size for corresponding characters.

\item Let $\sigma: \G \to \G$ be a bijective morphism commuting with $F$. Then for every $\lambda \in \mathcal{E}(G,s)/\mathcal{Z}_F$ we have ${}^{\sigma^\ast} \overline{J}_s(\overline{\lambda})=\overline{J}_{\sigma^\ast(s)}(\overline{\lambda}^\sigma)$.
\item  There exists a bijection $J_{s}:\mathcal{E}(G,s)  \to \mathcal{E}(\C_{G^\ast}(s),1)$ lifting the bijection $\overline{J}_s$ which for $\chi \in \mathcal{E}(G,s)$ satisfies
	$$\chi(1)=|G^\ast:\C_{G^\ast}(s)|_{p'} J_s(\chi)(1).$$
\end{enumerate}
\end{theorem}

\begin{proof}
  When the center $\Z(\G)$ is connected, it follows that $\C_{\G^\ast}(s)$ is connected as well and the existence of such a bijection follows from \cite[Corollary 2.6.6]{GeckMalle}.

In the general case one first considers a regular embedding $\G \hookrightarrow \Gtilde$ and by Theorem \ref{DigneMichel} one obtains a bijection $\mathcal{E}(\tilde{G},\tilde{s}) \to \mathcal{E}(\C_{\Gtilde^\ast}(\tilde{s}),1)$. The dual map $\iota^\ast: \Gtilde^\ast \to \G^\ast$ induces a surjective map $\C_{\tilde{\G}}(\tilde{s}) \to \C^\circ_{\G}(s)$ with kernel $\Z(\tilde{\G}^\ast)$, see \cite[Equation (2.2)]{Bonnafe} and therefore a bijection $\mathcal{E}(\C_{\tilde{G}}(\tilde{s}),1) \to \mathcal{E}(\C_{G}(s),1)$ between unipotent characters. By Clifford theory this leads to a bijection
$\mathcal{E}(G,s)/\mathcal{Z}_F \to \mathcal{E}(\C^\circ_{\G^\ast}(s),1)/A_{G^\ast}(s)$
on orbits which are of the same size for corresponding characters which shows part (a). Part (b) is now a consequence of Theorem \ref{equivariant sp}.

Lifting the bijection between orbits from part (b) (see also \cite[Problem 15.3]{MarcBook}) yields a bijection
$J_s:\mathcal{E}(G,s) \to \mathcal{E}(\C_{G^\ast}(s),1)$.
By Clifford theory, the degrees of these characters can be solely read off from the orbit lengths.
Indeed, let $\tilde{\chi} \in \mathcal{E}(\tilde{G},\tilde{s})$ and $\chi \in \mathcal{E}(G,s)$ below $\tilde{\chi}$ with Jordan correspondents $\tilde{\psi}$ and $\psi$ respectively, so that $\tilde{\chi}(1)=|\tilde{G}:\C_{\tilde{G}}(\tilde{s})|_{p'} \tilde{\psi}(1)$. Let $\mathcal{O}$ be the $A(s)$-orbit of the character $\tilde{\psi} \in \mathcal{E}(\C_{\tilde{G}^\ast}(s),1)$. Then $\chi(1)=\tilde{\chi}(1)/|\mathcal{O}|$ and $\psi(1)=\tilde{\psi}(1) |A(s)|/|\mathcal{O}|$. Moreover, we have $|\tilde{G}/G| = |\Z(\tilde{G})|= |\C_{\tilde{G}}(\tilde{s})|/|\C_{G}^\circ(s)|$. Putting these equations together yields $\chi(1)=|G:\C_{G^\ast}(s)|_{p'} \psi(1)$ as desired. 
\end{proof}
	
\begin{remark}\label{orbit}
	Keep the notation of Theorem \ref{equivariant jordan}.
Assume that $\iota:\G \hookrightarrow \hat{\G}$ is a closed immersion with $\iota([\G,\G])=[\hat{\G},\hat{\G}]$ and identify $\G$ with its image $\iota(\G)$. The map $\iota$ induces by duality a map $\iota^\ast: \hat{\G}^\ast \to \hat{\G}$ and we assume that $\C_{\hat{\G}^\ast}(\hat{s})^F=\C^\circ_{\hat{\G}^\ast}(\hat{s})^{F}$ for some $\hat{s} \in (\hat{\G}^\ast)^F$ with $\iota^\ast(\hat{s})=s$. A special case of this is when $\hat{\G}$ has connected center, so that the centralizer of any semisimple element in $\hat{\G}^\ast$ is connected, see \cite[Section 15]{MarcBook}. Nevertheless, the constructions in \cite[Section 15]{MarcBook} also work in our slightly more general situation and we obtain that the $(\Z(\G)/\Z(\G)^\circ)_F$-orbit of a character in $\mathcal{E}(G,s)$ is the same as its $\hat{\G}^F$-orbit.
\end{remark}

Let $\bH$ be a (not necessarily connected) reductive group and $\T$ a maximal torus of $\bH^\circ$. Then we denote $W_\bH(\T):=W(\bH,\T):=\N_\bH(\T)/\T$ and we let $p_\T:\N_{\bH}(\T) \to W_\bH(\T)$ be the natural surjection. In addition, if $\mathbf{M}$ is a closed subgroup of $\bH$ containing a maximal
torus $\T$ of $\bH$, then $W(\bH \mid \mathbf{M}, \T)$ is as in \cite[Notation 1.8]{Marc}. Namely, $W(\bH \mid \mathbf{M}, \T)$ is the subgroup of $W(\bH, \T)$ generated
by involutions associated with the roots of $\bH$ relative to $\T$ that are
orthogonal to all the roots of $\mathbf{M}$.
	
Let $\G$ be a connected reductive group. For a semisimple element $s \in \G^\ast$ whose $\G^\ast$-conjugacy class is $F$-stable we denote by $\tilde{\mathcal{E}}(\G,s)$ its geometric Lusztig series, see \cite[Definition 13.16]{DM}. Moreover, if the $F$-stable maximal tours $\T$ of $\G$ is in duality with the maximal $F$-stable torus $\T^\ast$ of $\G^\ast$, then we denote by ${}^\ast: W(\G,\T) \to W(\G^\ast,\T^\ast)$ the anti-isomorphism induced by duality, see the remarks after \cite[Definition 13.10]{DM}. The following proposition will be used frequently all the time.
	
	\begin{proposition}\label{CE}\label{Weyl}
	Let $\T \subset \Levi \subset \G$ and $\T^\ast \subset \Levi^\ast \subset \G^\ast$ be in duality
	with $\Levi$ an $F$-stable Levi subgroup of $\G$ and $\T$ an $F$-stable maximal torus of $\Levi$. Let $s \in  \T^\ast$ such that its $\Levi^\ast$-conjugacy class is $F$-stable.
		\begin{enumerate}[label=(\alph*)]
			\item There is a natural isomorphism 
			$$\N_{\G^F}(\Levi)/\Levi^F \to (\N_{W(\G,\T)}(W(\Levi,\T))/W(\Levi,\T))^F,\quad x \Levi^F \mapsto n(x) W(\Levi,\T),$$ for $x \in \N_{\G^F}(\Levi)$ with ${}^x \T = {}^l \T$ for $l \in \Levi$ and $n(x):=l^{-1}x$.
			\item For the geometric Lusztig series, we have ${}^x \tilde{\mathcal{E}}(\Levi,s)=\tilde{\mathcal{E}}(\Levi,{}^{n(x)^\ast} s)$.
			\item The map from part (a) induces an isomorphism from $\N_{\G^F}(\Levi, \tilde{\mathcal{E}}(\Levi,s))/\Levi^F$ onto a subgroup of $(\N_{W(\C_{\G^\ast}(s), \T^\ast)}(W(\C^\circ_{\Levi^\ast}(s), \T^\ast))/W(\C_{\Levi^\ast}(s), \T^\ast))^F$.
			\item If $x \in \N_{\G^F}(\Levi) \cap  p_{\T^{-1}}(W(C_{
			\G^\ast}^\circ(s) \mid \C^\circ_{\Levi^\ast}(s), \T^\ast))$ and $s$ have coprime orders, then $x$ fixes every element of $\tilde{\mathcal{E}}(\Levi^F, s)$.
		\end{enumerate}
	\end{proposition}
	
	\begin{proof}
		See \cite[Proposition 1.9]{Marc}.
	\end{proof}

\begin{remark}\label{Sylow Weyl}
	Let $\lambda \in \tilde{\mathcal{E}}(\Levi^F,s)$ where $s \in (\Levi^\ast)^{F^\ast}$ is of $\ell'$-order. Then by Proposition \ref{Weyl}(c) the group $W_{\G^F}(\Levi,\lambda):=\N_{\G^F}(\Levi,\lambda)/\Levi^F$ is isomorphic to a subgroup of $$(\N_{W(\C_{\G^\ast}(s), \T^\ast)}(W(\C^\circ_{\Levi^\ast}(s), \T^\ast))/W(\C_{\Levi^\ast}(s), \T^\ast))^F.$$ Moreover, according to part (d) a Sylow $\ell$-subgroup of $W_{\G^F}(\Levi,\lambda)$ is isomorphic to a Sylow $\ell$-subgroup of $(\N_{W(\C_{\G^\ast}(s), \T^\ast)}(W(\C^\circ_{\Levi^\ast}(s), \T^\ast))/W(\C_{\Levi^\ast}(s), \T^\ast))^F$.
	
	For instance if $\C_{\Levi^\ast}^\circ(s)$ is a maximal torus of $\C^\circ_{\G^\ast}(s)$ then a Sylow $\ell$-subgroup of $W_{\G^F}(\Levi,\lambda)$ is isomorphic to a Sylow $\ell$-subgroup of $W(\C_{\G^\ast}(s),\C^\circ_{\Levi^\ast}(s))^{F^\ast}$.
\end{remark}

\subsection{$d$-tori and minimal Levi subgroups}
An $F$-stable torus $\T$ of $\G$ is called a $d$-torus if there exists $a \geq 0$ such that $|\T^{F^k}| = \Phi_d(q^k)^a$
for all $k \geq 1$ with $(k,d)=1$, where $\Phi_d$ denotes the $d$th cyclotomic polynomial. We denote by $\T_{\Phi_d}$ the maximal $d$-subtorus of $\T$. Centralizers $\Levi$ of $d$-tori
of $\G$ are called $d$-split Levi subgroups. By construction, these satisfy $\Levi=\C_{\G}(\Z^\circ(\Levi)_{\Phi_d})$, see \cite[Proposition 3.5.5]{GeckMalle}.

\begin{lemma}\label{minimal split}
Let $s \in (\G^\ast)^F$ be a semisimple element and $d$ some positive integer. Then for a Levi subgroup $\Levi^\ast$ of $\G^\ast$ the following are equivalent:\begin{enumerate}[label=(\roman*)]
	\item $\Levi^\ast=\mathrm{C}_{\G^\ast}(\T_0^\ast)$, where $\T_0^\ast$ is a Sylow $d$-torus of $\C^\circ_{\G^\ast}(s)$.
	\item $\Levi^\ast$ is a minimal $d$-split Levi subgroup of $\G^\ast$ which contains $s$.
\end{enumerate}
In this case, the Levi subgroup $\Levi^\ast$ is unique up to $\C_{\G^\ast}(s)^F$-conjugation and we call it the $d$-split Levi subgroup associated to $s$.
\end{lemma}

\begin{proof}
 Indeed, if $\Levi^\ast$ is a Levi subgroup satisfying (ii), then $\Z^\circ(\Levi^\ast)_{\Phi_d} \subset \C^\circ_{\G^\ast}(s)$ and therefore by \cite[Theorem 25.11]{MT} we have $\Z^\circ(\Levi^\ast)_{\Phi_d} \leq \T_0^\ast$, up to $(\C^\circ_{\G^\ast}(s))^{F^\ast}$-conjugation. Since $\Levi^\ast=\C_{\G^\ast}(\Z^\circ(\Levi^\ast)_{\Phi_d})$ we find $\Levi^\ast=\C_{\G^\ast}(\T_0^\ast)$ by the minimality of $\Levi^\ast$ and also $\Z^\circ(\Levi^\ast)_{\Phi_d} = \T_0^\ast$. The converse direction is proved in the same fashion.
 
 The uniqueness statement follows from the fact that any two Sylow $d$-tori of $\C_{\G^\ast}^\circ(s)$ are $\C^\circ_{\G^\ast}(s)^F$-conjugate, see \cite[Theorem 25.11]{MT}.
\end{proof}

\begin{remark}\label{torus}
For $d \in \{1,2\}$	let $\Levi^\ast$ be the minimal $d$-split Levi subgroup containing $s$. A consequence of the characterization of the previous lemma and \cite[Lemma 3.17]{Cabanesgroup} is that $\C^\circ_{\Levi^\ast}(s)$ is always a maximal torus of $\C^\circ_{\G^\ast}(s)$.
\end{remark}

\subsection{Blocks of groups of Lie type}
	
	We let $\ell \neq p$ be a fixed prime and assume that $s \in G^\ast$ is a semisimple element of $\ell'$-order. Denote by $\C_{G^\ast}(s)_\ell$ the subset of $\ell$-power order elements of the centralizer of $s$.	By a fundamental result of Broué--Michel \cite[Theorem 9.12]{MarcBook}, the set
	$$\mathcal{E}_\ell(G,s)=\displaystyle\bigcup_{t \in \C_{G^\ast}(s)_\ell} \mathcal{E}(G,t)$$
	is a union of $\ell$-blocks. We denote by $e_s^{\G^F}$ the central idempotent in $\Z(\mathcal{O} \G^F)$ or $\Z(k \G^F)$ associated to this union of blocks.
	The following lemma can be seen as a variant of Proposition \ref{Weyl}(c).

\begin{lemma}\label{stabilizer}
	Let $s \in (\G^\ast)^F$ be a semisimple element of $\ell'$-order. For some integer $d \geq 1$ we let $\Levi^\ast$ be a minimal $d$-split Levi subgroup of $\G^\ast$ containing $s$ . If $\T^\ast$ is a maximally split torus of $\C^\circ_{\Levi^\ast}(s)$, then
	$$\N_{\G^F}(\Levi,e_s^{\Levi^F})/ \Levi^F \cong  W(\C_{\G^\ast}(s), \T^\ast)^F/W(\C_{\Levi^\ast}(s), \T^\ast)^F.$$
\end{lemma}

\begin{proof}	
	Duality between the Levi subgroups $\Levi$ and $\Levi^\ast$ yields an anti-isomorphism $\N_{\G^F}(\Levi)/\Levi^F \cong \N_{(\G^\ast)^F}(\Levi^\ast)/ (\Levi^\ast)^F$ under which the stabilizer of $e_s^{\Levi^F}$ corresponds to the stabilizer of the $(\Levi^\ast)^{F}$-conjugacy class of $s$. Let $n \Levi^F \in \N_{\G^F}(\Levi)/\Levi^F$ and $n^\ast (\Levi^\ast)^{F}$ the coset of $\N_{(\G^\ast)^F}(\Levi^\ast)/ (\Levi^\ast)^F$ corresponding to it. By multiplying $n^\ast$ by an element of $(\Levi^\ast)^F$ we can assume that $n^\ast$ stabilizes $s$. In particular, $n^\ast$ normalizes $\C_{\Levi^\ast}(s)$. Thus, there exists some $l \in \C^\circ_{\Levi^\ast}(s)^{F}$ such that $n^\ast l$ stabilizes $\T^\ast$. Hence, without loss of generality $n^\ast \in W(\C_{\G^\ast}(s),\T^\ast)^F$. This induces a well-defined map
	$$\N_{\G^F}(\Levi,e_s^{\Levi^F})/ \Levi^F \to  W(\C_{\G^\ast}(s), \T^\ast)^F/W(\C_{\Levi^\ast}(s), \T^\ast)^F, \quad $$
	by sending the coset of $n$ to the coset of $n^\ast$.

	Assume conversely that $w \in W(\C_{\G^\ast}(s),\T^\ast)^F$. By assumption $\Levi^\ast=\C_{\G^\ast}(\Z(\Levi^\ast)^\circ_{\Phi_d})$.
	We claim that $\Z(\Levi^\ast)^\circ_{\Phi_d}=\T^\ast_{\Phi_d}$. Indeed, $\Z(\Levi^\ast)^\circ$ is a central subgroup of $\C^\circ_{\Levi^\ast}(s)$ and hence contained in the maximal torus $\T^\ast$. Hence, $\Z(\Levi^\ast)^\circ_{\Phi_d} \leq \T^\ast_{\Phi_d}$. On the other hand, $\C_{\G^\ast}(\T^\ast_{\Phi_d}) \leq \Levi^\ast$ is a $d$-split Levi subgroup of $\G^\ast$ containing $s \in \T^\ast$. This implies  $\Levi^\ast=\C_{\G^\ast}(\T^\ast_{\Phi_d})$ and so $\T^\ast_{\Phi_d} = \Z(\Levi^\ast)^\circ_{\Phi_d}$. In particular, $w$ stabilizes the Levi subgroup $\Levi^\ast$ and the semisimple element $s$.
\end{proof}

We also mention the following fact for later:

\begin{lemma}\label{Brauer morphism}
	Let $\G$ be a connected reductive group, $s \in (\G^\ast)^{F^\ast}$ a semisimple element of $\ell'$-order and suppose that $\Levi:=\C_{\G}(Q)$ is a Levi subgroup for some $\ell$-subgroup $Q$ of $\G^F$. Then $\mathrm{br}_Q(e_s^{\G^F})=\sum_{t} e_{t}^{\Levi^F}$, where $t$ runs over a set of representatives of the $(\Levi^\ast)^{F^\ast}$-conjugacy classes of semisimple $\ell'$-elements of $(\Levi^\ast)^F$ which are $(\G^\ast)^{F}$-conjugate to $s$.
\end{lemma}

\begin{proof}
	This follows from the explicit description of the Brauer morphism in \cite[Theorem 4.14]{Dat}.
\end{proof}

\begin{corollary}
Under the assumptions of Lemma \ref{Brauer morphism} assume additionally that $\Levi^\ast$ is the minimal $d$-split Levi subgroup containing $s$ for some integer $d$. Then $\mathrm{br}_Q(e_s^{\G^F})=\sum_{t} e_{t}^{\Levi^F}$, where $t$ runs over a set of representatives of the $(\Levi^\ast)^{F}$-conjugacy classes of semisimple $\ell'$-elements of $(\Levi^\ast)^F$ which are $\N_{(\G^\ast)^{F}}(\Levi^\ast)$-conjugate to $s$.
\end{corollary}

\begin{proof}
Assume that $t$ is $(\G^\ast)^F$-conjugate to $s$ so that $s={}^g t$ for some $g \in (\G^\ast)^F$. Hence, $\Levi^\ast$ and ${}^g \Levi^\ast$ are both minimal $d$-split Levi subgroups that contain $s$ and so $\Levi^\ast={}^{xg} \Levi^\ast$ for some $x \in \C^\circ_{\G^\ast}(s)^F$. Hence, $s={}^{xg} t$ with $xg \in \N_{(\G^\ast)^{F}}(\Levi^\ast)$.
\end{proof}

\subsection{Characters of quasi-central defect}

We denote by $\mathcal{E}(G,\ell')$ the union of all Lusztig series $\mathcal{E}(G,s)$ with $s$ an element of $\ell'$-order.
To introduce the parametrization of $\ell$-blocks by Kessar--Malle, we need to recall the following definition:

\begin{definition}\label{central}
Let $\lambda \in \mathcal{E}(G
	,\ell')$. We say that $\lambda$ is of central $\ell$-defect if 
	$$|G|_\ell =
	\lambda(1)_\ell
	|\Z(G)|_\ell$$ and that $\lambda$ is of quasi-central $\ell$-defect if some (and hence any) character of
	$[\G, \G]^F$
	covered by $\lambda$ is of central $\ell$-defect.
\end{definition}

The blocks $b_{G}(\lambda)$ with $\lambda$ of central defect are quite easy to describe. By \cite[Proposition 2.5]{KessarMalle}, the block $b_G(\lambda)$ is a nilpotent block of central defect and hence by \cite[Theorem 9.12]{NavarroBook} we have a bijection $\Irr(\Z(G)_\ell) \to \Irr(G,b_G(\lambda))$. Moreover, by \cite[Proposition 2.5]{KessarMalle}, $G/[G,G] \Z(G)$ is an $\ell'$-group. Since $G/[G,G] \cong \Z^\circ(G)$ it follows that $$\Irr(G,b_G(\lambda))=\{\lambda \mu \mid \mu \in \Irr(G/[G,G])_\ell \},$$
whenever $\ell \nmid |\Z(G):\Z^\circ(G)|$.

	\section{Kessar--Malle's parametrization of $\ell$-blocks}\label{section 2}

\subsection{Kessar--Malle's parametrization of $\ell$-blocks}

From now on, unless stated explicitly otherwise, $\ell=2$ and $\G$ will always denote a simple, simply connected algebraic group of exceptional type with Frobenius endomorphism $F: \G \to \G$ defining an $\mathbb{F}_q$-structure on $\G$, where $q$ is an integral power of an odd prime $p$. Recall that a semisimple element $s \in \G^\ast$ is called quasi-isolated if its centraliser $\mathrm{C}_{\G^\ast}(s)$ is not contained in any proper Levi subgroup of $\G^\ast$. We assume now that $1 \neq s$ is such a quasi-isolated element of odd order in $G^\ast$ and we let $b$ be a $2$-block associated to $s$. Let $e$ denote the order of $q$ modulo $4$. According to \cite[Theorem 1.2]{KessarMalle} there exists an $e$-cuspidal pair $(\Levi,\lambda)$, unique up to $\G^F$-conjugation with $\lambda \in \mathcal{E}(L,s)$ of quasi-central defect such that all irreducible constituents of $R_\Levi^\G(\lambda)$ lie in $\mathcal{E}(G,s) \cap \Irr(b)$. We then write $b=b_{\G^F}(\Levi,\lambda)$.

We say that $b$ is a \textit{maximal block} if its defect group has order $|\C_{G^\ast}(s)|_\ell$. For each $s$ there is at least one maximal block associated to it, see \cite[Lemma 2.6(b)]{KessarMalle}. More precisely we can characterize maximal blocks as follows:

%

\begin{remark}\label{intrinsic}
	Assume that $b=b_{\G^F}(\Levi,\lambda)$ is a maximal block. Then from the tables in \cite{KessarMalle} it can be deduced that the Levi subgroup $\Levi^\ast$ is a minimal $e$-split Levi subgroup associated to the semisimple element $s$ with $\lambda \in \mathcal{E}(L,s)$. Moreover, the character $\lambda$ corresponds under Jordan decomposition to one of $|A_{L^\ast}(s)|$ many regular-semisimple characters of $\mathcal{E}(\C_{L^\ast}(s),1)$.
\end{remark}

\subsection{Jordan decomposition for blocks}\label{dual group}

	We follow \cite[3.5.1(a)]{EnguehardJordandecomp} to construct a group $\G(s)$ which is "dual" to  the not necessarily connected group $\C_{\G^\ast}(s)$. We observe that there exists a surjective morphism $\pi: \G \to \G_{\mathrm{ad}}$ and we may identifty $\G_{\mathrm{ad}}$ with $\G^\ast$ as $\G$ is simply connected of exceptional type. We define $\G(s):=\pi^{-1}(\C_{\G^\ast}(s))$ which is an $F$-stable subgroup of $\G$. However, note that $\G(s)$ is not necessarily the centralizer of an $F$-stable semisimple element of $\G$. We observe that the connected component $\G^\circ(s)$ of $\G(s)$ is a connected reductive group in duality with $\C^\circ_{\G^\ast}(s)$ and the map $\pi: \G(s) \to \C_{\G^\ast}(s)$ induces by \cite[Theorem 13.14]{MarcBook} an $(F,F^\ast)$-equivariant isomorphism $\G(s)/\G^\circ(s)\cong A(s)$. By duality we obtain the following lemma:

\begin{lemma}
Let $s \in G^\ast$ be a semisimple element of odd order and $G(s)=\G(s)^{F}$ the group constructed above. Then we have $$\N_G(\G(s),e_s^{G^\circ(s)})=G(s).$$
\end{lemma}

\begin{proof}
If $\C^\circ_{\G^\ast}(s)$ is a Levi subgroup of $\G^\ast$, then the conclusion of the lemma is contained in \cite[Equation 7.1]{Dat}.

We let $\T^\ast$ be a maximally split torus of $\C^\circ_{\G^\ast}(s)$ and $\T \leq \G^\circ(s)$ a maximal torus in duality with it. Any element $g \in \N_G(\G^\circ(s),e_s^{\G^F})$ will map the torus $\T$ to a $G^\circ(s)$-conjugate. We can therefore assume that $g \in \N_G(\T)$. The image $w$ of $g$ in $W^F=\N_G(\T)/T$ corresponds to an element $w^\ast \in  (W^\ast)^F$ in the dual Weyl group. By Theorem \ref{equivariant jordan}(b) we therefore have ${}^w \mathcal{E}(G^\circ(s),s)=\mathcal{E}(G^\circ(s),{}^{w^\ast} s)$. However, $s$ is $\C_{G^\ast}(s)$-conjugate to ${}^{w^\ast} s$ if and only if ${}^{w^\ast} s=s$.  By construction this means $g \in G(s)$.
\end{proof}
	
	As before, let $\iota:\G \hookrightarrow \Gtilde$ be a regular embedding and fix a semisimple element $\tilde{s} \in (\Gtilde^\ast)^F$ of $2'$-order with $\iota^\ast(\tilde{s})=s$. The group $\Gtilde$ is self-dual and we shall identify it with its dual $\Gtilde^\ast$ compatible with the map $\pi: \G \to \G^\ast$. We let $\Gtilde(\tilde{s})$ be the subgroup of $\Gtilde$ corresponding to $\C_{\Gtilde^\ast}(\tilde{s})$ under this identification. We have $\Gtilde(\tilde{s})=\Z(\Gtilde) \G^\circ(s)$ and so the group $\tilde{G}(\tilde{s})$ induces all diagonal automorphisms of $G$. Finally note that $G(s)$ normalizes $\tilde{G}(\tilde{s})$.
	
	\begin{lemma}\label{jordan construct}\label{central character}
With the notation from above, there exists a $\tilde{G}(\tilde{s})$-equivariant bijection $\mathcal{J}: \mathcal{E}_2(G,s) \to \mathcal{E}_2(G(s),s) $, which preserves the underlying character of $\Z(G)$.
	\end{lemma}
	
	\begin{proof}

Suppose first that $A(s)^F=1$.
By Theorem \ref{DigneMichel}, for $t \in \C_{G^\ast}(s)_2$ one has bijections $J_{G,st}: \mathcal{E}(G,st) \to \mathcal{E}(\C_{G^\ast}(st),1) $ and $J_{G(s),st}:\mathcal{E}(G(s),st) \to \mathcal{E}(\C_{G^\ast}(st),1)$, where we observe that $\C_{\C_{G^\ast}(s)}(t)=\C_{G^\ast}(st)$. This yields a bijection $\mathcal{J}:= J_{G(s),st}^{-1} \circ J_{G,st}:  \mathcal{E}(G,st) \to  \mathcal{E}(G(s),st)$.

We now first construct the bijection in the general case. As before, we obtain a bijection $\tilde{\mathcal{J}}: \mathcal{E}_2(\tilde{G},\tilde{s}) \to \mathcal{E}_2(\tilde{G}(\tilde{s}),\tilde{s})$. By duality, we have a canonical bijective map $\Z(\tilde{G}^\ast) \to \Irr(\tilde{G}/G),z \mapsto \hat{z},$ and since $\tilde{G}/G \cong \tilde{G}(\tilde{s})/G^\circ(s)$ restriction defines a bijection
$$\Irr(\tilde{G}/G) \to \Irr(\tilde{G}(\tilde{s})/G^\circ(s)).$$
Fix a character $\chi \in \mathcal{E}(\tilde{G},\tilde{s} \tilde{t})$ with $\tilde{t} \in \C_{\tilde{G}^\ast}(\tilde{s})_2$ and let $$\psi:=\tilde{\mathcal{J}}(\chi)=J_{\tilde{G}(\tilde{s}),\tilde{s} \tilde{t}}^{-1}(J_{\tilde{G},\tilde{s} \tilde{t}}(\chi)) \in \mathcal{E}(\tilde{G}(\tilde{s}),\tilde{s} \tilde{t}).$$
Observe that $\hat{z} \chi=\chi$ for $z \in \Z(\tilde{G}^\ast)$ implies that $z\tilde{s} \tilde{t}={}^{\tilde{a}} (\tilde{s} \tilde{t})$ for some $\tilde{a} \in \tilde{G}^\ast$ with image $a^\ast=\iota^\ast(\tilde{a}) \in \C_{G^\ast}(st) \subset \C_{G^\ast}(s)$.
 Let $a \in G(s)/G^\circ(s)$ be the element corresponding to $a^\ast$ in $A(s)^{F}$ under the isomorphism $G(s)/G^\circ(s) \cong A(s)^F$. By Theorem \ref{DigneMichel} and Theorem \ref{equivariant sp}, we obtain 
$$J_{\tilde{G}(\tilde{s}),\tilde{s} \tilde{t}} \tilde{z}^{-1}=J_{\tilde{G}(\tilde{s}),\tilde{s} \tilde{t} \tilde{z}}=J_{\tilde{G}(\tilde{s}),{}^{a^\ast}(\tilde{s} \tilde{t})}={}^a J_{\tilde{G}(\tilde{s}),\tilde{s} \tilde{t}}.$$ Hence, $\hat{z} \chi=\chi$ if and only if $\hat{z} {}^a \psi=\psi$. Note that $a$ is unique up to multiplication with an element of $\C_{\tilde{G}^\ast}(st)$. We therefore obtain a bijection
$$\Irr(\tilde{G}/G)_\chi \to ( \Irr(\tilde{G}(\tilde{s})/G^\circ(s)) \times G(s)/G^\circ(s))_{\psi}, \hat{z} \mapsto (\hat{z},a).$$
Observe that $A(s)$ has $2'$-order by \cite[Proposition 13.16(i)]{MarcBook}. Moreover, for $z \in \Z(\tilde{G}^\ast)$ the elements $z_{2'} \tilde{s}$ and $\tilde{s}$ are only  $\C_{\Gtilde^\ast}(\tilde{s})$-conjugate when $z_{2'}=1$. From this we deduce that $\Irr(\tilde{G}(\tilde{s})/G^\circ(s))_{\psi}$ is a $2$-group. We therefore have a bijection
$$(\Irr(\tilde{G}/G)_\chi)_2 \to \Irr(\tilde{G}(\tilde{s})/G^\circ(s))_{\psi}.$$
Hence, by Clifford theory the character $\Res^{\tilde{G}(\tilde{s})}_{G^\circ(s)}(\psi)$ has a $2$-power number of constituents. Denote $\mathcal{A}:=(G(s)/G^\circ(s))_{\Res^{\tilde{G}(\tilde{s})}_{G^\circ(s)}(\psi)}.$
Similar considerations show that we have a bijection
$$(\Irr(\tilde{G}/G)_\chi)_{2'} \to \mathcal{A}.$$
In particular, since $\mathcal{A}$ is a $2'$-group and  $\Res^{\tilde{G}(\tilde{s})}_{G^\circ(s)}(\psi)$ has a $2$-power number of constituents we deduce that there exists an $\mathcal{A}$-stable constituent $\psi_0$. Since $[G(s),\tilde{G}(\tilde{s})] \subset G^\circ(s)$ it follows that every constituent of $\Res^{\tilde{G}(\tilde{s})}_{G^\circ(s)}(\psi)$ is $\mathcal{A}$-stable. 

From this we deduce again by Clifford theory that the number of characters in $\Irr(G(s) \mid \Res^{\tilde{G}(\tilde{s})}_{G^\circ(s)}(\psi))$, i.e. the number of irreducible characters of $G(s)$ which lie over an irreducible constituent of $\Res^{\tilde{G}(\tilde{s})}_{G^\circ(s)}(\psi)$, is equal to the number of characters in $\Irr(G \mid \chi )$. Moreover, by Mackey's formula $$\Irr(G(s) \mid \Res^{\tilde{G}(\tilde{s})}_{G^\circ(s)}(\psi))=\Irr(G(s) \mid \Ind_{\tilde{G}(\tilde{s})}^{\tilde{G}(\tilde{s}) G(s)}(\psi)).$$
In particular, these characters form one $\tilde{G}(\tilde{s})$-orbit. We can therefore define a bijection $\mathcal{J}: \mathcal{E}_2(G,s) \to \mathcal{E}_2(G(s),s)$ by sending the set $\Irr(G \mid \chi)$ in a $\tilde{G}(\tilde{s})$-equivariant way to $\Irr(G(s),\Ind_{\tilde{G}(\tilde{s})}^{\tilde{G}(\tilde{s}) G(s)}(\psi))$. The compatibility of $\mathcal{J}$ with central characters follows from \cite[Equation 9.12]{Bonnafe2}.
	 \end{proof}
	 
%
	 
	 Note that the bijection $\mathcal{J}$ depends on our particular choices of Jordan decompositions involved in its definition. We will later refine the construction to make it suitable for our purposes.
	 
	 Our aim is now to show that the number of unipotent $2$-blocks of $G(s)$ and the number of $2$-blocks of $G$ associated to $s$ coincide.

 Let $N$ be a normal subgroup of a finite group $H$. Then we denote by $B_0(N)$ its principal $\ell$-block and by $\mathrm{Bl}(H,B_0(N))$ the set of blocks of $H$ covering $B_0(N)$.
	
	\begin{lemma}\label{dade}
		Let $N$ be a normal subgroup of a finite group $H$ such that the factor group $H/N$ is cyclic of $\ell'$-order. Then $|\mathrm{Bl}(H,B_0(N))|=|\C_{H}(P)N:N|,$ where $P$ is a Sylow $\ell$-subgroup of $N$. Moreover, any block of $\C_{H}(P)N$ covering $B_0(N)$ is Morita equivalent by restriction to $B_0(N)$.
	\end{lemma}
\begin{proof}
	Denote $b=B_0(N)$. We let $H[b] \lhd H$ be Dade's ramification group as in \cite{Murai}. 
	By \cite[Theorem 3.5]{Murai}, block induction yields a surjective map $\mathrm{Bl}(H[b],b) \to \mathrm{Bl}(H,b)$.
	However, this map is injective as well since every block of $\mathrm{Bl}(H[b],b)$ is $H$-stable, due to it containing an extension of the trivial character and $H/N$ being abelian. However, we have $H[b] \geq \C_{H}(P) N$ by \cite[Lemma 3]{Dade2}. On the other hand, by \cite[Theorem 9.19, Lemma 9.20]{NavarroBook}, block induction gives a bijection 
	$\mathrm{Bl}(\C_{H}(P)N,b) \to \mathrm{Bl}(H[b],b).$
	Since $H/N$ is cyclic of $\ell'$-order, $\C_{H}(P)N=H[b]$ by \cite[Proposition 3.9]{Murai}. The last statement follows from the main theorem of \cite{Dade2}.
\end{proof}
Table \ref{table}, which was taken from \cite[Table 3]{KessarMalle}, lists all quasi-isolated $2$-blocks of $E_6(q)$.

\begin{table}[htbp]\label{table}
	\caption{Quasi-isolated $2$-blocks in $E_6(q)$}   \label{tab:quasi-E6}
	\[\begin{array}{|r|r|l|llll|}
		\hline
		\text{No.}& C_{\bG^*}(s)^F& (\ell,e)& \bL^F& C_{\bL^*}(s)^F& \lambda& W_{\bG^F}(\bL,\lambda)\\
		\hline\hline
		1&        A_2(q)^3.3& 1& \Ph1^6& \bL^{*F}& 1& A_2\wr3\\
		\hline
		2&        A_2(q^3).3& 1& \Ph1^2.A_2(q)^2& \Ph1^2\Ph3^2.3& 1& A_2\\
		\hline
		3&   \Ph1^2.D_4(q).3& 1& \Ph1^6& \bL^{*F}& 1& D_4.3\\
		&                  &      & \Ph1^2.D_4(q)& \bL^{*F}& D_4[1]& 3\\
		\hline
		4&   \Ph1\Ph2.\tw2D_4(q)& 1& \Ph1^4.A_1(q)^2& \Ph1^4\Ph2^2& 1& B_3\\
		\hline
		5& \Ph3.\tw3D_4(q).3& 1& \Ph1^2.A_2(q)^2& \Ph1^2\Ph3^2.3& 1& G_2\\
		&                  &      & \bG^F& C_{\bG^*}(s)^F& \tw3D_4[\pm1]& 1\\
		\hline
		6& A_2(q^2).\tw2A_2(q)& 1& \Ph1^3.A_1(q)^3& \Ph1^3\Ph2^3& 1& A_2\times A_1\\
		&                  &      &  \Ph1^2.D_4(q)&  \Ph1^2\Ph2^2.\tw2A_2(q)& \phi_{21} & A_2\\
		\hline\hline
		7&        A_2(q)^3.3& 2& \Ph1^2\Ph2^3.A_1(q)& \Ph1^3\Ph2^3& 1& A_1\wr3\\
		&                  & & \Ph1\Ph2^2.A_3(q)& \Ph1^2\Ph2^2.A_2(q)& \phi_{21}& A_1\times A_1\\
		&                  & & \Ph2.A_5(q)& \Ph1\Ph2.A_2(q)^2& \phi_{21}\otimes\phi_{21}& A_1\\
		&                  & & \bG^F& C_{\bG^*}(s)^F& \phi_{21}^{\otimes3}& 1\\
		\hline
		8&        A_2(q^3).3& 2& \Ph2.A_2(q^2)A_1(q)& \Ph1\Ph2\Ph3\Ph6.3& 1& A_1\\
		&                  & & \bG^F& C_{\bG^*}(s)^F& \phi_{21}& 1\\
		\hline
		9&   \Ph1^2.D_4(q).3& 2& \Ph1^2\Ph2^4& \bL^{*F}& 1& D_4.3\\
		&                  & & \bG^F& C_{\bG^*}(s)^F& \phi,\phi',\phi''& 1\\
		\hline
		10&   \Ph1\Ph2.\tw2D_4(q)& 2& \Ph1^2\Ph2^4& \bL^{*F}& 1& B_3\\
		\hline
		11&  \Ph3.\tw3D_4(q).3& 2& \Ph2^2.A_2(q^2)& \Ph2^2\Ph3\Ph6.3& 1& G_2\\
		&                   &      & \bG^F& C_{\bG^*}(s)^F& \phi_{2,1},\phi_{2,2}& 1\\
		\hline
		12&  A_2(q^2).\tw2A_2(q)& 2& \Ph1^2\Phi_2^4& \bL^{*F}& 1& A_2\times A_2\\
		\hline
	\end{array}\]
\end{table}

		 \begin{lemma}\label{block bijection}
		There exists a $\tilde{G}(\tilde{s})$-equivariant defect preserving bijection between the set of $2$-blocks of $G$ associated to $s$ and the set of $2$-blocks of $G(s)$ associated to $s$.
	\end{lemma}
	
	\begin{proof}
		 The number of $2$-blocks of $\mathcal{E}_2(G,s)$ is given by the tables in \cite{KessarMalle}. Note that every numbered line corresponds to one $2$-block whenever $A_{\Levi^\ast}(s)^F =1$. If $A_{\Levi^\ast}(s)^F  \neq 1$ then $\G=E_6$ and the $2$-blocks associated to $s$ are all maximal blocks and correspond to the $|A_{\Levi^\ast}(s)^F|$ many regular-semisimple characters of $\mathcal{E}(L,s)$. 
		
		We now compare this number to the number of unipotent $\ell$-blocks of $G(s)$. Observe that if $\mathbf{H}$ is a connected reductive group then the number of unipotent $2$-blocks of $\mathbf{H}^F$ is known: First use \cite[Proposition 17.1]{MarcBook} to reduce to the case of adjoint (simple) groups and then use the parametrization of $2$-blocks in \cite{Enguehard}. Applying this to the case $\mathbf{H}:=\G^\circ(s)$ we can compute the number of blocks of $G^\circ(s)$. Using the tables of \cite{KessarMalle} and the aforementioned argument the lemma follows whenever $G^\circ(s)=G(s)$ or equivalenty when $A_{\G^\ast}(s)^{F}=1$.
		
		We can therefore assume that $\G$ is of type $E_6$. Let us first assume that $e=1$ and that $G=E_6(q)$. We inspect the cases in Table \ref{table}. In case (6) we have $A_{\G^\ast}(s)^F=1$ and the lemma follows. In cases (3)-(5), the connected centralizer $\C^\circ_{\G^\ast}(s)$ is a Levi subgroup of $\G^\ast$ and so the lemma follows (more naturally) from the Bonnafé--Dat--Rouquier equivalence \cite{Dat}.

 In case (1) or (2) note that $G^\circ(s)$ is of classical type and thus $\mathcal{E}_2(G^\circ(s),1)$ is the set of characters lying in the principal $2$-block by \cite[Theorem 21.14]{MarcBook}. To compute the set of $2$-blocks of $G(s)$ lying over the principal block of $G^\circ(s)$ it is by Lemma \ref{dade} enough to compute the centralizer of a Sylow $2$-subgroup of $G^\circ(s)$ inside $G(s)$. In case (1), a non-trivial element in $A(s)$ permutes the three rational components of $G^\circ(s)$. We therefore obtain that $G(s)$ has one block covering the principal block of $G^\circ(s)$. In case (2) we obtain that a non-trivial element in $A(s)$ acts (modulo inner automorphisms of $G^\circ(s)$) as a field automorphism of order $3$ on $G^\circ(s) \cong A_2(q^3)$. Since $|A_2(q^3)|_2=|A_2(q)|_2$, we obtain that such a field automorphism centralizes a Sylow $2$-subgroup of $G^\circ(s)$. We obtain again by Lemma \ref{dade} that there are exactly three blocks covering the principal block of $G^\circ(s)$.

The same arguments work for $e=2$, and the case $G={}^2 E_6(q)$ follows from Ennola duality.
%
%
	\end{proof}

%
%
%
%
%
%
%
%

%

	\section{Harish-Chandra theory}\label{section 3}

As before, let $\G$ be a simple, simply connected group of exceptional type.
Recall that $s \in \G^\ast$ is called isolated if $\C^\circ_{\G^\ast}(s)$ is not contained in any proper Levi subgroup of $\G^\ast$. Since Levi subgroups are the centralizer of their connected center this is equivalent to $\Z^\circ(\C^\circ_{\G^\ast}(s))=1$, or equivalently to $\C^\circ_{\G^\ast}(s)$ being a semisimple algebraic group. A character $\chi \in \Irr(H)$ of a finite group $H$ is called $2'$-character if $\chi(1)$ is odd. 

\begin{lemma}\label{principal series}
Let $s \in G^\ast$ be a quasi-isolated semisimple element of odd order. Then any $2'$-character of $\mathcal{E}_\ell(G^\circ(s),1)$ lies in the principal series of $G^\circ(s)$.
\end{lemma}

\begin{proof}

We first discuss the case when $s$ is not isolated. In this case $\G$ is of type $E_6$ and $\mathbf{H}:=\G^\circ(s)$ is a Levi subgroup of $\G$ of type $D_4$ by Table \ref{table}. Since the center of $G$ has order $3$, it follows from \cite[Proposition 13.12]{MarcBook} that the center of $\bH$ is connected and thus $\bH^F=[\bH,\bH]^F \times \Z(\bH)^F$. By \cite[Proposition 12.14]{MT} and \cite[Proposition 8.7]{MS} any $2'$-character of $[\bH,\bH]^F$ lies in the principal series. Hence, by compatibility of Harish-Chandra series with the embedding of $[\bH,\bH]^F$ into $\bH^F$, any $2'$-character of $\bH^F$ lies in the principal series.
	Assume now that $s$ is isolated so that $\mathbf{H}:=\G^\circ(s)$ is a semisimple algebraic group.
	Let $\mathbf{H} \hookrightarrow \tilde{\mathbf{H}}$ be a regular embedding and $\tau: \bH_{\ssc} \to \bH$ a simply connected covering. By \cite[Example 1.7.14]{GeckMalle} there exists a regular embedding $\bH_{\ssc} \hookrightarrow \hat{\bH}$ together with a surjective map $\hat{\bH} \to \tilde{\bH}$ whose kernel is a central torus such that the following diagram commutes:
	\begin{center}
		\begin{tikzpicture}
			\matrix (m) [matrix of math nodes,row sep=3em,column sep=4em,minimum width=2em] {
				
				\hat{\bH}  & \tilde{\bH}  \\
				\bH_{sc} & \bH
				\\};
			\path[-stealth]
			(m-2-2) edge node [right] {$\iota$} (m-1-2)
			(m-1-1) edge node [above] {$ $} (m-1-2)
			(m-2-1) edge node [above] {$ $} (m-2-2)
			(m-2-1) edge node [left] {$ $} (m-1-1);
			
		\end{tikzpicture}
	\end{center}
	Let $\chi \in \Irr(H)$ and $\tilde{\chi} \in \Irr(\tilde{H} \mid \chi)$ a character of $\tilde{H}$ covering it. As the kernel of $\hat{\bH} \to \tilde{\bH}$ is connected the map $\hat{H} \to \tilde{H}$ on $F$-fixed points is surjective and we can thus consider $\tilde{\chi}$ via inflation as character of $\hat{H}$ along this surjection. Let $\chi'$ be a character of $H_{sc}$ below $\tilde{\chi}$. By \cite[Theorem 15.11]{MarcBook} and Clifford theory we have the equalites 
	$$\frac{\tilde{\chi}(1)}{\chi(1)}=|\{\lambda \in \Irr(\tilde{H}/H \Z(\tilde{H})) \mid \lambda \tilde{\chi}= \tilde{\chi} \}|$$ and
	$$\frac{\tilde{\chi}(1)}{\chi'(1)}=|\{\lambda \in \Irr(\hat{H}/ H_{\ssc} \Z(\hat{H})) \mid \lambda \tilde{\chi}= \tilde{\chi} \}|.$$
	The map $\hat{H} \to \tilde{H}$ is surjective with central kernel and maps $ H_{\ssc} \Z(\hat{H})$ to $\Z(\tilde{H}) H$. Thus, deflation yields a surjection $\Irr(\hat{H}/ H_{\ssc} \Z(\hat{H})) \to \Irr(\tilde{H}/Z(\tilde{H}) H)$.
	Therefore, $\chi'(1) \mid \chi(1)$. Thus, if $\chi$ is a $2'$-character then so is $\chi' \in \Irr(H_{\ssc})$.
	
	Now, $H_{\ssc} \cong H_1 \times \dots \times H_r$ where each $H_i$ is the fixed points under a Frobenius of a simple algebraic group of simply connected type. An inspection of the relevant tables in \cite{KessarMalle} shows that $H_i$ is never of rational type $C_n(q')$ for $n \geq 1$ and $q'$ an integral power of $q$. By \cite[Theorem 3.3]{MalleGalois} it thus follows that every $2'$-degree character of $H_i$ lies in the principal series. From this, we deduce that $\chi'$ lies in the principal series of $H_{\ssc}$. It thus follows by the commutative diagram above and the compatibility of Harish-Chandra induction with regular embeddings and central quotients that $\chi$ lies in the principal series as well.
	%
\end{proof}

\begin{corollary}\label{HarishChandra}
Let $s \in G^\ast$ be a quasi-isolated $2'$-element and $\Levi_0$ be the minimal $1$-split Levi subgroup associated to $s$, see Lemma \ref{minimal split}.	Then every height zero character of $\mathcal{E}_\ell(G,s)$, which lies in a maximal block, lies in the Harish-Chandra series of a semisimple character in $\mathcal{E}_\ell(L_0,s)$.
\end{corollary}

\begin{proof}
	Let us first assume that the center of $\G$ is connected.
Observe by Remark \ref{torus} or alternatively by checking the tables in \cite{KessarMalle} that $\C_{\Levi_0^\ast}({s})$ is a maximally split torus of $\C_{\G^\ast}({s})$. Therefore, $\C_{\Levi_0^\ast}({st})=\C_{\Levi_0^\ast}({s})$ is a maximal torus of $\C_{\G^\ast}(st)$ for all $t \in \C_{G^\ast}(s)_2$. Thus, by applying \cite[Theorem 4.8.24]{GeckMalle} twice, we can find appropriate Jordan decompositions such that the left square in the following diagram commutes. For $H \in \{G,L_0\}$ we define $\mathcal{J}_H:=J_{H(s),st}^{-1} \circ J_{H,st}$ so that also the right square in the diagram commutes.

	
	
	\begin{center}
		\begin{tikzpicture}
			\matrix (m) [matrix of math nodes,row sep=3em,column sep=4em,minimum width=2em] {
				
				\mathcal{E}({C}_{G^\ast}(s t),1)  & \mathcal{E}(G(s),t)   & \mathcal{E}({G},{st}) \\
				\mathcal{E}({C}_{L_0^\ast}({s} {t}),1) & \mathcal{E}(L_0(s),t)  &  \mathcal{E}({L_0},{st})
				\\};
			\path[-stealth]
			(m-2-2) edge node [right] {$R_{{L_0(s)}}^{{G(s)}}$} (m-1-2)
			(m-1-2) edge node [above] {$J_{G(s),st}$} (m-1-1)
			(m-2-2) edge node [above] {$J_{L_0,st}$} (m-2-1)
			(m-2-1) edge node [left] {$R_{{C}_{L_0^\ast}({s} {t})}^{{C}_{G^\ast}(s t)}$} (m-1-1)
			
					(m-2-3) edge node [right] {$R_{{L_0}}^{{G}}$} (m-1-3)
					(m-1-3) edge node [above] {$\mathcal{J}_{G}$} (m-1-2)
					(m-2-3) edge node [above] {$\mathcal{J}_{L_0}$} (m-2-2)

			;
			
		\end{tikzpicture}
	\end{center}
 Now, under the map $\mathcal{J}_G$ any character $\chi \in \mathcal{E}(G,st)$ with $\chi(1)_2=|G^\ast:\C_{G^\ast}(s)|_2$ corresponds to a character $\psi:=\mathcal{J}_G(\chi) \in \mathcal{E}_2(G(s),t)$ of $2'$-degree. By Lemma \ref{principal series} this character lies in the principal series. It follows that $\chi$ lies in the Harish-Chandra series of the (unique) semisimple character in $\mathcal{E}(L_0,st)$.

	Finally assume that $G$ is of type $E_6$. We consider a regular embedding $\G \hookrightarrow \Gtilde$ such that $\Gtilde=\Gtilde^\ast$. Recall $\tilde{\G}(\tilde{s})=\C_{\Gtilde^\ast}(\tilde{s})$ and that $\tilde{\G}(\tilde{s})=\Z(\tilde{\G}) \G^\circ(\tilde{s})$.
	
	Now suppose that $\chi \in \mathcal{E}_\ell(G,s)$ is a height zero character. Then there exists a character $\tilde{\chi} \in \mathcal{E}_\ell(\tilde{G},\tilde{s})$, which covers $\chi$. Since $2 |\nmid \tilde{G}/G \Z(\tilde{G})|$ this character has necessarily height zero again. As the arguments from above apply in the group $\tilde{\G}$ as well, we deduce that $\tilde{\chi}$ also lies in the principal series. Since Harish-Chandra series are compatible with regular embeddings it follows that $\chi$ also lies in the principal series.
%
\end{proof}

Note that in all cases under consideration, $\C^\circ_{\Levi_0^\ast}(s)$ is a maximal torus and $\Z^\circ(\C^\circ_{\Levi_0^\ast}(s))_{\Phi_1}=\Z^\circ(\Levi_0^\ast)_{\Phi_1}$, see Remark \ref{torus}. Hence, every character in $\mathcal{E}_\ell(\Levi_0,s)$ is cuspidal, see \cite[Propositiion 1.10]{Marc}. 

We let $\Lambda_0$ be an $\N_{G}(\Levi_0)$-equivariant extension map for $L_0 \lhd \N_{G}(\Levi_0)$ (such a map always exists, see the beginning of \cite[Section 5]{MS}). As in \cite[Section 5]{MS} the extension map $\Lambda_0$ uniquely determines for every (cuspidal) character $\theta \in \mathcal{E}_\ell(\Levi_0,s)$ a bijection $$\Irr(W_{\G}(\Levi_0,\theta)) \to \mathcal{E}(G,(\Levi_0,\theta)), \quad \eta \mapsto R_{L_0}^{G}(\theta)_\eta.$$ 

\begin{corollary}\label{necessary height zero}
In the situation of Corollary \ref{HarishChandra}, assume that a character $\chi=R_{L_0}^G(\theta)_\eta \in \mathcal{E}_\ell(G,s)$ satisfies $\chi(1)_{2}=|G^\ast:\C_{G^\ast}(s)|_{2}$. Then $W_G(\Levi_0,\theta)$ contains a Sylow $2$-subgroup of $W_G(\Levi_0,e_s^{L_0})$ and $\eta$ is of $2'$-degree. 
\end{corollary}

\begin{proof}
	We use the notation of Corollary \ref{HarishChandra} and \cite[Section 8.A]{MS}. We argue as in \cite[Lemma 8.9]{MS}.
	Let $D_\chi \in \mathbb{Q}(X)$ as in the remarks preceeding \cite[Lemma 8.2]{MS}. As argued there, the numerator and denominator of $D_\chi$ are coprime to $X-1$. Define $f_\chi':=(|\mathbb{G}|/|\mathbb{L}_0|)_{X'} D_\chi f_\theta$, where $f_\theta$ is the degree polynomial of the cuspidal character $\theta$. Note that $\theta \in \mathcal{E}(L_0,st)$ for some $t \in \C_{\G^\ast}(s)_{\ell}$. Recall that $\C^ \circ_{\Levi_0^\ast}(st)=\C^\circ_{\Levi_0^\ast}(s)$ is a maximal torus of $\Levi_0^\ast$ and so $\theta$ is a regular-semisimple character in the Lusztig series $\mathcal{E}(L_0,st)$. For $\mathbf{H} \in \{\Levi_0^\ast,\G^\ast\}$ let $\mathbb{H}_{s}$ be the degree polynomial of $(\C_{\mathbf{H}}(s),F)$, see \cite[Section 8.A]{MS}. We deduce that $f_\theta=(|\mathbb{L}_0|/|\mathbb{L}_{0,s}|)_{X'}$. It follows that $$g_\chi':=f_\chi' (|\mathbb{G}_s|/|\mathbb{G}|)_{X'}=(|\mathbb{G}_s|/|\mathbb{L}_{0,s}|)_{X'} D_\chi \in \mathbb{Q}(X).$$ We have $\chi(1)=|G:L_0|_{p'} D_\chi(q) \theta(1)$. Therefore, we obtain $g_\chi'(q)=f_\chi'(q)/|G:G(s)|=\chi(1)/|G:G(s)|$ which is an odd integer by assumption. Since this is an integer for all prime powers $q$, it follows that $g_\chi' \in \mathbb{Q}[X]$. Now evaluating $g_\chi'$ at $1$ yields $g_\chi'(1)=|W_{\C_{G^\ast}(s)}(\C_{\Levi_0^\ast}(s))| D_\chi(1)$ by \cite[Theorem 3.2(d)]{MalleHZ}. Moreover, $D_\chi(1)=\eta(1)/|W_G(\Levi_0,\theta)|$ by \cite[Equation 8.2]{MS}. For $i \geq 2$ we have $\Phi_i(q)_2 \geq \Phi_i(1)_2$ (see the proof of \cite[Proposition 8.4]{MS}) and so $g_\chi'(q)_2 \geq g_\chi'(1)_2$. It follows that $g'_\chi(1)$ is odd as well. By Remark \ref{Sylow Weyl} we conclude that $\eta(1)$ and $|W_G(\Levi_0,e_s^{L_0})/W_G(\Levi_0,\theta)|$ have to be odd as well.
	%
	%
	%
	%
\end{proof}

For later we also state the following result about Harish-Chandra series and automorphisms. We let $\mathcal{B} \subset \mathrm{Aut}(\tilde{G})$ be the subgroup generated by field and graph automorphisms as in \cite[Section 2.B]{MS}.

\begin{theorem}\label{HC}
	Assume that there exists an $\N_{{G} \mathcal{B}}(\Levi_0)$-equivariant extension map $\Lambda_0$ for $L_0 \lhd \N_{G}(\Levi_0)$. Assume that $\lambda_0 \in \Irr(L_0)$ is a $\tilde{L}_0$-stable character.
	Let $x\in \N_{\tilde{G} \mathcal{B}}(\Levi_0)$ and $\delta_{\lambda_0,x} \in \Irr(W_G(\Levi_0,{}^x \lambda_0))$ such that $\delta_{\lambda_0,x} \Lambda_0({}^x \lambda_0)={}^x \Lambda_0(\lambda_0)$. 
	Then $${}^x(R_{L_0}^{G_0}(\lambda_0)_{\eta_0})=R_{L_0}^{G_0}({}^x \lambda_0)_{{}^x \eta_0 \delta_{\lambda_0,x}^{-1} }.$$
\end{theorem}

\begin{proof}
	Denote by $W$ the relative Weyl group $W_G(\Levi_0)$.
	We want to deduce the result from \cite[Theorem 5.6,Theorem 5.7]{MS} as explained in the proof of \cite[Proposition 6.3]{MS}. For this we need to check that the assumptions of \cite[Theorem 5.6]{MS} hold. As in \cite[Lemma 6.1,Lemma 6.2]{MS} it suffices to check that $R(\lambda_0) \leq W(\tilde \lambda_0)$ for $\tilde{\lambda}_0 \in \Irr(\tilde{L}_0 \mid \lambda_0)$, where $R(\lambda_0) \leq W(\lambda_0)$ is defined as in the remarks before\cite[Lemma 5.2]{MS}. This is however the case by \cite[Lemma 4.14]{CSSF}.
\end{proof} 

For $x \in \N_{G \mathcal{B}}(\Levi_0)$ the properties of $\Lambda_0$ imply that $\delta_{\lambda_0,x}$ is trivial. In this case we have by the proof of \cite[Proposition 6.3]{MS} that ${}^x(R_{L_0}^{G_0}(\lambda_0)_{\eta_0})=R_{L_0}^{G_0}({}^x \lambda_0)_{{}^x \eta_0}$ without the assumption on $\lambda_0$.

%

\section{The Cabanes property for Sylow $2$-subgroups}\label{section 4}

\subsection{The Cabanes property}

\begin{definition}\label{Cabanes def}
Let $H$ be a finite group and $A$ an abelian subgroup of $H$. Then $H$ is said to be Cabanes with Cabanes subgroup $A$ if $A$ is the unique maximal abelian normal subgroup of $H$.
\end{definition}

It was observed in \cite{Unicite} that the Sylow $\ell$-subgroups of groups of Lie type are Cabanes provided that $\ell \geq 5$.

%

%
Let $A \lhd H$ be an abelian normal subgroup. Then an element $h \in H$ is said to act quadraticely on $A$ if $[h,[h,A]]=1$. Equivalently, we have $(h-1)^2=0$ in $\mathrm{End}(A)$. Quadratic actions are related to the Cabanes property as the following proposition shows.

\begin{proposition}\label{quad prop}
Let $A \lhd H$ be an abelian normal subgroup. Then $A$ is the unique maximal normal abelian subgroup of $H$ if and only if no non-trivial element in $H/A$ acts quadraticely on $A$.
\end{proposition}

\begin{proof}
See \cite[Proposition 2.3]{Unicite}.
\end{proof}

The following example shows that the criterion in the previous proposition is not necessary for a group to have a characteristic maximal normal abelian subgroup.

\begin{example}\label{quad examp}
Assume that $H=D_8=C_4 \rtimes C_2$. Then $H$ has a unique cyclic normal subgroup $A =C_4$ but $h \in H \setminus A$ acts quadraticely on $A$. On the other hand, $H=Q_8$ has three different maximal abelian normal subgroups, with none of them being characteristic.
\end{example}


\begin{lemma}\label{quadratic}
Let $A \lhd H$ be a normal abelian subgroup such that $H/A$ is a $2$-group. Then $H$ is Cabanes unless there exists an involution $h \in H \setminus A$ with $({}^h t t^{-1})^2=1$ for all $t \in A$.
\end{lemma}

\begin{proof}
A short calculation shows that $[h,[h,t]]={}^{h^2} t ({}^h t)^{-2} t$ for $h \in H$ and $t \in A$. If $h$ acts quadraticely on $A$ then so does any power of $h$. Thus, if $H/A$ is a $2$-group it suffices to determine all involutions which act quadraticely on $A$. If $h^2=1$ then $1=[h,[h,t]]=t^2 ({}^h t)^{-2}$ implies ${}^h t \in \{t,t^{-1} \}$ for all $t \in A$.
\end{proof}

\subsection{Sylow $2$-subgroup of groups of Lie type in odd characteristic}\label{Sylow section}

In order to construct Sylow $2$-subgroups of groups of Lie type we use the following theorem of Malle, see \cite[Theorem 5.19]{MalleHZ}.

\begin{theorem}\label{Malle}
	Let $\bH$ be a simple algebraic group with Frobenius $F$ and $\mathbf{T}$ a Sylow $e$-torus, where $e$ is the order of $q$ modulo $4$. Assume that $\bH^F$ is not isomorphic to $\mathrm{Sp}_{2n}(q)$ whenever $n \geq 1$ and $q \equiv 3,5  \mod 8$. Then there exists a Sylow $2$-subgroup $P$ of $\bH^F$ with $\N_{\bH^F}(P) \leq \N_{\bH^F}(\mathbf{T})$.
	%
\end{theorem}

%

\begin{remark}
	It was stated erroneously in \cite[Remark 7.9]{CabanesSurvey} that the Sylow $2$-subgroup $P$ of $\bH^F$ is always Cabanes in the situation of the previous theorem. A counterexample to this is for instance the Sylow $2$-subgroup of $\mathrm{PGL}_2(\mathbb{F}_{13})$ which is not Cabanes. We will however see in the following that in many cases the Sylow $2$-subgroup is actually Cabanes.
\end{remark}

Let $\G$ be a simply connected, simple algebraic group with root system $\Phi$ and base $\Delta$. As in \cite[Theorem 1.12.1]{GLS} let $x_\alpha: \overline{\mathbb{F}}_q \to  \G$ be the Chevalley generators of $\G$ so that $\G=\langle x_\alpha(t) \mid \alpha \in \Phi, t \in \overline{\mathbb{F}}_q \rangle$. Denote by $S:=\{ s_\alpha \mid \alpha \in \Delta \}$ the set of simple reflections. We let $F$ be the Frobenius endomorphism acting as $x_\alpha(t) \mapsto x_{\tau \alpha}(t^q)$ for $t \in \overline{\mathbb{F}}_q$, $\alpha \in \Delta$ and $\tau$ a (possibly trivial) symmetry of the Dynkin diagram associated to $\G$. For $t \in \overline{\mathbb{F}}_q^\times$ and $\alpha \in \Phi$ we define
$$h_\alpha(t) := n_\alpha(t)n_\alpha(-1) \text{ where } n_{\alpha}(t) := x_\alpha(t)x_{-\alpha}(-t^{-1})x_{\alpha}(t),$$
as in the remarks after \cite[Definition 2.1.5]{BS2006}. We let $\T=\langle h_\alpha(t) \mid \alpha \in \Delta \rangle$ which is a maximal torus of $\G$. As $\G$ is simply connected the torus $\T$ is the direct product of the subgroups $h_\alpha(\overline{\mathbb{F}}_q^\times)$, $\alpha \in \Delta$, see \cite[Theorem 1.12.5]{GLS}. Let $W=\langle S \rangle$ be the Weyl group of $\T$ and $w_0$ the longest element in $W$.

Let $I \subset S$ be a subset of the set of simple reflections. We denote by $w_I$ the longest element of the Weyl group associated to the parabolic system $W_I$.

\begin{lemma}\label{longest element}
	Let $\emptyset \neq I$ be a proper subset of $S$. Then there exists $\beta \in \Delta$ such that $w_I(\beta) \in \Phi^{+} \setminus \{ \beta \}$.
\end{lemma}

\begin{proof}
As the Dynkin diagram associated to $\Delta$ is connected there exists $\beta \in \Delta \setminus \Delta_I$ such that $(\alpha,\beta) < 0$ for some $\alpha \in \Delta_I$. We have $w_I(\beta) \in \Phi^+ $ by \cite[Corollary A.26(b)]{MT}. Moreover, by \cite[Lemma B.1]{MT} we have $\alpha + \beta \in \Phi^+$. The roots $\beta$ and $\alpha+\beta$ have the same shape $\mathcal{S}$ with respect to $I$. Moreover, $\beta$ is the root of minimal height of shape $\mathcal{S}$. Therefore, using the hint in \cite[Exercise A.7]{MT}  (and that $w_I^2=1$) we deduce that $w_I(\beta)$ is the root of maximal height of shape $\mathcal{S}$. As there are roots of different height of shape $\mathcal{S}$ (as clearly $\alpha$ and $\alpha+\beta$ have different heights) we must necessarily have $w_I(\beta) \neq \beta$.
\end{proof}

\begin{lemma}\label{coefficients}
Assume that $\Phi$ is an irreducible root system of type $A_n,D_n,E_6$ with base $\Delta$. Let $\alpha \in \Phi^+ \setminus \Delta$ be a positive root and $\alpha=\sum_{\beta \in \Delta} n_\beta \beta$. Then there exist at least two $\beta \in \Delta$ with $n_\beta=1$.
\end{lemma}

\begin{proof}
Without loss of generality we can replace the root system $\Phi$ by the support of $\alpha$ (i.e. the root system generated by the set of simple roots $\beta \in \Delta$ with $n_\beta \neq 0$). This is an irreducible root system again of type $A_n$, $D_n$ with $n >1$ or $E_6$. Let $\Delta' \subset \Delta$ be its base. An inspection of \cite[Table B.1]{MT} shows that the highest root $\alpha_0=\sum_{\gamma \in \Delta'} m_\gamma \gamma$ in this new root system has at least two roots $\gamma \in \Delta'$ with $m_\gamma=1$. Since $1 \leq n_\gamma \leq m_\gamma=1$ by \cite[Proposition B.5]{MT}, we also find that $n_\gamma=1$.
\end{proof}

We also make use of the following well-known theorem, see the main theorem of \cite{Richardson}.

\begin{theorem}\label{Richardson}
Let $W$ be a Weyl group and $w \in W$ an involution. Then there exists a subset $I \subset S$, unique up to $W$-conjugation, such that $w$ is conjugate to $w_I$ and $w_I$ acts by $-1$ on $W_I$.
\end{theorem}

If for instance $W$ is of type $A_n$, then the longest element in $W$ is conjugate to the longest element in a parabolic subgroup of type $(A_1)^{\lfloor n-1/2 \rfloor}$.

\begin{theorem}\label{Cabanes}
	Let $\G$ be a simply connected, simple algebraic group with Frobenius endomorphism $F: \G \to \G$.	Suppose that $\G^F$ is not isomorphic to ${}^3 D_4(q)$.
	Assume that $w_0$ does not act as $-1$ on $W$ or $8 \mid \Phi_e(q)$. Then the Sylow $2$-subgroups of $\G^F$ are Cabanes.
\end{theorem}

\begin{proof}
	 Theorem \ref{Malle} can be used to describe a Sylow $2$-subgroup $P$ of $\G^F$, see \cite[Lemma 3.2]{MS}. If $e=1$ it follows that $P$ is an extension of $\T^F$ by a Sylow $2$-subgroup of $W^F$ while for $e=2$ the Sylow $2$-subgroup $P$ is isomorphic to an extension of $\T^{w_0 F}$ by $W^{w_0 F}$.
	
Let us first assume that $e=1$ and $(\G,F)$ is untwisted.
By Lemma \ref{quadratic} it is enough to show that no involution $w$ in $W$ acts quadraticely on $A:=\T^F \cong (\mathbb{F}_q^\times)^n$. By Theorem \ref{Richardson} we know that $w$ is conjugate to the longest element $w_I$ in $W_I$ for some subset $I$ of the simple reflections. We may thus assume $w=w_I$ and our aim is to construct an element $a \in A$ such that ${}^ w a a^{-1}$ has at least order $4$. 
	
Assume first that $8 \mid (q-1)$ and $w=w_0$ is the longest element acting as $-1$ on $W$. Then for any root $\alpha \in \Delta$ and $t \in \mathbb{F}_q^\times$ of order $8$ we have $${}^{w} h_\alpha(t) h_{\alpha}(t)^{-1}= h_{w(\alpha)}(t) h_{\alpha}(t^{-1})=h_{\alpha}(t^{-2}),$$ which has order $4$.
	

If $I$ is a proper subset of $S$ and $w=w_I$, then by Lemma \ref{longest element} there exists $\alpha \in \Delta \setminus \Delta_I$ such that $w(\alpha) \neq \pm \alpha$. We obtain 
$${}^w h_\alpha(t) h_\alpha(t^{-1})=h_{w(\alpha)}(t) h_{ \alpha}(t^{-1})=h_{w(\alpha)-\alpha}(t).$$
We can now write $w(\alpha)=\sum_{\gamma \in \Delta} n_\gamma \gamma$ and by Lemma \ref{coefficients} there exists some $\beta  \in \Delta \setminus \{\alpha \}$ with $n_\beta=1$. Recall that $\T$ is the direct product of the subgroups $h_\delta(\overline{\mathbb{F}}_q^\times)$, $\delta \in \Delta$. In particular, we deduce that $h_{w(\alpha)-\alpha}(t)$ has the same order as $t \in \mathbb{F}_q^\times$. This shows the claim since $4 \mid (q-1)$.

Assume now still that $e=1$, but $F=\sigma F_q$ for some non-trivial graph automorphism $\sigma$ of order $2$. The group $W^\sigma$ is again an irreducible Weyl group of different type. Its type is determined by the following rule (see \cite[page 121]{MarcBook} or \cite[Table 23.1]{MT}):
$${}^2 A_n \to BC_{[(n+1)/2]}, \quad {}^2 D_n \to BC_{n-1}, \quad {}^2 E_6 \to F_4.$$
Let $\tilde{S}$ (resp. $\tilde{\Delta}$) be the set of $\sigma$-orbits on $S$ (resp. $\Delta$).
We have $W^\sigma= \langle s_\omega \mid \omega \in \tilde{S} \rangle$, see \cite[Lemma 23.3]{MT}. For each $\sigma$-orbit of $\Delta$ fix a representative $\alpha$ and denote by $\tilde{\alpha}$ its $\sigma$-orbit. Similarly, we obtain $\T^F=\prod_{\tilde\alpha \in \tilde{\Delta}} \T_{\tilde{\alpha}}^F$, where $\T_{\tilde{\alpha}}:=\prod_{\gamma \in \tilde{\alpha}} \T_\gamma$. More precisely, such an isomorphism is given by the map $$\T_{\tilde \alpha}^F \to \mathbb{F}_{q^{|\tilde{\alpha}|}}^\times, \, t \mapsto h_{\tilde{\alpha}}(t):=h_\alpha(t) h_{\sigma\alpha}(t^q).$$
Again, by Theorem \ref{Richardson} an involution in $W^\sigma$ is $W^\sigma$-conjugate to an element $w_{\tilde{J}}$, where $\tilde{J} \subset \tilde{S}$. Let $I$ be the union of all orbits $\omega \in \tilde{J}$. By \cite[Lemma C.3]{MT} and \cite[Corollary A.23]{MT} we have $w_I=w_{\tilde{J}}$.
 Let $w_0 \in W^\sigma$ and assume again that $w_0$ acts as $-1$ on $W$. For $\tilde{\alpha}$ we obtain ${}^{w_0} h_{\tilde{\alpha}}(t) h_{\tilde \alpha}(t)^{-1}=h_{\tilde{\alpha}}(t^{-2})$, which has order $4$ again.

 However, if the element $w_0$ acts as the non-trivial graph automorphism on $W$ it follows that $w_0 \in \Z(W^\sigma)$. In particular, the element $w_0$ acts on every root subgroup $\T_\omega^F \cong \mathbb{F}_{q^{|\omega|}}^\times$ as the automorphism $t \mapsto t^{q}$. Let $\alpha \in \Delta$ be an element of the base which is not $\sigma$-stable. Let $t \in \mathbb{F}_{q^2}^\times$ be an element of order $8$ (this is always possible as $8 \mid (q^2-1)$). In particular, ${}^{w_0} h_{\tilde{\alpha}}(t) h_{\tilde{\alpha}}(t)^{-1}= h_{\tilde{\alpha}}(t^{-(q+1)})$ which has order $4$ since $q \equiv 1  \mod  4$.
 
Now let us suppose that $I$ is a proper $\sigma$-stable subset of $S$ and assume that $\alpha \in \Delta \setminus \Delta_I$ is as above with $w(\alpha) \neq \pm \alpha$. If $\alpha$ is $\sigma$-stable then the exact argument from the untwisted case applies here as well. Otherwise, for $t \in \mathbb{F}_{q^2}^\times$ an element of order $8$ consider $${} ^w (h_{\alpha}(t) h_{\sigma\alpha}(t^q)) (h_{\alpha}(t) h_{\sigma\alpha}(t^q))^{-1}.$$ 
Let $\beta \notin \{ \pm \alpha \}$ be a simple root whose coefficient equals $n_\beta= 1$ in the linear combination of $w(\alpha)$. Assume first that $\beta$ is $\sigma$-stable. Hence, $h_\beta(t^2)$ is the contribution by $\beta$ in the element ${} ^w (h_{\alpha}(t) h_{\sigma\alpha}(t^q))$. On the other hand, if $\beta$ is not $\sigma$-stable then a short calculation shows that $h_{\tilde{\beta}}(t^{n_\beta+q n_{\sigma(\beta)}})$ is the contribution by $\beta+\sigma \beta$ in the element ${} ^w (h_{\alpha}(t) h_{\sigma\alpha}(t^q))$. As $q \equiv 1  \mod  4$ we deduce that $n_\beta+q n_{\sigma(\beta)}$ is divisible by four if and only if $4 \mid n_\beta + n_{\sigma \beta}$. However, an analysis of \cite[Table B.1]{MT} together with the proof of Lemma \ref{coefficients} shows that there is no element $w(\alpha) \in \Phi$ with this property.

 
We have therefore dealt with all the cases when $e=1$ and $F$ arbitrary. The remaining cases follow by Ennola duality observing that either $w_0$ or $\sigma w_0$ acts by $-1$. The previous arguments therefore remain true when replacing $q$ by $-q$ everywhere.	 
\end{proof}

In the following remark, we show that the Sylow $2$-subgroups of ${}^3 D_4(q)$ have a characteristic maximal abelian normal subgroup (which is sufficient for most purposes). For this recall that the Weyl group $W^F$ of $\G^F = {}^3 D_4(q)$ is isomorphic to $S_3 \times C_2$. We have $|\T^{w_0^{e-1} F}|=\Phi_e^2(q) \Phi_{3e}(q)$ in this case. The longest element $w_0$ of $W$ acts by inversion and there exists a non-central involution in $W$ which acts by permuting the two copies of $\T^{w_0^{e-1} F}_2 \cong C_{a} \times C_{a}$ with $a=\Phi_e(q)_2$.

\begin{remark}\label{characteristic subgroup}
	Assume that $H$ is a $2$-group with normal subgroup $A=C_a \times C_a$ where $4 \mid a$ and $H/A= \langle w_1,w_2 \rangle$ where $w_1,w_2 \in H/A$ are two commuting involutions such that $w_1$ acts by inversion on $A$ and $w_2$ acts on $A$ by permuting the two copies of $C_a$. Note that $[H,H] \leq A$ as $H/A$ is abelian. From the explicit action of $H/A$ on $A$ we can deduce that $[H,H]$ has index $2$ in $A$. Furthermore, we see that no non-trivial element of $H/A$ centralizes $[H,H]$. From this it follows that $\C_H([H,H])=A$ and thus $A$ is a characteristic subgroup of $H$.
\end{remark}

\section{Sylow $2$-subgroups of $G^\circ(s)$}\label{section 5}


We want to apply the results of the previous section to the Sylow $2$-subgroups of $G^\circ(s)$. For this we need the following lemma:

\begin{lemma}\label{Sylow}
Let $\mathbf{H}$ be a semisimple algebraic group with Frobenius endomorphism $F: \bH \to \bH$ and $\mathbf{H}_1,\mathbf{H}_2$ be two $F$-stable components of $\mathbf{H}$ such that $\mathbf{H}=\mathbf{H}_1 \mathbf{H}_2$. Assume that $\ell \nmid |\mathrm{Z}(\mathbf{H}_1)^F|$. Then the following holds:
\begin{enumerate}[label=(\alph*)]
	\item $\ell \nmid |\mathbf{H}^F: \mathbf{H}_1^F \mathbf{H}_2^F|$.
	\item The kernel of the natural map $\mathbf{H}_1^F \times \mathbf{H}_2^F \to \mathbf{H}_1^F \mathbf{H}_2^F$ is an $\ell'$-subgroup.
\end{enumerate}
In particular, the Sylow $\ell$-subgroups of $\mathbf{H}_1^F \times \mathbf{H}_2^F$ and $\mathbf{H}^F$ are isomorphic.
\end{lemma}

\begin{proof}
	By Lang's theorem and the isomorphism theorem, we have $(\bH_1 \bH_2)^F/\bH_2^F \cong (\bH_1 / \bH_1 \cap \bH_2)^F$. Since $(\bH_1 \cap \bH_2)^F \leq \Z(\bH_1)^F$ is a finite central $\ell'$-subgroup, the group $\bH_1^F/(\bH_1 \cap \bH_2)^F$ is isomorphic to a normal subgroup of $\ell'$-index of $(\bH_1 / \bH_1 \cap \bH_2)^F$ by \cite[Theorem 8.1(a)]{MarcBook}. The lemma now follows from these observations.
%
\end{proof}

Recall the map $\pi: \G \to \G^\ast$ from \ref{dual group}. Let $b=b_G(\Levi,\lambda)$ be an isolated maximal block associated to a semisimple quasi-isolated element $s \in G^\ast$ of odd order. Let $\Levi(s):=\pi^{-1}(\C_{\Levi^\ast}(s))$ and $\Levi^\circ(s):=\Levi(s)^\circ$ be the $F$-stable maximal torus of $\G$ in duality with the torus $\C_{\Levi^\ast}^\circ(s)$.

\begin{lemma}\label{centralizer}
 With the notation from above, we have $\C_{\G}(L(s)_2)=\Levi^\circ(s)$, unless $b$ is one of the blocks in case (2) or (8) of Table \ref{table}.
\end{lemma}

\begin{proof}
 We have $\Z(\Levi^\ast) \subset \Levi^\circ(s)^\ast$ since $\Z(\Levi^\ast)$ centralizes $s$ and so
 $\Levi^\circ(s)_2^F \geq \Z(\Levi)_2^F$. Hence, $\Levi=\C_{\G}(\Z(\Levi)_2^F) \geq \C_{\G}(\Levi^\circ(s)_2^F)$ and so $\C_{\Levi}(\Levi^\circ(s)_2^F)=\C_{\G}(\Levi^\circ(s)_2^F)$. By an inspection of \cite{KessarMalle} the Levi subgroup $\Levi$ has only components of type $A$ and thus $2$ is good for $\Levi$.
 
 We claim that $2 \nmid |\Z(\Levi):\Z(\Levi)^\circ|$. If $G$ is not of type $E_7(q)$ this follows from \cite[Proposition 13.12]{MarcBook}. For $G=E_7(q)$, we go through the cases in \cite[Table 4]{KessarMalle}. In case (1) in $\Levi$ is a maximal torus, so there is nothing to check and in case (2) we note that $\Levi$ is a Levi subgroup of type $A_1(q)^3 \Phi_1^4$ (or its Ennola dual) whose relative Weyl group is of type $C_3 \times A_1$. This Levi subgroup corresponds to the simple roots $\{\alpha_1,\alpha_2,\alpha_5\}$ (in Bourbaki labeling). From the data given on the Homepage of Lübeck \cite{Luebeck} we conclude that the center of $\Levi$ is connected.
 
Since we are not in case (2) or (8) of Table \ref{table}, it follows by an inspection of the tables in \cite{KessarMalle} that $\Levi^\circ(s)$ is an $E_{q,\ell}$-torus of $(\Levi,F)$, where $E_{q,2}=\{ d : 2 \mid \Phi_d(q)\}$. By \cite[Lemma 13.17(a)]{MarcBook} we therefore deduce that $\C^\circ_{\Levi}(\Levi^\circ(s)_2^F)=\C^\circ_{\Levi}(\Levi^\circ(s))=\Levi^\circ(s)$, where the last equality follows from the fact that $\Levi^\circ(s)$ is a maximal torus of $\Levi$. By \cite[Proposition 13.16]{MarcBook} we deduce that $\C_{\Levi}(L(s)_2)=\Levi^\circ(s)$ since $2 \nmid |\Z(\Levi^\ast):\Z(\Levi^\ast)^\circ|$ by \cite[Proposition 13.12]{MarcBook}.
\end{proof}

\begin{remark}
Note that in case (2) of Table \ref{table} we have $L(s)_2=\Z(L)_2$ and so $\C_{G}(L(s)_2)=L$ by \cite[Lemma 4.2]{KessarMalle}.
\end{remark}

According to \cite[Proposition 5.20]{MalleHZ} there exists a Sylow $2$-subgroup ${P'}$ of $G(s)$ with ${P'} \leq \mathrm{N}_{G^\circ(s)}(\Levi^\circ(s))$. In particular, ${D'}:=L(s)_2$ is a normal subgroup of ${P'}$

\begin{lemma}\label{Sylow Cabanes}
	Suppose that we are in the situation of Lemma \ref{centralizer}. Then the Sylow $2$-subgroup $P'$ of $G^\circ(s)$ is Cabanes with Cabanes subgroup $D'=L(s)_2$.
\end{lemma}

\begin{proof}
We consider a simply connected covering map $\pi:\G^\circ(s)_{\ssc} \to \G^\circ(s)$ and we lift the Frobenius endomorphism $F$ to a Frobenius endomorphism on $\G^\circ(s)_{\ssc}$ (also denoted by $F$) such that $\pi$ is $F$-equivariant. Let $\mathbf{H}_{\ssc}$ be a minimal $F$-stable component of $\G^\circ(s)$ and denote by $\mathbf{H}$ its image in $\G^\circ(s)$ under $\pi$. An inspection of the corresponding tables in \cite{KessarMalle} shows that $\Z(\mathbf{H}_{\ssc})$ is always a $2'$-group
except when $\G$ is of type $E_7$ and $\C_{\G^\ast}(s)$ is of type $A_5 A_2$, so that exactly one component of the centralizer has a center with  nontrival $2$-part. We can therefore apply Lemma \ref{Sylow} which shows that it's enough to prove that the Sylow $2$-subgroups of $\mathbf{H}^F$ are Cabanes. As the kernel of the natural isogeny $\mathbf{H}_{\ssc} \to \mathbf{H}$ is always a central subgroup of $2'$-order it follows that the Sylow $2$-subgroups of $\mathbf{H}^F$ and $\mathbf{H}_{\ssc}^F$ are isomorphic and the latter are Cabanes by Theorem \ref{Cabanes}. It thus follows that the Sylow $2$-subgroups of $G^\circ(s)$ are Cabanes.
%

 We must finally show that $L(s)_2$ is the Cabanes subgroup of $P'$. However, by Lemma \ref{centralizer} we find that $\C_{P'}(L(s)_2)=L(s)_2$. Therefore, $L(s)_2$ is a maximal normal abelian subgroup of $P'$ and consequently the Cabanes subgroup of $P'$. 
\end{proof}

%
%
\section{Local block theory}\label{section 6}

	Our aim now is to discuss the local situation. Let $(\Levi,\lambda)$ be the quasi-central $e$-cuspidal pair associated to a quasi-isolated maximal block $b$ of $G$.
	By the proof of \cite[Theorem 7.12]{KessarMalle} we have $L=C_G(\Z(L)_2)$ and $(\Z(L)_2, b_{L}(\lambda))$ is a $b$-Brauer pair. Moreover, $(\Z(L)_2, b_{L}(\lambda))$ is contained in a maximal $b$-Brauer pair $(P,b_P)$, where $P \leq \mathrm{N}_G(\Levi,\lambda)$ has a normal series
$$\Z(L)_2 \lhd D:=C_P(\Z(L)_2) \lhd P.$$
The aim of this section is to give an alternative interpretation of the subgroup $D$ and show that $D$ is (in most cases) the maximal abelian normal subgroup of $P$.

%
By \cite[Theorem 1.2]{KessarMalle}, the quotient group $P/D$ is isomorphic to a Sylow $2$-subgroup of $W_G(\Levi,\lambda)$. Moreover, $D/\Z(L)_2$ is isomorphic to a Sylow $2$-subgroup of $\Levi^F/\Z(\Levi)^F [\Levi,\Levi]^F$. As observed in \cite[Proposition 2.7(g)]{KessarMalle} this implies that $D=\Z(L)_2$ whenever $\Z^\circ(\Levi)^F \cap [\Levi,\Levi]^F$ is a $2'$-group.
	
%
%
%
	The following proposition gives (in most cases) a different interpretation of the defect group~$P$. 
	
	\begin{proposition}\label{local}
		Let $b$ be a maximal quasi-isolated block of $G$. Suppose that $b$ is not one of the blocks associated to a semisimple element with centralizer $A_2(\pm q^3).3$ or ${}^3 D_4( q)$ in $E_6(\pm q)$.
		Then there exists a Sylow $2$-subgroup $P'$ of $G(s)$ such that the following holds: $P'$ is a defect group of $b$ and for a maximal $b$-Brauer pair $(P',b_{P'})$ we have $\N_G(P',b_{P'}) = \N_{G(s)}(P'
		)$.
	
\end{proposition}

\begin{proof}
	Note that if $s$ is quasi-isolated but not isolated then $\C^\circ_{\G^\ast}(s)$ is a Levi subgroup of $\G^\ast$. Hence, in this case the result follows easily from the main theorem of \cite{Dat}. We may therefore assume that $s$ is isolated.
	
	



Denote ${D'}:=L(s)_2$ and note that $A_{\Levi^\ast}(s)^F=1$ in all considered cases so that $\Levi(s)$ is a maximal torus of $\Levi$. We therefore have a Bonnafé--Rouquier Morita equivalence, see \cite{Dat}, between $\Lambda \Levi^F e_s^{\Levi^F}$ and $\Lambda \Levi(s)^F e_s^{\Levi(s)^F}$. By Lemma \ref{centralizer} we have $$\C_{\Levi^F}({D'})=\C_{\Levi(s)^F}({D'})=\Levi(s)^F.$$
Let $\hat{s} \in  \Irr(L(s))$ be the linear character associated to $s$ by duality and denote $b_{D'}:=b_{L(s)}(\hat{s})$.
We deduce by the description of the Brauer morphism in Lemma \ref{Brauer morphism}
that the subpair $({D'},b_{D'})$ is a Brauer pair of the block 
%
$b_{L}(\lambda)$. On the other hand, the pair $(\Z(L)_2,b_{L}(\lambda))$ is a $b$-Brauer pair, see the proof of \cite[Theorem 7.12(d)]{KessarMalle}. Hence, by transitivity of Brauer pairs we obtain that $({D'},b_{D'})$ is a $b$-Brauer pair. Note that $\Z({D'})={D'}$ is a defect group of the block $b_{D'}$. Hence, $({D'},b_{D'})$ is a self-centralizing (or centric) $b$-subpair with canonical character $\hat{s}$, see \cite[Definition 22.12]{MarcBook}.
Recall from Remark \ref{intrinsic} and Remark \ref{minimal split} that $\Levi^\circ(s)$ is the centralizer of a Sylow $e$-torus of $\G^\circ(s)$. The canonical character $\hat{s}$ of the block $b_{L(s)}(\lambda)$ is ${P'}$-stable, as ${P'} \leq G(s)$. Let $b_{P'}$ be the unique block of $\C_G({P'}) \leq \C_G({D'})=L(s)$ below $b_{D'}$. From \cite[Proposition 22.14, Remark 22.15]{MarcBook} we deduce that $({P'},b_{P'})$ is again a centric $b$-Brauer pair. By construction $|{P'}|=|G(s)|_2=|\C_{G^\ast}(s)|_2$ and so $({P'},b_{P'})$ is a maximal $b$-Brauer pair.

By Proposition \ref{CE}(b) the stabilizer of the linear character $\hat{s} \in \Irr(L(s))$ in $\N_G(\Levi(s))$ is $\N_{G(s)}(\Levi(s))$.
 If $x \in \N_G({P'},b_{P'})$ then by Lemma \ref{Sylow Cabanes} it follows that $x \in \N_G({D'},b_{D'})$ and thus $x$ must stabilize the canonical character $\hat{s}$ of the block $b_{D'}$. We deduce that $x \in \N_{G(s)}({P'})$. Conversely, any $x \in \N_{G(s)}({P'})$ stabilizes $\hat{s}$ and therefore also its restriction to $\C_G({P'})$. Therefore, $\N_{G(s)}({P'})=\N_G({P'},b_{P'})$ as claimed.
\end{proof}


\begin{lemma}\label{local 2}
The pair $(D',P')$ constructed in the previous proposition is $G$-conjugate to $(D,P)$.
\end{lemma}

\begin{proof}
Note that $\Z(\Levi)^\circ_{\Phi_e}=\Levi^\circ(s)_{\Phi_e}$ by Remark \ref{torus}. Thus, 
$$P' \leq \N_{G^\circ(s)}(\Levi^\circ(s)) \leq \N_{G^\circ(s)}(\Levi)$$
as $\Levi$ is an $e$-split Levi subgroup of $\G$ and so $\Levi=\C_{\G}(\Z(\Levi)^\circ_{\Phi_e})$. Moreover, note that $\lambda \in \mathcal{E}(L,s)$ corresponds under Jordan decomposition to the regular-semisimple character of $\mathcal{E}(\C_{\Levi^\ast}(s),1)$. It follows that $\mathrm{N}_{G^\circ(s)}(\Levi) \leq \mathrm{N}_{G}(\Levi,\lambda)$. Consequently, the group $P'$ constructed in the previous proposition can be chosen to coincide with the defect group $P$ constructed in \cite{KessarMalle}. Hence, we can assume that $P=P'$.

%
Furthermore, by \cite[Proposition 2.7(d)]{KessarMalle} we have $P \cap L= D$. Therefore, $D'=L(s)_2 \leq D$. On the other hand, $D'=L(s)_2$ is by Lemma \ref{Sylow Cabanes} the maximal normal abelian subgroup of $P$ and thus $D=D'$.
\end{proof}

\begin{remark}\label{br application}
Let $c$ be the block of $G(s)$ corresponding to $b$ under the Jordan decomposition $\mathcal{J}$ from Lemma \ref{jordan construct}. Then $\mathcal{J}$ induces a bijection $\Irr_{0}(G(s),c) \to \Irr_{0}(G,b)$. We observe that the block $b_D=b_{L(s)}(\hat{s})$ is a $c$-subpair. Let $B_D$ be the unique block of $\N_G(D)$ covering $b_D$. Then induction defines a bijection $\Irr_0(\N_{G}(D,b_D), b_D) \to \Irr_0(\N_G(D), B_D)$. 
Since the block $c$ has maximal defect in $G(s)$ and $\N_{G(s)}(P) \leq \N_{G(s)}(D)$ it follows by the main theorem of \cite{BroughRuhstorfer} that $|\Irr_0(G(s),c)|=|\Irr_0(\N_{G(s)}(D),b_D)|$. Hence, $|\Irr_0(\N_G(D), B_D)|=|\Irr_{0}(G,b)|$.
\end{remark}

\subsection{The remaining cases}
We give an analysis of the defect group in the cases which were left out in Proposition \ref{local}. We need the following easy group theoretic lemma.
\begin{lemma}\label{wreath}
	Let $r$ be a natural number and $P$ a finite group with normal subgroup $A^r$ of index $r$ such that some $x \in P\setminus A^r$ acts by transitively permuting the $r$ copies of $A^r$. Then $P \cong A \wr C_r$.
\end{lemma}

\begin{proof}
	Consider the map $f: A \to A^r$ which embedds $A$ into the first component of $A^r$ and let $\pi \in C_r$ be the permutation induced by $x$ on $A^r$. Since $x$ acts transitively on the $r$ components of $A$, we have $\Z(P)=\C_{A^r}(x)=\{ f(a) \, {}^x f(a) \cdots {}^{x^{r-1}} f(a) \mid a \in A \}$. As $x^r \in \C_{A^r}(x)$, we thus have $x^r=f(a) \, {}^x f(a) \cdots {}^{x^{r-1}} f(a)$ for some $a$. Define $y:= f(a)^{-1} x$. Then $y^r=f(a^{-1}) \, {}^x f(a^{-1}) \cdots {}^{x^{r-1}} f(a^{-1}) x^r=1$. It is then easily checked that there exists a unique isomorphism $g:A \wr C_r \to P$ with $g(a_1,\ddots,a_r)= f(a_1) {}^x f(a_2) \cdots {}^{x^{r-1}} f(a_r)$ and $g(\pi)=g(y)$.
\end{proof}

\begin{lemma}\label{local2}
		Suppose that $b$ is a maximal block of $G$ associated to a semisimple element $s$, which was excluded in Proposition \ref{local}. Then $D'=L(s)_2$ is a characteristic subgroup of a defect group $P'$ of $b$.  Moreover, the subgroup $\Z(L)_2$ is a characteristic subgroup of $P'$.
\end{lemma}

\begin{proof}
	We shall first assume that $b$ is quasi-isolated, which is not isolated. According to the proof of Proposition \ref{local} and the main theorem of \cite{Dat} we can assume that $P'$ is a Sylow $2$-subgroup of $G(s)$ with normal subgroup $D'=L(s)_2$. Also observe that $L(s)_2=\Z(L)_2$ by Table \ref{table}. Moreover the information given there also shows that $P'$ is isomorphic to a Sylow $2$-subgroup of $[G^\circ(s),G^\circ(s)] \cong {}^3 D_4(q)_{\mathrm{sc}}$. The claim follows therefore from Remark \ref{characteristic subgroup}. 
	
We can therefore assume that $b$ is an isolated block.
We first show that there exists indeed a defect group of the block $b$ containing $L(s)_2$. According to the main theorem of \cite{Dat} there is a splendid Rickard equivalence between $\mathcal{O} L e_s^{L}$ and $\mathcal{O} L(s) e_s^{L^\circ(s)}$. Here, $L(s)/L^\circ(s)$ is isomorphic by duality to $\C_{L^\ast}(s)/\C^\circ_{L^\ast}(s)$. As there are exactly $3$ blocks of $L(s)$ covering the principal block of $L^\circ(s)$ by Lemma \ref{dade}, we have by the main theorem of \cite{Dade2} a splendid Rickard equivalence between $\mathcal{O} L^\circ(s) e_s^{L^\circ(s)}$ and $\mathcal{O} L b_L(\lambda)$. As $D'=L(s)_2$ is a defect group of the block $e_s^{L^\circ(s)}$ this yields by the main theorem of \cite{Dat} a local Rickard equivalence between $\mathcal{O} L^\circ(s) e_s^{L^\circ(s)}$ and $\mathcal{O} \C_{L}(D') b_{D'}$ given by Deligne--Lusztig induction followed by multiplication with the idempotent $b_{D'}$. It follows that $(D',b_{D'})$ is a $b_L(\lambda)$-subpair whose canonical character $\lambda_0$ is the unique irreducible constituent of $R_{L^\circ(s)}^{\C_L(D')}(\hat{s})$ which lies in $b_{D'}$. By transitivity of the Brauer subpair inclusion we deduce that $(D',b_{D'})$ is a $b$-Brauer pair. Therefore, there exists a maximal $b$-Brauer pair $(P',b_{P'})$ containing it. As $(D',b_{D'})$ is a maximal $b_L(\lambda)$-Brauer pair it follows that $(D',b_{D'})$ is maximal among the $b$-subpairs with $D' \leq \C_G(\Z(L)_\ell)=L$.  Therefore, $(D',b_{D'})$ is a $b$-subpair as in \cite[Proposition 2.7]{KessarMalle}. We can therefore assume that $(D',P')=(D,P)$. It follows from \cite[Proposition 2.7]{KessarMalle} that $P/D$ is isomorphic to a Sylow $2$-subgroup of $W_G(L,\lambda)/L$. This allows us to describe the action of $P$ on $Z(L)_2$ since $L=\C_{G}(Z(L)_2)$.
 
	In case (2) we have $D=\Z(L)_2 \cong A \times A$ such that any element $\sigma \in P \setminus D$ acts by permuting these two components. Therefore, by Lemma \ref{wreath}, $D \cong C_{(q-1)_2} \wr C_2$ which is Cabanes. In case (8) we have that $\Levi^\circ(s)^F \cong \mathbb{F}_{q^6}^\times$ is cyclic. We argue now as in case (6) of \cite[Proposition 6.4]{KessarMalle}. As in this proof we see that $D \cong C_{(q+1)_2 2}$ and any $\sigma \in P \setminus D$ acts on the cyclic subgroup $\Z(L)_2$ of order $(q+1)_2$ by inversion. From this we can conclude as in \cite[Proposition 6.4]{KessarMalle} that $\Z(L)_2$ and $D$ are characteristic subgroups of $P$.
\end{proof}
We will see later in \ref{more info}, that $P'$ is indeed conjugate to a Sylow $2$-subgroup of $G(s)$. However we require explicit calculations to obtain this.


\subsection{Suitable subgroups for the iAM-condition}

The following lemma will allow us to do calculations with the normalizer $\N_G(\Levi)$ instead of the less accessible subgroup $\N_{G}(P)$.

\begin{lemma}\label{L is suitable}
Let $b$ be a maximal block of $G$ with associated $e$-cuspidal pair $(L,\lambda)$.
	\begin{enumerate}[label=(\alph*)]
		\item The subgroup $\N_G(\Levi)$ is a proper $\mathrm{Aut}(G)_{P,b}$-stable subgroup of $G$ and $\N_G(\Levi)$ contains $\N_G(P)$.
		\item There exists a unique block $B$ of $\N_G(\Levi)$ covering the block $b_L(\lambda)$ of $L$ and $B^G=b$.
		\item Let $\tilde \lambda$ be a character of $\tilde L$ covering $\lambda$. Then there exists a block $\tilde{b}$ of $\tilde{L}$ such that $(\Z(\tilde{L})_2,b_{\tilde{L}}(\tilde{\lambda}))$ is a $\tilde{b}$-Brauer pair and a unique block $\tilde B$ of $\N_{\tilde G}(\tilde{\Levi})$ covering the block $b_{\tilde{L}}(\tilde{\lambda})$ of $\tilde{L}$ and $\tilde{B}^G=\tilde{b}$.
	\end{enumerate}
	
\end{lemma}

\begin{proof}
	Let $\sigma: \G \to \G$ be a bijective morphism commuting with $F$ which stabilizes both $P$ and $b$. In the situation of Lemma \ref{local2}, $\Z(L)_2$ is characteristic in $P$ and so $\sigma$ stabilizes $\Z(L)_2$ and hence also its centralizer $L=\C_{G}(\Z(L)_2)$.
	
	Assume now that we are in the situation of Proposition \ref{local}. Then $\sigma$ stabilizes the characteristic subgroup $D=L(s)_2$ of $P$ and also its centralizer $\Levi^\circ(s)=\C_{\G}(D)$, see Lemma \ref{centralizer}. As $\Z(\Levi)^\circ_{\Phi_e}=\Levi^\circ(s)_{\Phi_e}$ it follows that $\sigma$ also stabilizes $\Levi=\C_{\G}(\Z(\Levi)^\circ_{\Phi_e})$. This shows that $\N_G(\Levi)$ is an $\mathrm{Aut}(G)_{P,b}$-stable subgroup of $G$.
	Moreover, the equality $\N_G(\Levi)=G$ would imply $L=G$, as $G$ has no non-central proper normal subgroups. By \cite[Theorem 1.2]{KessarMalle} this would mean that $P \leq \Z(G)$ which is not the case. Thus $\N_{G}(\Levi)$ is a proper subgroup of $G$.
	
	For part (b) observe that any defect group of the block $b_L(\lambda)$ contains $\Z(L)_2$, see \cite[Proposition 2.5(e)]{KessarMalle}. Since $L=\C_G(\Z(L)_2)$ it follows from \cite[Theorem 9.19,9.20]{NavarroBook} that there exists a unique block of $\N_G(\Levi)$ covering $b_L(\lambda)$. Since $(\Z(L)_2,b_L(\lambda))$ is a $b$-Brauer pair it follows that $B^G=b$. 
	
	From the proof of part (a) of \cite[Lemma 7.11]{KessarMalle} we deduce that there exists a block $\tilde{b}$ of $\tilde{G}$ covering $b$ such that $(\Z(\tilde{L})_2,b_{\tilde{L}}(\tilde{\lambda}))$ is a $\tilde{b}$-Brauer pair. Moreover, the proof of said lemma shows that $\C_{\tilde{G}}(\Z(\tilde{L})_2)=\C_{\tilde{G}}(\Z(L)_2)=\tilde{L}$. The last claim of part (c) follows therefore with the same arguments as in part (b).
\end{proof}

\section{Analysis of non-maximal blocks}\label{section 7}

The non-maximal blocks for quasi-isolated elements have some similarities with maximal blocks but they also require some additional arguments. Let $b$ be a quasi-isolated non-maximal block of $G$. Recall that by the results of \cite{KessarMalle} and \cite{Enguehard} there exists some $e$-cuspidal pair $(\Levi,\lambda)$ with $\lambda$ of quasi-central defect defining these blocks.
The following lemma describes all possible situations:

\begin{lemma}\label{class nonmax}
	There are non-maximal quasi-isolated blocks $b=b_G(\Levi,\lambda)$ of non-central defect in the following cases for $e=1$.
	\begin{enumerate}[label=(\roman*)]
		\item $G=E_7$ and $s=1$: Case (3) in \cite[p. 354]{Enguehard}; $L=E_6 \Phi_1$ and $W_G(\Levi)=W_G(\Levi,\lambda) \cong C_2$
		\item $G=E_8$ and $s=1$: Case (3) in \cite[Table I]{Enguehard}; $L=E_6 \Phi_1^2$ and $W_G(\Levi)=W_G(\Levi,\lambda) \cong S_3 \times C_2$
		\item $G=E_8$: Case (4) \cite[Table 5]{KessarMalle}: $L=E_6 \Phi_1^2$  and  $W_G(\Levi,\lambda) \cong S_3$;
		Case (6) \cite[Table 5]{KessarMalle}: $L=E_7 \Phi_1$ and $W_G(\Levi,\lambda) \cong C_2$.
		
	\end{enumerate}
	The cases for $e=2$ are obtained by Ennola duality.
\end{lemma}

 Note that in the cases of Lemma \ref{class nonmax} the character $\lambda \in \Irr(L)$ is always of central defect. In the unipotent cases this is immediate from the parametrization of \cite{Enguehard} and in the non-unipotent case this follows by an inspection of the relevant tables in \cite{KessarMalle} together with \cite[Proposition 2.5(g)]{KessarMalle}.

We collect some more common properties in the following lemma:

\begin{lemma}\label{properties}
	Let $b=b_G(\Levi,\lambda)$ be a non-maximal quasi-isolated block as in Lemma \ref{class nonmax}. Let $P$ be a defect group of $b$ which contains $\Z(L)_2$.
	\begin{enumerate}[label=(\alph*)]
		\item We have $\Levi=\C_{\G}(\Z(L)_2)$ and $\Levi(s)=\C_{\G}(\Z(L(s))_2)$.
		\item The subgroup $D:=\Z(L(s))_2$ is a characteristic subgroup of $P$ with Brauer pair $(D,b_D)$ with $b_D:=b_{L(s)}(\mathcal{J}_L(\lambda))$. Moreover, $\N_G(D,b_D)=\N_{G(s)}(D)$.
	\end{enumerate}
\end{lemma}
\begin{proof}
	The first statement of part (a) follows from \cite[Theorem 7.12]{KessarMalle}, while the second part follows from an inspection of \cite[Table 5]{KessarMalle}.
	
	We now show part (b) for the unipotent blocks.
Observe that $(\Z(L)_2,b_L(\lambda))$ is a $b$-Brauer pair. Moreover, an inspection of \cite[Table 5]{KessarMalle} shows that $\Levi(s)$ is a Levi subgroup of $\Levi$. Hence, by Bonnafé--Rouquier, we have a Morita equivalence between $\mathcal{O} L e_s^{L}$ and $\mathcal{O} L(s) e_s^{L(s)}$ given by Deligne--Lusztig induction. In particular, the block $b_{L(s)}(\mathcal{J}_L(\lambda))$ corresponds to $b_L(\lambda)$ under this Morita equivalence. Since $\Levi(s)=\C_{\G}(D)$, it follows by the main result of \cite{Dat} that we have a local Rickard equivalence between $\mathcal{O} L(s) b_{L(s)}(\mathcal{J}_L(\lambda))$ and $\mathcal{O} L(s) b_D$, where $(D,b_D)$ is a $b_L(\lambda)$-subpair, which is induced by the Deligne--Lusztig induction functor from $L(s)$ to $L(s)$. Since the latter functor is clearly the identity, it follows $b_D=b_{L(s)}(\mathcal{J}_L(\lambda))$, as expected. Moreover, the group $D$ is a characteristic subgroup of $P$ by the proof of \cite[Proposition 6.4]{KessarMalle}. The normalizer property is clear from \cite[Table 5]{KessarMalle}.

	For the unipotent blocks observe that $(\Z(L)_2,b_L(\lambda))$ is a $b$-Brauer pair by the proof of \cite[Theorem 7.12]{KessarMalle} and the normalizer statement is contained in Lemma \ref{class nonmax}. We therefore only need to show that $\Z(L)_2$ is a characteristic subgroup of $P$. For this we use the description of defect groups from \cite[Section 3.2]{Enguehard}. We go through the different cases in Lemma \ref{class nonmax}. In case (i) the defect groups are dihedral by \cite[Section 3.2]{Enguehard} and thus the cyclic group $\Z(L)_2$ of index $2$ in $P$ is characteristic. In case (ii)
	recall that $P \leq \N_G(\Levi,\lambda)$ and $P/\Z(L)_2$ is isomorphic to a Sylow $2$-subgroup of $W_G(\Levi,\lambda)$. This subgroup is generated by two involutions $w_1,w_2$ such that $w_1$ acts by inversion and $w_2$ by permuting the two components of $\Z(L)=\Phi_1^2$. Hence, the claim follows in this case from Remark \ref{characteristic subgroup}. 
\end{proof}

Let $\bH$ be a connected reductive group with Frobenius endomorphism $F: \bH \to \bH$ defining an $\mathbb{F}_q$-structure.
For $d \in \{1,2\}$ observe that by \cite[Lemma 3.17]{Cabanesgroup} the centralizer $\mathbf{S}$ of a Sylow $d$-torus of $(\bH,F)$ is a maximal torus of $\bH$. We say that a character in $\Irr(H)$ lies in the principal $d$-Harish-Chandra series if it is a constituent of $R_{\mathbf{S}}^{\bH}(\lambda)$ for some $\lambda \in \Irr(\mathbf{S}^F)$. 

\begin{lemma}\label{typeA}
	Let $\bH$ be a connected reductive group with connected $\Z(\bH)$ and all whose simple components are of type $A$ and $F: \bH \to \bH$ a Frobenius endomorphism defining an $\mathbb{F}_q$-structure. Then every $2'$-character of $\bH^F$ lies in the principal $1$- and $2$-Harish-Chandra series. 
\end{lemma}

\begin{proof}
	Let $d \in \{1,2\}$. We have $\chi \in \mathcal{E}(H,s)$ for some semisimple element $s \in H^\ast$ such that $\chi(1)=|H^\ast: \C_{H^\ast}(s)|_{p'} J_{H,s}(\chi)(1)$. Our assumption implies by \cite[Proposition 13.16]{MarcBook} that $\C_{\bH^\ast}(s)$ is a Levi subgroup of $\bH^\ast$. In particular, the polynomial order of $\C_{\bH^\ast}(s)$ divides the polyomial order of $\bH^\ast$. Hence, in order for $\C_{H^\ast}(s)$ to contain a Sylow $2$-subgroup of $\bH^\ast$, the group $\C_{\bH^\ast}(s)$ must contain a Sylow $d$-torus of $\bH^\ast$. Moreover, by the order formula of \cite[Corollary 4.6.25]{GeckMalle}, the character $J_{H,s}(\chi)$ must necessarily lie in the principal $d$-Harish-Chandra series since $\Phi_d(q)$ would divide $J_{H,s}(\chi)(1)$ otherwise. It follows from Theorem \ref{DigneMichel}(1) that $\chi$ must lie in the principal $d$-Harish-Chandra series as well.
\end{proof}

\begin{lemma}\label{nonmaximal}
	For every non-maximal block $b$ of $G$ with defect group $P$ occuring in Lemma \ref{class nonmax} we have $|\Irr_0(b)|=|P:P'|$.
\end{lemma}
\begin{proof}
	We go through the cases of Lemma \ref{class nonmax}. Since the defect groups in case (i) are dihedral, this equality is well-known, see \cite[Theorem 8.1]{Sambale}.
	Suppose therefore now that we are in case (ii), i.e. $b$ is the unipotent block of $E_8(q)$ associated to the $e$-cuspidal character $({}^{(2)} E_6,{}^{(2)} E_6[\theta])$. The description of characters in this block is given by \cite[Theorem B]{EnguehardJordandecomp}, see also the proof of \cite[Theorem 6.2]{MalleBrauer}. We have $\Irr(b) \subset \mathcal{E}_2(G, 1)$, and for any $2$-element $t \in G^\ast$, $\Irr(b) \cap \mathcal{E}_2(G, t)$ is in bijection with the irreducible characters in a
		corresponding unipotent block of $\C_{G^\ast}(t)$, belonging to the "same" $e$-cuspidal pair (apart
		from certain exceptions as described in \cite[Proposition 17]{Enguehard}). In particular, the only Lusztig series $\mathcal{E}(G,t)$, with $t$ a $2$-element, that can contribute to the block $b$ have centralizer $\C_{\G^\ast}(t)$ of type
	$${}^{(2)} E_6,{}^{(2)} E_6 A_1, E_7, E_7 A_1 , E_8$$
	and these contribute with $|\Irr(b) \cap \mathcal{E}(G,t)|=1,2,2,4,6$ characters. (In \cite[Theorem 6.2]{MalleBrauer} the centralizer of type $E_7 A_1$ was erroneously omitted).

	Let us first consider the six unipotent characters of the block $b$. There are four unipotent characters whose $2$-part is smaller than the $2$-part of the other two unipotent characters. The $2$-part of these four characters is precisely the $2$-part of $|G:P|$ and these are thus of height zero.
	
	
	Let's look at the unique $2$-element $t \in G^\ast$ (up to conjugation) whose centralizer is of type $E_7 A_1$.
	The $2$-part of the character degrees of the four characters in the Harish-Chandra series lying above the cuspidal character is again $|G:P|_2$ and thus these are of height zero again.
	
	A similar calculation shows on the other hand that characters associated with centralizers of type $E_7$, $E_6$ or $E_6 A_1$ cannot contribute to height zero characters of $b$. 
	
	Suppose now that we are in case (iii). Here, the calculation is slightly more complicated as no description of the characters in $\mathcal{E}_2(G,s)$ is known. Using the description of defect groups given in the proof of \cite[Proposition 6.4]{KessarMalle}, we deduce that $P \cong C_{(q-1)_2} \wr C_2$ by Lemma \ref{wreath} in the first case and $P$ is metacyclic in the second case. In particular, Olsson's conjecture is known for these cases; i.e. $|\Irr_0(b)| \leq |P:P'|$, see \cite[Theorem 10.22]{Sambale} respectively. We show that there are at least $|P:P'|$ characters of height zero lying in the block $b$, which then proves that these are all height zero characters.

	We assume $e=1$ and consider first case (4) in \cite{KessarMalle}. We first count the height zero characters in the unipotent block $c_0$ associated to $b$ through Lemma \ref{block bijection}. The group $G(s)$ is of type $E_6(q) A_2(q)$ and the block $c_0$ is by \cite[Theorem 17.7]{MarcBook} isomorphic to a block of $E_6(q)_{\mathrm{ad}} \times A_2(q)_{\mathrm{ad}}$. More precisely, it is isomorphic to the tensor product of the unique unipotent block of $E_6(q)$ of trivial defect containing the cuspidal character $E_6[\theta]$ with the principal block of $A_2(q)$. Now, the principal block of $A_2(q)$ has exactly $|P:P'|$ height zero characters in this case as can be seen easily from \cite[Lemma 3.8]{BroughRuhstorfer}.
	Recall that the bijection $\mathcal{J}$ is defined as the union of the bijections $J_{G(s),st}^{-1} \circ J_{G,st}:  \mathcal{E}(G,st) \to  \mathcal{E}(G(s),st)$ for $t \in \C_{G^\ast}(s)_2$. We now want to show that $\hat{s}^{-1}\mathcal{J}$ induces a bijection betwen the height zero characters of $b$ and its corresponding unipotent block $c_0=\hat{s}^{-1} c$ of $G(s)$. 
	
	For this let's first consider the unipotent characters. We observe that $L^\ast=\C_{L^\ast}(s)$ by \cite[Table 5]{KessarMalle} and we let $\lambda_s=\lambda \hat{s}^{-1}=E_6[\theta] \in \mathcal{E}(\C_{L^\ast}(s),1)$ be the Jordan correspondent of $\lambda$. Note that every character in the principal block of $A_2(q)$ lies in the principal Harish-Chandra series by Lemma \ref{typeA}. Since the block $c_0$ is the tensor product of a cuspidal character and the principal block of $A_2(q)$, it follows that every character in $\Irr_0(c_0) \cap \mathcal{E}(\C_{G^\ast}(s),1)$ appears as an irreducible constituent of $R_{\C_{\Levi^\ast}(s)}^{\C_{\G^\ast}(s)}(\lambda_s)$. By \ref{DigneMichel}(5), we have $J_{G,s} R_L^G(\lambda)=R_{\C_{\Levi^\ast}(s)}^{\C_{\G^\ast}(s)}(J_{L,s}(\lambda))$. In other words, the Harish-Chandra series of $\lambda_s$ is mapped via Jordan decomposition to the Harish-Chandra series of $\lambda$.
	
	Again by the same arguments as above, the characters in $\Irr_0(c_0) \cap \mathcal{E}(\C_{G^\ast}(s),t)$ appear as irreducible constituents of $R_{\C_{\Levi^\ast}(s)}^{\C_{\G^\ast}(s)}(\lambda_{st} )$, where $\lambda_{st}:=\lambda_s \hat{t}$ with $t \in \Z(\C_{\Levi^\ast}(s))_2=\Z(L^\ast)_2$.
Moreover, $\C_{L^\ast}(s)=E_6(q) \Phi_1^2$ is contained in $\C_{G^\ast}(st)$ which is a proper subgroup of $\C_{G^\ast}(s)=E_6(q) A_2(q)$. Comparing this with the list of possible centralizers in $E_8(q)$ given in \cite[Table 1]{Deriziotis}, we obtain that $\C_{\G^\ast}(st)$ is necessarily a Levi subgroup of $\G^\ast$. In particular, Jordan decomposition is given by Deligne--Lusztig induction, i.e. $J_{G,st}^{-1}=R_{\C_{G^\ast}(st)}^{G^\ast} \widehat{st}$ and $J_{L,st}(\lambda \hat{t})=\lambda_s$ since $\C_{L^\ast}(s)=\C_{L^\ast}(st)$. From this, we deduce by transitivity of Deligne--Lusztig induction that
	$J_{st}^G R_L^G(\lambda \hat{t})=R_{\C_{\Levi^\ast}(st)}^{\C_{\G^\ast}(st)}(J^L_{st}(\lambda \hat{t}))$. Since $\C_{\G^\ast}(st)$ is a Levi subgroup of $\C_{\G^\ast}(s)$, we also get $J_{G(s),st}^{-1}=R^{\C_{\G^\ast}(s)}_{\C_{\G^\ast}(st)} \widehat{st}$. We claim that the characters $R_L^G(\lambda \hat{t})$ with $t \in \Z(L^\ast)_2$ all lie in the same $2$-block. For this observe that $d^1(R_L^G(\lambda \hat{t}))=d_1(R_L^G(\lambda))$, by \cite[Proposition 9.6(iii)]{MarcBook}. As all constituents of $R_L^G(\lambda)$ lie in $b_G(\Levi,\lambda)$ the same is therefore true for the constituents of $R_L^G(\lambda \hat{t})$. This implies that $\mathcal{J}^{-1}(\Irr_0(c)) \subset \Irr_0(b)$ and so in particular $|P:P'|=|\Irr_0(c_0)| \leq |\Irr_0(b)|\leq |P:P'|$. Hence, we must have equality everywhere and so $\mathcal{J}^{-1}(\Irr_0(c))= \Irr_0(b).$
	For case (6) of \cite{KessarMalle} we can use the same arguments as in case (4) but instead of using the Levi subgroup of case (6) we use the alternative $2$-split Levi subgroup (6b) in \cite[Table 5]{KessarMalle}.
	
	Again the same considerations apply by Ennola duality to the case $e=2$.
\end{proof}

\section{Equivariant Jordan decomposition for blocks with cyclic stabilizer in $\mathrm{Out}(G)$}\label{section 8}

\subsection{Blocks with cyclic stabilizer}

Recall that $\G$ is a simple, simply connected algebraic group of exceptional type with Frobenius endomorphism $F: \G \to \G$. For the following lemma recall that every character $\chi \in \mathcal{E}(G,t)$ for a semisimple element $t \in G^\ast$ has the same restriction to $\Z(G)$, see \cite[Equation 9.12]{Bonnafe2}. Moreover since $\Z(G)$ is cyclic of prime order, the center $\Z(G)$ is in the kernel of $\chi$ if and only if $t \in [G^\ast,G^\ast]$, see \cite[Section 3.2]{MalleNavarro}. For the splitting of the geometric classes we use \cite{Luebeck}.

\begin{lemma}\label{block stabilizer}
	Let $b$ be an isolated, non-unipotent $2$-block of $G$. Then $\mathrm{Out}(G)_b$ is cyclic except possibly in the following cases:
	´\begin{enumerate}[label=(\roman*)]
		\item Case (1) and (7) in Table \ref{table} in type $E_6(q)$.
		\item The Ennola duals in ${}^2 E_6(q)$ of cases (1) and (7).
		\item Case (1) in \cite[Table 4]{KessarMalle} when $G=E_7(q)$ and $q \equiv 1  \mod 12$.
	\end{enumerate}
	Moreover, in these cases we have $\mathrm{Out}(G)_b = \mathrm{Out}(G).$
\end{lemma}

\begin{proof}
	Let us first assume that $G=E_6(q)$.
	Case (1) (resp. (7)) occurs when $q \equiv 1 \mod 3$. Furthermore, in this case the semisimple element $s$ is the unique element in its rational class with this precise centralizer. Thus, its rational conjugacy class is stable under all automorphisms and so is the unique block $b$ associated to $s$.
	The two rational conjugacy classes associated to case (2) (resp. (8)) (still for $q \equiv 1 \mod 3$) are geometrically conjugate to the semisimple element of case (1). These elements don't lie in $[G^\ast, G^\ast]$ and therefore the characters in the corresponding block are non-trivial on the center. Therefore, the graph automorphism must necessarily interchange their rational conjugacy classes. Moreover, there are three blocks associated to each of the two rational conjugacy classes. Since these blocks are conjugate under diagonal automorphisms none of them is fixed under a diagonal automorphism. In particular, no combination of a graph and diagonal automorphism stabilizes a block. Since $\mathrm{Out}(G) \cong S_3 \times C_a$ for some $a$ we can use \cite[Lemma 2.4]{Niamh} to conclude that $\mathrm{Out}(G)_b$ is cyclic. 
	Case (6) (resp. (12)) occurs when $q \equiv 2 \mod 3$, i.e. when $\Z(G)=1$. Moreover, the group of field automorphisms has odd order since $q$ cannot be a square. Therefore, $\mathrm{Out}(G)$ is cyclic in this case.

	The case $G={}^2 E_6(q)$ follows by Ennola duality: Here $\mathrm{Out}(G)$ is cyclic unless $q \equiv 2 \mod 3$. If $q \equiv 2 \mod  3$ then $\mathrm{Out}(G) \cong S_3 \times C_a$ with odd $a$. Therefore, the same arguments from before apply as well. This leaves a rational conjugacy class with centralizer $({}^2 A_2(q))^3.3$, i.e. the Ennola dual of cases (1) resp. (7).
	
	Let us assume that $G=E_7(q)$. Case (1) appears whenever $q \equiv 1 \mod 3$ and case (2) whenever $q \equiv 2 \mod 3$. Thus, in case (2) and whenever $q \equiv 3 \mod 4$ the group of field automorphisms has odd order again. This leave case (1) with $q \equiv 1 \mod 12$.
	
	When $\G$ is not of type $E_6$ or $E_7$ observe that $\mathrm{Out}(G)$ is cyclic and so there is nothing to show.
\end{proof}

We also discuss the blocks of $G=E_6(q)$ which are strictly quasi-isolated but not isolated. For this recall that the semisimple element $s \in (\G^\ast)^F$ is called strictly quasi-isolated if $\C^\circ_{\G^\ast}(s) \C_{\G^\ast}(s)^F$ is not contained in a proper Levi subgroup of $\G^\ast$. 

\begin{lemma}\label{block stabilizer2}
	Let $b$ be a strictly-quasi isolated but not isolated block of $G=E_6(q)$. Then $\mathrm{Out}(G)_b$ is cyclic unless we are in case (3) or (9) in Table \ref{table} (or its Ennola dual for $G={}^{2} E_6(q)$).
\end{lemma}

\begin{proof}
	We need to consider the geometric conjugacy class of the quasi-isolated element in $\G=E_6$ whose centralizer is of type $D_4.3$. For $q \equiv 1 \mod 3$ this geometric conjugacy class splits into one rational class with centralizer $D_4(q) \Phi_1^2 .3$ (case (3) resp. case (9) in Table \ref{table}) and two rational classes with centralizer $({}^3D_4(q) \Phi_3).3$ (case (5) resp. case (10) in Table \ref{table}). In the latter case, there are three blocks of $G(s)$ covering the block of $G^\circ(s)$. Furthermore, the central characters of these blocks are non-trivial.  Again, we deduce that no combination of a graph and diagonal automorphism can fix such a block, i.e. the block stabilizer is cyclic again.
	
	On the other hand, for $q \equiv 2 \mod 3$ the geometric conjugacy class doesn't split and a rational representative of it has centralizer ${}^2 D_4(q) \Phi_1 \Phi_2$ (case (4) resp. (11) in Table \ref{table}). In this case the element is quasi-isolated but not strictly quasi-isolated, as $\C_{(\G^\ast)^{F}}(s)$ is the set of fixed points of a Levi subgroup of $\G^\ast$. 
	
	Again, the case $G={}^2 E_6(q)$ can be obtained by Ennola duality.
\end{proof}

\subsection{Equivariant Jordan decomposition}

Recall that we have a surjective morphism $\pi: \G \to \G^\ast$. Therefore, for any bijective morphism $\phi: \G \to \G$ commuting with the action of $F$, we can define a dual morphism $\phi^\ast: \G^\ast \to \G^\ast$ as the unique morphism which satisfies $\pi \circ \phi=\phi^\ast \circ \pi$.

Now assume that $\phi$ stabilizes the idempotent $e_s^{\G^F}$. Then its dual morphism fixes the $(\G^\ast)^F$-conjugacy class of the semisimple element $s$, see \cite[Corollary 2.4]{Navarro}. Thus, by possibly changing $\phi$ by an inner automorphism of $\G^F$ we can assume that its dual $\phi^\ast$ stabilizes the semisimple element $s$. Hence, $\phi^\ast$ stabilizes $\mathrm{C}_{\G^\ast}(s)$. By construction the group $\G(s)=\pi^{-1}(\mathrm{C}_{\G^\ast}(s))$ and the idempotent $e_s^{G(s)}$ are therefore $\phi$-stable.
Recall that $\mathcal{B} \subset \mathrm{Aut}(\tilde{G})$ denotes the subgroup generated by field and graph automorphisms as in \cite[Section 2.B]{MS}.

Using this we can show the following:

\begin{proposition}\label{equivariant Jordan}
	Let $b$ be a quasi-isolated block of $G$ associated to the semisimple element $s \in G^\ast$.
	Assume that $\mathrm{Out}(G)_b$ is cyclic and $\mathrm{C}_{\G^\ast}(s)^F=\mathrm{C}^\circ_{\G^\ast}(s)^F$. Then there exists an $\N_{\tilde{G}\mathcal{B}}(G(s),s)$-equivariant bijection $\mathcal{E}_2(G,s) \to \mathcal{E}_2(G(s),s)$ which induces a bijection $\Irr_0(G,b) \to \Irr_0(G(s),c) $.
\end{proposition}

\begin{proof}
	Recall that by Lemma \ref{jordan construct} we have a $\tilde{G}(\tilde{s})$-equivariant bijection $\mathcal{J}: \mathcal{E}_2(G,s) \to \mathcal{E}_2(G(s),s)$. By the degree properties of Jordan decomposition, the height zero characters of a maximal block are mapped to the height zero characters of a maximal block. By Lemma \ref{block bijection} and its proof, the maximal blocks of $\mathcal{E}_2(G,s)$ form one $\tilde{G}(\tilde{s})$-orbit and the maximal blocks of $\mathcal{E}_2(G(s),s)$ form a $\tilde{G}(\tilde{s})$-orbit as well. Moreover, by Lemma \ref{nonmaximal}, the height zero characters of a non-maximal block of $\mathcal{E}_2(G,s)$ are always mapped to height zero characters of the non-maximal block of $\mathcal{E}_2(G(s),s)$. From these considerations it's clear that we can choose $\mathcal{J}$ with the additional property that  $\mathcal{J}(\Irr_0(b))=\Irr_0(c)$, where $c$ corresponds to $b$ under the bijection from Lemma \ref{block bijection}.
	
	We are left to show that the so-obtained bijection is equivariant on the blocks with cyclic outer automorphism group. For this let $\phi: \G \to \G$ be a morphism with $\phi^\ast(s)=s$ such that $\phi$ together with $\tilde{G}_b$ generates $\mathrm{Out}(G)_b$ and the order of $\phi$ is coprime to the order of $\tilde{G}_b/G \Z(\tilde{G}) \cong \tilde{G}(\tilde{s})_c/G(s) \Z(\tilde{G})$. We fix a $\tilde{G}_b \langle \phi \rangle$-transversal $\mathcal{T}$ of $\Irr_0(b)$ and wish to define a bijection $\mathcal{J}':\Irr_0(b) \to \Irr_0(c)$ by defining $\mathcal{J}'$ as $\mathcal{J}$ on $\mathcal{T}$ and extending it $\tilde{G}_b \langle \phi \rangle$-equivariantly. To show that the so-obtained map is a well-defined bijection it suffices to show that $\chi$ and $\mathcal{J}(\chi)$ have the same stabilizer in $\tilde{G}_b \langle \phi \rangle$. By construction of the map $\mathcal{J}$, for $\chi \in \Irr_0(b)$ and $\tau \in \langle \phi \rangle_\chi$ it follows that $\tau(\mathcal{J}(\chi))$ and $\mathcal{J}(\tau(\chi))$ are in the same $\tilde{G}(\tilde{s})_c$-orbit. However since the order of $\phi$ is coprime to the order of $\tilde{G}(\tilde{s})_c/G(s) \Z(\tilde{G})$ it follows that necessarily $\tau(\mathcal{J}(\chi))=\mathcal{J}(\tau(\chi))$. 
\end{proof}

This deals with all cases where the block stabilizer is cyclic except for case (2) and (8) in type $E_6$ and their Ennola duals. 

\begin{lemma}\label{isomorphic blocks}
	Suppose that we are either in case (2) or (8) of type $E_6$.
	Let $c_0$ be the unique block associated to the union $\mathcal{E}_2(G^\circ(s),s)$ of Lusztig series. Then there exists an $\N_{\tilde{G}\mathcal{B}}(G(s))_b$-equivariant bijection $\Irr_0(G^\circ(s),c_0) \to \Irr_0(G,b)$.  
\end{lemma}

\begin{proof}
	Using Theorem \ref{equivariant sp} and the proof of Lemma \ref{jordan construct}, we find that the Jordan decomposition $\tilde{\mathcal{J}}: \mathcal{E}_2(\tilde{G},\tilde{s}) \to \mathcal{E}_2(\tilde{G}(\tilde{s}), \tilde{s})$ is $\N_{\tilde{G}\mathcal{B}}(G(s),s) \times \Irr(\tilde{G}/G)_2$-equivariant.
	
	Our assumption implies by Lemma \ref{block bijection} that every $\tilde{G}$-orbit of a character in $\mathcal{E}_2(G,s)$ has size three and there exists a unique character in this orbit which lies in the block $b$.
	It follows that the map $b \Res_{G}^{\tilde{G}}:\mathcal{E}_2(\tilde{G},\tilde{s})/\Irr(\tilde{G}/G)_2 \to \Irr(G,b)$ is a bijection. It follows from the proof of Lemma \ref{block bijection} that the map $\Res_{G^\circ(s)}^{\tilde{G}(\tilde{s})}:\mathcal{E}_2(\tilde{G}(\tilde{s}),\tilde{s})/ \Irr(\tilde{G}(\tilde{s})/G^\circ(s))_2 \to \Irr(G^\circ(s),c_0)$ is a bijection as well. The composition of these bijections therefore yields an $\mathrm{N}_{\tilde{G} \mathcal{B}}(G(s))_b$-equivariant bijection as claimed.
\end{proof}

\subsection{Central characters}

Let $\bH$ be a semisimple algebraic group with Frobenius endomorphism $F: \bH \to \bH$ definining an $\mathbb{F}_q$-structure. Recall that by \cite[Proposition 9.15]{MT}  there exist natural isogenies $\tau: \bH_{\ssc} \to \bH$ and $\pi: \bH \to \bH_{\ad}$.

\begin{lemma}\label{im center}
	Assume that $\bH$ is semisimple and $\bH_{\ssc}^F$ is perfect and  suppose that $\tau(\Z(\bH_{\ssc})^F) = \Z(\bH)^F$. Then $\Z(\bH)^F \subset [\bH^F,\bH^F]$.
\end{lemma}

\begin{proof}
	By  \cite[Proposition 24.21]{MT} (whose assumption can be weakened to $\bH$ semisimple) $\tau(\bH_{\ssc}^F)=[\bH^F,\bH^F]$. Therefore, $\Z(\bH)^F=\tau(\Z(\bH_{\ssc})^F) \subseteq \tau(\bH_{\ssc}^F)=[\bH^F,\bH^F]$.
\end{proof}

The assumption $\tau(\Z(\bH_{\ssc})^F) = \Z(\bH)^F$ on the center in the previous lemma is for instance satisfied if $\bH=\bH_{\ssc},\bH_{\ad}$ or $\Z(\bH_{\ssc})^F=\Z(\bH_{\ssc})$.

The following lemma can be compared to \cite[Proposition 4.8]{MalleUnip}:

\begin{lemma}\label{center kernel}
	Let $H$ be as in the previous lemma and $\ell \neq p$ be arbitrary. Then the characters of $\Irr_0(B_0(H))$ have $\Z(H)$ in their kernel.
	
	Moreover, there exists a bijection between the set of $H^1(F,\mathrm{Z}(\bH))_\ell$-stable characters of $\Irr_{0}(B_0(H))$ and orbits $\Irr_{0}(B_0(H_{\ad}))/ \Irr(H_{\ad}/\pi(H))_\ell$. 
	
\end{lemma}

\begin{proof}
	Let $\chi \in \Irr_{0}(B_0(H))$ and $\lambda=\mathrm{det}(\chi)$ its associated determinantal character. Since $\lambda$ is linear it follows from Lemma \ref{im center} that $\lambda$ has $\Z(H)$ in its kernel. For $z \in \Z(H)_\ell$ there exists an $\ell$-power root of unity $\zeta$ such that $\chi(z)=\chi(1) \zeta$ and $1=\lambda(z)=\zeta^{\chi(1)}$. Since $\chi$ is of $\ell'$-degree, it follows that $\zeta=1$ and we conclude that $\Z(H)_\ell$ is in the kernel of $\chi$. Moreover, $\chi$ has $\Z(H)_{\ell'}$ in its kernel since it lies in the principal block.
	Hence, any $\ell'$-character of the principal block of $H$ is trivial on the center of $H$.
	
	Now we proceed as in the proof of \cite[Section  17]{MarcBook}. Take $\iota: \bH \to \tilde{\bH}$ a regular embedding and $\pi: \tilde{\bH} \to \bH_{\ad}$ the quotient map.
	We define a map from the set of $H^1(F,\mathrm{Z}(\bH))_\ell$-invariant characters of $\Irr_0(B_0(H))$ to $\Irr_0(B_0(H_{\mathrm{ad}}))$ as follows. Any $H^1(F,\mathrm{Z}(\bH))_\ell$-stable character $\chi \in \Irr_{0}(B_0(H))$ extends to an $\ell'$-character $\tilde{\chi}$ of $\tilde{H}$. This character can be assumed to have $\Z(\tilde H)$ in its kernel and to lie in the principal block of $\tilde H$ by \cite[Theorem 15.11]{MarcBook}. The character $\tilde{\chi}$ considered as a character of $\tilde{H}/\Z(\tilde{H})$ lies again in the principal block of $H_{\ad}$. Note that if we were to choose a different character $\tilde{\chi}'$ with the same properties as $\tilde{\chi}$, then $\tilde{\chi}'=\lambda \tilde{\chi}$ for some $\lambda \in \Irr(\tilde{H}/H \mathrm{Z}(\tilde{H}))$ of $\ell$-power order. As $H_{\ad}/\pi(H) \cong  \tilde{H}/H \mathrm{Z}(\tilde{H})$ by \cite[Equation 15.4]{MarcBook} it follows that the images of $\tilde{\chi}'$ and $\tilde{\chi}$ in $H_{\ad}$ are in the same $\Irr(H_{\ad}/\pi(H))_\ell$-orbit. The bijectivity of the map follows easily from this.
\end{proof}

In particular, if $\ell \nmid |H^1(F,\Z(\bH))|$ then the sets $\Irr_0(B_0(H))$ and $\Irr_0(B_0(H_{\ad}))$ are in bijection. We obtain the following consequence for our situation:

\begin{lemma}\label{trivial on center}
	As before, let $G$ be an exceptional group of Lie type and $b$ be a quasi-isolated $2$-block of $G$. Then $\Z(G)_2 \leq \mathrm{Ker}(\chi)$ for every height zero character $\chi \in \Irr_0(G,b)$.
\end{lemma}

\begin{proof}
	We can assume that $G=E_7(q)$, since $\Z(G)_2=1$ otherwise. Let $s \in G^\ast$ be the quasi-isolated element associated to the block $b$.
By Lemma \ref{central character} it suffices to show that every height zero character in the principal block $B_0(G(s))$ has $\Z(G)_{2}$ in its kernel. Let $\tau:\G(s)_{\ssc} \to \G(s)$ be the simply connected covering of the connected reductive group $\G(s)$. An inspection of \cite[Table 4]{KessarMalle} shows that $\Z(\G(s)_{\ssc})=\Z(\G(s)_{\ssc})^F$. Therefore, by Lemma \ref{center kernel} every $2'$-character of $G(s)$ has $\Z(G(s))$ in its kernel.
\end{proof}

%
	%

\section{The inductive AM-condition}\label{section 9}

\subsection{A criterion for the iAM-condition}

We recall the definition of AM-good, see also \cite[Definition 1.9]{Jordan2}. For the precise definition of the order relation "$\geq_b$" (which we will not need in the following) see also \cite[Section 1]{Jordan2}.

\begin{definition}\label{IAMsuitable}
	Let $H$ be a finite group and $b$ a block of $H$ with non-central defect group $D$. Assume that for $\Gamma:= \mathrm{N}_{\mathrm{Aut}(H)}(D,b)$ there exist
	\begin{enumerate}[label=(\roman*)]
		\item a $ \Gamma$-stable subgroup $M$ with $\mathrm{N}_H(D) \leq M \lneq H$;
		\item a $\Gamma$-equivariant bijection $\Psi: \Irr_0(H , b) \to \Irr_0(M , B)$ where $B \in \mathrm{Bl}(M \mid D)$ is the unique block with $B^H=b$; 
		\item $$( H/ Z \rtimes \Gamma_\chi, H/Z, \overline{\chi}) \geq_{b} ( M/Z \rtimes \Gamma_\chi, M/Z, \overline{\Psi(\chi)} ),$$
		for every $ \chi \in \Irr_0(H,b)$ and $Z= \mathrm{Ker}(\chi) \cap \mathrm{Z}(H)$, where $\overline{\chi}$ and $\overline{\Psi(\chi)}$ lift to $\chi$ and $\Psi(\chi)$, respectively.
	\end{enumerate}
	Then we say that $b$ is AM-good or satisfies the iAM-condition (with respect to the subgroup $M$). Moreover, we say that $\Psi: \Irr_0(H , b) \to \Irr_0(M , B)$ is a \textit{strong iAM-bijection for the block} $b$ with respect to the subgroup $M$.
\end{definition}

\begin{remark}
If $\Psi$ is a strong iAM-bijection as in the definition above, then we always have (see \cite[Lemma 3.12]{JEMS})
$$( H \rtimes \Gamma_\chi, H , \chi) \geq_{b} ( M \rtimes \Gamma_\chi, M, \Psi(\chi) ).$$
If $\Psi$ satisfies only this weaker condition, we say that $\Psi$ is an \textit{iAM-bijection}. 
\end{remark}

\subsection{Groups of Lie type}

We now assume again that $G$ is an exceptional group of Lie type as in Section \ref{section 2}.
We recall the following theorem:

\begin{theorem}\label{12}
Let $b$ be a block of $G$ with defect group $P$
and $M$ a proper $\mathrm{Aut}(G)_b$-stable subgroup of $G$ containing $\N_G(P)$. In addition, denote by $B$ the unique block of $M$ with $B^G=b$ and we set $\tilde{M}:=M \N_{\tilde{G}}(P)$.
	Let $ \chi \in \Irr(G,b)$ and $\chi' \in \Irr(M,B)$ such that the following holds:
	\begin{enumerate}[label=(\roman*)]
		\item We have $(\tilde{G} \mathcal{B})_\chi = \tilde{G}_\chi \mathcal{B}_\chi$ and 
		$\chi$ extends to $G \mathcal{B}_\chi$.
		\item
		We have $( \N_{\tilde{G}}(M)  \mathrm{N}_{G \mathcal{B}}(M) )_{\chi'}= \N_{\tilde{G}}(M)_{\chi'} \mathrm{N}_{G \mathcal{B}}(M)_{\chi'}$ and $\chi'$ extends to $ \mathrm{N}_{G \mathcal{B}}(M)_{\chi'}$.
		\item $(\tilde{G} \mathcal{B})_\chi = G ( \N_{\tilde{G}}(M)  \mathrm{N}_{G \mathcal{B}}(M) )_{\chi'}$.
		\item There exists $\tilde{\chi} \in \Irr(\tilde{G} \mid \chi)$ and $\tilde{\chi}' \in \Irr(\N_{\tilde{G}}(M) \mid \chi')$ such that the following holds:
		\begin{itemize}

			\item  For all $m \in {\mathrm{N}_{G \mathcal{B}}(M)}_{\chi'}$ there exists $\nu \in \Irr(\tilde{G} /G)$ with $\tilde{\chi}^m= \nu \tilde{\chi}$ and $\tilde{\chi}'^m=\mathrm{Res}^{\tilde{G}}_{\N_{\tilde{G}}(M)}(\nu) \tilde{\chi}'$.
			\item The characters $\tilde{\chi}$ and $\tilde{\chi}'$ cover the same underlying central character of $\mathrm{Z}(\tilde{G})$.
		\end{itemize}
		\item The Clifford correspondents $\tilde{\chi}_0 \in \Irr(\tilde{G}_\chi \mid \chi)$ and $\tilde{\chi}'_0 \in \Irr(\N_{\tilde{G}}(M)_{\chi'} \mid \chi')$ of $\tilde{\chi}$ and $\tilde{\chi}'$ respectively satisfy $\mathrm{bl}(\tilde{\chi}_0)= \mathrm{bl}(\tilde{\chi}_0')^{\tilde{G}_\chi}$.
	\end{enumerate}
	Let $\mathrm{Z}:=  \mathrm{Ker}(\chi) \cap \mathrm{Z}(G) $.
	Then
	$$(( \tilde{G} \mathcal{B})_\chi /Z, G/Z, \overline{\chi}) \geq_b ((\N_{\tilde{G}}(M) \mathrm{N}_{G \mathcal{B}}(M))_{\chi'} / Z,M/Z, \overline{\chi'}),$$
	where $\overline{\chi}$ and $\overline{\chi'}$ are the characters which inflate to $\chi$, respectively $\chi'$.
\end{theorem}

\begin{proof}
	This is a consequence of \cite[Theorem 2.1]{Jordan2} and \cite[Lemma 2.2]{Jordan2}. 
\end{proof}

Assumption (i) can essentially always be satisfied:

\begin{theorem}\label{star}
	In every $\tilde{G}$-orbit of $\Irr(G)$ there exists a character $\chi$ satisfying assumption (i) of Theorem \ref{12}.
\end{theorem}

\begin{proof}
	 See \cite[Theorem B]{TypeB}.
\end{proof}

\begin{lemma}\label{block theory}
	Let $ \chi \in \Irr(G,b)$.
	Then block induction yields a bijective map $\mathrm{Bl}(\tilde{G}_\chi \mid b) \to \mathrm{Bl}(\tilde{G} \mid b)$.
\end{lemma}

\begin{proof}
	 By \cite[Corollary 6.2]{NavarroBook} we deduce that block induction yields a surjective map $\mathrm{Bl}(\tilde{G}_\chi \mid b) \to \mathrm{Bl}(\tilde{G} \mid b)$. We therefore only need to show injectivity.
	
	If $\tilde{G}_b=\tilde{G}_\chi$, then this follows from \cite[Theorem 9.14]{NavarroBook}. Otherwise, since $\tilde{G}/ G \Z(\tilde{G})$ has prime order (or order $1$) in our situation we have $\tilde{G}_b=\tilde{G}$ and $\tilde{G}_\chi=G \Z(\tilde{G})$. Since $\tilde{G}_\chi=G \Z(\tilde{G})$ is a central product, we deduce that $\mathrm{Bl}(\tilde{G}_\chi \mid b)$ is in bijection with $(\Z(\tilde{G})/\Z(G))_{2'}$. On the other hand, by \cite[Corollary 2.8, Corollary 2.9]{Bonnafe} we have $|\mathrm{Bl}(\tilde{G} \mid b)| \geq |(\Z(\tilde{G})/\Z(G))|_{2'}$. Hence, the map is necessarily bijective.
\end{proof}

As a consequence of the previous lemma, we can in Theorem \ref{12} replace condition (v) by the requirement that $\mathrm{bl}(\tilde{\chi}')^{\tilde{G}}=\mathrm{bl}(\tilde{\chi})$.


%
%
\subsection{A criterion for blocks with cyclic stabilizer}

In cases where we have additional information about the block the criterion of Theorem \ref{12} simplifies.

\begin{lemma}\label{cyclic AM}
	Let $b$ be a strictly quasi-isolated $2$-block of $G$ such that $\mathrm{Out}(G)_b$ is cyclic and let $D$ be a characteristic subgroup of a defect group $P$ of $b$. Suppose that $f: \Irr_0(G,b) \to \Irr_0(\N_G(D),B)$ is an $\N_{\tilde{G} \mathcal{B}}(D,B)$-equivariant bijection which preserves central characters. Assume that one of the following holds:
	\begin{enumerate}[label=(\alph*)]
		\item $\tilde{G}_b=G \Z(\tilde{G})$.
		\item $\Z(G)$ is a $2$-group and every character of $\Irr_0(G,b)$ has $\Z(G)$ in its kernel.
	\end{enumerate}
Then the block $b$ is AM-good with respect to the subgroup $\N_G(D)$.
\end{lemma}

\begin{proof}
	We show that $\chi \in \Irr_0(G,b)$ and $\chi':=f(\chi)$ satisfy the assumptions of Theorem \ref{12}. Firstly, by Theorem \ref{star} we can assume by possibly replacing $\chi$ by a $\tilde{G}$-conjugate of it that $\chi$ and $\chi'$ satisfy assumptions (i)-(iii) of Theorem \ref{12}.
	
	Let's suppose first that (a) holds. By assumption $\chi$ and $\chi'$ lie over the same character of $\Z(G)$. Observe that $G \Z(\tilde{G})$ is a central product of $G$ and $\Z(\tilde{G})$. Hence, there exist extensions $\tilde{\chi}_0$ and $\tilde{\chi}_0'$ which lie over the same character of $\Z(\tilde{G})$.  Therefore, we have $\mathrm{bl}(\tilde{\chi}_0)= \mathrm{bl}(\tilde{\chi}_0')^{\Z(\tilde{G}) G}$. Hence, condition (v).
	Since the extension of $\chi$ to $G \Z(\tilde{G})$ is uniquely determined by its values on $\Z(\tilde{G})$ it follows that the characters $\tilde{\chi}$ and $\tilde{\chi}'$ satisfy assumption (iv). 
	
	Assume now that (b) holds. 	We can consider $\chi$ as a character of the simple group $S=G/\Z(G)$. As $\mathrm{Out}(G)_b$ is cyclic, the character $\chi$ extends to a character $\hat{\chi}$ of $(\tilde{G} \mathcal{B})_\chi$ with $\Z(\tilde{G})$ in its kernel. Denote by $\tilde{\chi}_0$ the restriction of $\hat{\chi}$ to $\tilde{G}_\chi$. Let $\overline{b}$ be the unique block of $G/\Z(G)$ which is dominated by $b$, see \cite[Theorem 7.6]{NavarroBook}. By assumption $G/\Z(G)$ is a normal subgroup of $2$-power index in $\tilde{G}/\Z(\tilde{G})$. Hence, there exists a unique block $\overline{\tilde{b}}$ of $\tilde{G}/\Z(\tilde{G})$ which covers $\overline{b}$, see \cite[Corollary 9.6]{NavarroBook}. Thus, the character $\tilde{\chi}$ necessarily lies in the block $\tilde{b}$ dominating $\overline{\tilde{b}}$.	We observe that the same arguments apply to the local character $\chi'$. 
	
	Now assumption (iv) in Theorem \ref{12} follows from observing that $\tilde{\chi}$ and $\tilde{\chi}'$ are both $\N_{G \mathcal{B}}(D)_{\chi'}$-stable as $\tilde{\chi}$ extends to the character $\hat{\chi} \in \Irr((\tilde{G} \mathcal{B})_{\chi})$ and similarly $\chi'$ extends to $\N_{(\tilde{G} \mathcal{B})_\chi}(D)$. For assumption (v) observe that $\tilde{\chi}$ and $\tilde{\chi}'$ lie in the unique block covering $b$ resp. $B_D$ whose characters have $\Z(\tilde{G})_{2'}$ in their kernel. 
\end{proof}

\section{Blocks with non cyclic-stablizer}\label{section 10}

\subsection{Quasi-isolated blocks with non-cyclic stabilizer}

As before, let $b$ be a $2$-block associated to a quasi-isolated semisimple element $s$. In this section, we assume that $\mathrm{Out}(G)_b$ is non-cyclic. In particular, this means by Lemma \ref{block stabilizer} and Lemma \ref{block stabilizer2} that we are in one of the following cases, where $\varepsilon \in \{ \pm 1\}$:

\begin{itemize}
	\item $G=E_6(\varepsilon q)$ and $s$ has centralizer $A_2(\varepsilon q)^3.3$
	\item $G=E_6(\varepsilon q)$ and $s$ has centralizer $D_4( \varepsilon q).3$
	\item $G=E_7$ with $e=1$ and $s$ with centralizer $A_5(q) A_2(q)$.
\end{itemize}
We make the following useful observation:

\begin{corollary}\label{observation}
	Let $b$ be a quasi-isolated block of $G$ associated to the semisimple element $s$ such that
	$\mathrm{Out}(G)_b$ is non-cyclic. Then the following hold:
	\begin{enumerate}[label=(\alph*)]
		\item We have $b=e_s^{\G^F}$ and $\lambda \in \mathcal{E}(L,s)$ is a character of central $2$-defect with $b_L(\lambda)=e_s^{\Levi^F}$.
		\item Assume that $s$ is not a semisimple element in $E_6$ with centralizer of type $A_2^3$ and let $\Levi_0$ be the $1$-split Levi subgroup associated to $s$. Then we have $\Levi^\ast=\C_{\Levi^\ast}(s)$ and $\Levi_0^\ast=\C_{\Levi_0^\ast}(s)$. In particular, $\Levi^\ast$ (resp. $\Levi_0^\ast$) is the centralizer of a Sylow $e$-torus (resp. Sylow $1$-torus) of $(\G^\ast,F)$.  
	\end{enumerate}
	\end{corollary}
	
	\begin{proof}
	 That $b=e_s^{\G^F}$ follows from Lemma \ref{block bijection}. Moreover, $\lambda$ is of central defect by using \cite[Proposition 2.7(g)]{KessarMalle} and consulting the relevant tables in \cite{KessarMalle}. Moreover, we see from these tables that $\C_{\Levi^\ast}(s)=\C^\circ_{\Levi^\ast}(s)$ is a torus and hence $e_s^{\Levi^F}$ is a $2$-block which yields $b_L(\lambda)=e_s^{\Levi^F}$. Part (b) follows again by an inspection of the relevant tables in \cite{KessarMalle}.
	\end{proof}

We let $B$ be the unique block of $\N_{G}(\Levi)$ covering $e_s^{\Levi^F}$ such that $B^G=b$, see Lemma \ref{L is suitable}. Our aim in this and the next section is to show that there exists an iAM-bijection $\Irr_0(G,b) \to \Irr_0(\N_{G}(\Levi),B)$. Because of part (b) of Corollary \ref{observation} we will need to focus particulary on the semisimple element in $E_6(\varepsilon q)$ with centralizer $A_2(\varepsilon q)^3.3$.

\subsection{Chevalley presentation and the extended Weyl group}\label{c}

Let $\G$ be a simply connected simple algebraic group with root system $\Phi$ and base $\Delta$. We let $x_\alpha: \overline{\mathbb{F}}_q \to  \G$, $\alpha \in \Phi$, be the Chevalley generators of $\G$ so that $\G=\langle x_\alpha(t) \mid \alpha \in \Phi, t \in \overline{\mathbb{F}}_q \rangle$. We let $\X_{\alpha}$ denote the image of $x_\alpha$ and set $\T:=\langle h_\alpha(t) \mid \alpha \in \Delta \rangle$. As $\G$ is simply connected the torus $\T$ is the direct product of the subgroups $h_\alpha(\overline{\mathbb{F}}_q^\times)$, $\alpha \in \Delta$. Let $W:=\N_\G(\T)/\T$ the Weyl group of $\G$ with respect to the maximal torus $\T$ and $w_0 \in W$ the longest element in $W$. Denote by $S:=\{ s_\alpha \mid \alpha \in \Delta \}$ the set of simple reflections such that $W=\langle s_\alpha \mid \alpha \in W \rangle$.

Set $V:=\langle n_\alpha(1) \mid \alpha \in \Phi \rangle \subset \N_{\G}(\T)$. We have a surjective map $V \to W$ with kernel $H:=V \cap \T$ which is an elementary abelian $2$-group of rank $|\Delta|$. By Matsumotos's lemma we have a canonical set-theoretic splitting $r:W \to V, w\, \mapsto n_{\alpha_1}(1) \cdots n_{\alpha_r}(1)$, where $w=s_{\alpha_1} \cdots s_{\alpha_r}$ is a reduced expression of $w$.

 We will slightly deviate from the definition of the Frobenius endomorphism used in Section \ref{Sylow section}, which will make it easier to exploit Ennola duality. For this let $F_q$ be the Frobenius endomorphism acting as $x_\alpha(t) \mapsto x_{\alpha}(t^q)$ for $t \in \overline{\mathbb{F}}_q$, $\alpha \in \Delta$. Moreover, if $\G$ is of type $E_6$, let $\gamma: \G \to \G$ be the graph automorphism given by $\gamma(x_\alpha(t))=x_{\tau(\alpha)}(t)$ for $t \in \overline{\mathbb{F}}_q$, $\alpha \in \Delta$. Here, $\tau$ denotes the unique non-trivial symmetry of the associated Dynkin diagram. For $\varepsilon \in \{1,-1\}$ we then define the Frobenius endomorphism $F$ to be 
 $$F:=F_q \text{ if }\varepsilon=1 \text{ and } F:=F_q \gamma_0 \text{ if } \varepsilon=-1,$$
 where $\gamma_0:=\gamma \tilde{w}_0$ with $\tilde{w}_0:=r(w_0) \in V$ the canonical representative of the longest element in the Weyl group. By \cite[Lemma 3.12]{MS} we know that
$\C_V(\tilde{w}_0 \gamma)=V$. 
 

Let $s \in (\T^\ast)^F$ be the representative constructed in \cite[Theorem 5.1]{Bonnafe} of the rational conjugacy class of one of the elements relevant in Corollary \ref{observation}. Set $W(s):= \{w \in W^\ast \mid {}^w s=s \}$ and let $W^\circ(s)$ the Weyl group of $\C^\circ_{\G^\ast}(s)$ relative to the maximal torus $\T^\ast$. We let $v \in V$ be the canonical representative (of a $W^\circ(s)$-conjugate) of the longest element of the Weyl group $W^\circ(s)$. For $s$ with centralizer of type $A_2(\varepsilon q)^3.3$ the above constructions will be made more explicit later. As the image of $v$ in $W$ lies in $W^\circ(s)$ there exists by Lang's theorem an element $g \in \G^\circ(s)$ such that $g F(g)^{-1}=v$.
%
\begin{notation}
	For an integer $d \in \{1,2\}$ and a sign $\varepsilon \in \{ \pm \}$ we define a new integer $d_\varepsilon$ as $d_{\varepsilon}:=d$ if $\varepsilon=1$ and $d_\varepsilon:=d+(-1)^{d+1}$ if $\varepsilon=-1$. We define $v_{(d_\varepsilon)}:=1$ if $d=1$ and $v_{(d_\varepsilon)}:=v$ if $d=2$.
	
\end{notation}

\subsection{Construction of the semisimple element}\label{construction}

In this subsection, let $\G$ be of type $E_6$. We will now focus on the element $s$ whose centralizer has type $A_2(\varepsilon q)^3 .3$. We will first construct a concrete representative of the conjugacy class of $s$.
For this, let $\tilde{\alpha}$ be the highest root of $\Phi$ and set $\alpha_0=-\tilde{\alpha}$.
We label the simple roots in the affine Dynkin diagram of $E_6$ as follows:
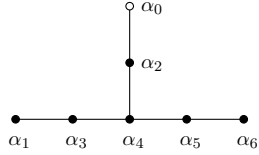
\newline
\begin{center}
	\dynkin[extended,label,label macro/.code={\alpha_{\drlap{#1}}},
	edge
	length=.75cm]E6
\end{center}
Recall that the highest root satisfies $\tilde{\alpha}=\alpha_1 + \alpha_6 +2 \alpha_3+2 \alpha_5+2 \alpha_2+ 3 \alpha_4$
and that
$$\Z(\G)=\{h_{\alpha_1}(\omega)
h_{\alpha_3}(\omega^2)
h_{\alpha_5}(\omega)
h_{\alpha_6}(\omega^2) \mid \omega^3=1 \}$$ by \cite[Table 1.12.6]{GLS}.
 
For $\omega \in \overline{\mathbb{F}}_p^\times$ of order $3$, we consider the semisimple element
$$s_0:=h_{\alpha_1}(\omega^2) h_{\alpha_3}(\omega) h_{\alpha_5}(\omega) h_{\alpha_6}(\omega^2) \in \T^{F}$$
with image $s:=\pi(s_0) \in \T^\ast$. 

It follows that $s_0$ is $\gamma$-stable, where $\gamma$ is the graph automorphism of order $2$ of $E_6$, and $F_p(s_0)s_0^{-1} \in \Z(\G)$ if and only if $p \equiv 1 \mod 3$. In particular, $s$ is $F$-stable if and only if $q \equiv 1_\varepsilon \mod 3$. In this case, the element $s$ is isolated with connected centralizer of type $(A_2(\varepsilon q))^3$. On the other hand, if $q \equiv 2_\varepsilon \mod 3$, then $s$ is $w_0 F$-stable. 

 A base of the root system of $\C^\circ_{\G^\ast}(s)$ is given by $\Delta(s)=(\Delta \setminus \{\alpha_4 \}) \cup \{\alpha_0\}.$ Let $A_1:=\{\alpha_1,\alpha_3\}$, $A_2:=\{\alpha_5, \alpha_6\}$ and $A_3:=\{ \alpha_0 , \alpha_2 \}$. Denote by $\G_i:=\langle x_{\alpha}(t) \mid t \in \overline{\mathbb{F}}_p, \alpha \in A_i \rangle$ the reductive group with base $A_i$. Recall that $\mathrm{C}^\circ_{\G^\ast}(s)$ is a semisimple algebraic group as the element $s$ is isolated. We obtain $\G^\circ(s) =\G_1 \G_2 \G_3$, where $\G_i \cong \SL_3(\overline{\mathbb{F}}_p)$. There exists an $F$-stable central subgroup $\bZ$ of order $3$ of $\G(s)$ such that $\G_i \cap \G_j = \bZ$ for all $i \neq j$. We let $\mathbf{\T}_i= \T \cap \mathbf{G}_i$ such that $\T$ is a central product $\T=\T_1 \T_2 \T_3$.  We define $v_1:=n_{\alpha_1}(1)$, $v_2:=n_{\alpha_0}(1)$, $v_3:=n_{\alpha_6}(1)$, and denote
 $$v:=v_1 v_2 v_3.$$
As before, we let $g \in \G^\circ(s)$ such that $g F(g)^{-1}=v$.
 
%

\subsection{Construction of Levi subgroups}


Let $\G$ be again of exceptional type. For a subset $I$ of the simple roots $S$ we denote by $\Levi_I= \T \langle \X_\alpha \mid \alpha \in \Phi_I \rangle$ the standard Levi subgroup of $\G$ associated to it.

\begin{lemma}\label{split constr}

Assume the notation from above and let $s \in (\T^\ast)^F$ be one of the quasi-isolated elements appearing in Corollary \ref{observation}. Then the following hold:
	\begin{enumerate}[label=(\alph*)]
		\item The torus $\T$ is the $1_\varepsilon$-split Levi subgroup associated to the semisimple  element $s \in (\G^\ast)^{F}$.
		\item Assume that $\G$ is of type $E_6$ and let $I= \{\alpha_4\}$ if $s$ has rational type $A_2(\varepsilon q)^3.3$ and $I=\emptyset$ otherwise. Then the Levi subgroup ${}^g\Levi_I$ is the $2_\varepsilon$-split Levi subgroup associated to the semisimple element $s \in (\G^\ast)^{F^\ast}$ with $F$-fixed points isomorphic to $({}^g \Levi_I)^F \cong\Levi_I^{vF}$.
	\end{enumerate}
\end{lemma}

\begin{proof}
	Recall that by construction $s \in (\T^\ast)^{F}$ and $(\T^\ast)^{F}$ has polynomial order $\Phi_{1_\varepsilon}^6$. Since $\T$ is a maximal torus of $\G$ it follows that $\C_{\G}(\T)=\T$, whence part (a).

	For part (b) recall that the Lang image under $F$ of the element $g \in \G^\circ(s)$ is $v$ and so $\Levi^F \cong \Levi_I^{vF}$.
	If $\C_{G^\ast}(s)$ has type $D_4(\varepsilon q)^3.3$ the result in part (b) follows from Corollary \ref{observation}(b) and the fact that $v \in W^\circ(s)$ is conjugate to the longest element in $W$.	If $s$ has centralizer of type $A_2(\varepsilon q)^3.3$, then 
	the torus $(\T^{g})^F \leq \C_{\G^\ast}(s)$ has order $$|\T^{vF}|=|\T_1^{vF}| |\T_2^{vF}| |\T_3^{vF}|=\Phi_1(q)^3 \Phi_2(q)^3.$$
	In particular, $\T^{g}$ contains a Sylow $2_\varepsilon$-torus of $(\G^\circ(s),F)$. Hence, by Lemma \ref{minimal split} the centralizer $\C_{\G}(({}^g\T)_{\Phi_{2_\varepsilon}})$ in $\G$ of the Sylow $\Phi_{2_\varepsilon}$-torus of $({}^g \T,F)$ is the $2_\varepsilon$-split Levi subgroup associated to $s$. We compute that $\Levi_I \cong \C_{\G}(\T_{\Phi_{2_\varepsilon}})$, where $$\T_{\Phi_{2_\varepsilon}}=\{ h_{\alpha_0}(t_0) h_{\alpha_1}(t_1) h_{\alpha_6}(t_6) \mid t_0,t_1,t_6 \in \overline{\mathbb{F}}_{q}^\times \}$$ is the Sylow $\Phi_{2_\varepsilon}$-torus of $(\T,vF)$. It follows from this that $\Levi={}^g \Levi_I$ is the $2_\varepsilon$-split Levi subgroup associated to $s$.
\end{proof}

In view of the previous lemma, we define $L_{(1_\varepsilon)}:= \T^{F}$ and $N_{(1_\varepsilon)}:=\N_{\G^F}(\T,e_s^{\T^F})$. Whereas, we define $L_{(2_\varepsilon)}:=\Levi_I^{vF}$ and $N_{(2_\varepsilon)}:=\N_{\G^{vF}}(\Levi_I,e_s^{\Levi_I^{vF}})$. For $d \in \{1,2\}$ we let $W_{(d)}:=N_{(d)}/L_{(d)}$ and observe that $W_{(d)} \cong \C_{W(s)}(v_{(d)})$.

\subsection{Construction of the normalizer}

Our aim in this section is to construct a subgroup $V_{(2_\varepsilon)} \leq \C_V(v)$ such that $N_{(2_\varepsilon)}=L_{(2_\varepsilon)} V_{(2_\varepsilon)}$.

For this, we first need the following two general observations about Chevalley relations.

\begin{lemma}\label{Chevalley}
Let $n \in V$ representing $w \in W$. For $\beta \in \Phi$ and $u \in \overline{\mathbb{F}}_q^\times$ we have ${}^n n_\beta(u)=n_{w(\beta)}( \pm u)$ and similarly ${}^n x_\beta(u)=x_{w(\beta)}(\pm u)$. In particular, if $n$ has odd order and $w(\beta)=\beta$, then ${}^n x_\beta(u)=x_\beta(u)$.
\end{lemma}

\begin{proof}
The group $V$ is generated by $\{ n_\alpha(1) \mid \alpha \in \Phi \}$, so that we can write $n=n_{\alpha_1}(1) \cdots n_{\alpha_l}(1)$ for some $l$. If $n=n_\alpha(1)$ for some $\alpha$ we have ${}^{n} x_{\beta}(u)=x_{w_\alpha(\beta)}( \pm u)$ (and ${}^{n} n_{\beta}(u)=n_{w_\alpha(\beta)}( \pm u)$) by \cite[Satz 2.1.6]{BS2006}. The statement follows now by induction on $l$.
\end{proof} 

Recall that $r:W \to V$ denotes the canonical set theoretic splitting of the quotient map $V \to W$.
We have the following lemma, see \cite[Lemma 3.2]{AdamsHe}.

\begin{lemma}\label{AdamsHe}
Suppose $I \subset S$ is a set of simple roots with corresponding Levi factor $\Levi_I$. Let $\rho^\vee(I)$ be one-half the sum of the positive coroots of
$\Levi_I$ and let $z_I := (2 \rho^\vee(I))(-1)$ be the principal involution in $\Levi_I$. Then
$$r(w_0(I))^2 = z_I.$$
\end{lemma}


\begin{proposition}\label{Chevalley2}
	Let $s$ be the quasi-siolated semisimple element with centralizer $A_2(\varepsilon q)^3.3$ constructed above.
There exists a section $V_{(2_\varepsilon)}$ of $\C_{W(s)}(v)$ under the natural map $V \to W$ which has the following properties:
	\begin{enumerate}[label=(\alph*)]
\item $V_{(2_\varepsilon)}$ is $\gamma$-stable.
\item $V_{(2_\varepsilon)} \cong C_4 \wr C_3$.
\item $V_{(2_\varepsilon)} \cap H= \langle h_{\alpha_1}(-1),h_{\alpha_4}(-1),h_{\alpha_6}(-1) \rangle \cong C_2^3$.
\item $V_{(2_\varepsilon)}$ centralizes the subgroup $\langle \X_{\alpha_4}, \X_{-\alpha_4} \rangle$.
\end{enumerate}
\end{proposition}

\begin{proof}
Let $V^\circ_{(2_\varepsilon)}=\langle v_1,v_2,v_3 \rangle$. For $i \neq j$ the sets $A_i$ and $A_j$ are orthogonal by definition and all the roots of the root system $E_6$ have the same length. Hence, \cite[Bemerkung 2.1.7(b)]{BS2006} ensures that $[v_i,v_j]=1$. Hence, $V^\circ_{(2_\varepsilon)}$ is abelian.

Let $\tilde{w}_0=r(w_0)$ be the canonical representative of the longest element of the Weyl group $W$. Take $\tilde{w}_0'=r(w_0')$ to be the canonical representative of the longest element $w_0'$ of the Weyl group of type $D_5$ associated to the base $\Delta'=\{ \alpha_2,\alpha_3,\alpha_4,\alpha_5,\alpha_6 \}$.
We define $x:= \tilde{w}_0 \tilde{w}_0'$. By construction, the element $x$ induces a triality graph automorphism $\sigma$ on the extended Dynkin diagram of $E_6$. It follows from Lemma \ref{AdamsHe} that $\tilde{w}_0$ and $\tilde{w}_0'$ both have order $2$. Consequently, 
$$\gamma(x)=\tilde{w}_0 {}^{\tilde{w}_0} \tilde{w}_0'=\tilde{w}_0' \tilde{w}_0=x^{-1}.$$ 
The image of $x$ in $W$ has order $3$. Observe that the kernel of the natural map $V \to W$ is the elementary $2$-group $H$ and thus $x$ has order $3$ or $6$. We define $v':=x^4$ if $x$ has order $6$ and $v':=x$ if $x$ has order $3$. Thus, the element $v'$ has order $3$ and still induces the graph automorphism $\sigma$ on the extended Dynkin diagram. We denote $V_{(2_\varepsilon)}:=\langle V^\circ_{(2_\varepsilon)}, v' \rangle$. By construction, the group $V_{(2_\varepsilon)}$ is a $\gamma$-stable section of $\C_{W(s)}(v)$.
 From the Chevalley relations \cite[Satz 2.1.6]{BS2006} (see also Lemma \ref{Chevalley}) it follows that ${}^{v'} n_{\delta}(1)=n_{\sigma(\delta)}( \varepsilon_\delta)$ for $\delta \in \{ \alpha,\beta\}$ and signs $\varepsilon_\delta \in \{ \pm 1 \}$ (possibly depending on the choice of structure constants in the Chevalley groups).
%
This ensures that $v'$ normalizes $V^\circ_{(2_\varepsilon)}$. Since $V^\circ_{(2_\varepsilon)} \cong C_4^3$ it follows from this that $V_{(2_\varepsilon)} \cong C_4 \wr C_3$ (see Lemma \ref{wreath}). From this, we also see that $V_{(2_\varepsilon)} \cap H= \langle h_{\alpha_1}(-1),h_{\alpha_0}(-1),h_{\alpha_6}(-1) \rangle \cong C_2^3$ with $h_{\alpha_0}(-1)=h_{\alpha_6}(-1) h_{\alpha_4}(-1) h_{\alpha_1}(-1)$.

We are left to show that $V_{(2_\varepsilon)}$ centralizes the root subgroup $\X_{\pm \alpha_4} $.
Note that the roots $\alpha_0,\alpha_1,\alpha_6$ are all orthogonal to $\alpha_4$. Hence, by the Chevalley relations (see \cite[Bemerkung 2.1.7(b)]{BS2006}) the elements $n_\delta(1)$ with $\delta \in \{ \alpha_0,\alpha_1,\alpha_6\}$ centralize the root subgroup $\X_{\pm \alpha_4}$. Moreover, the element $v'$ has odd order and thus it centralizes $\X_{ \pm \alpha_4}$ by Lemma \ref{Chevalley}.
%
\end{proof}

\subsection{Automorphisms and Lusztig series of twisted groups}

The following notation which follows \cite[Notation 3.3]{MS} will be used for working with twisted groups.

\begin{notation}   \label{not:3.4}
	As before let $F$ be the fixed Frobenius endomorphism from \ref{c} defining an $\mathbb{F}_q$-structure on $\G$ with $q=p^m$.
	Set $h:=\mathrm{o}(\gamma) \mathrm{o}(v)$. Let $E:= \C_{hm} \times \langle \gamma \rangle$ if $\varepsilon=1$ and $E:=\C_{hm}$ if $\varepsilon=-1$ such that the first factor of $E$ acts on
	${\tilde \G}^{F_p^{2hm}}$ by $\langle F_p \rangle$ and the second by the group $\langle \gamma \rangle$ generated by graph automorphisms.
	Note that this action is faithful. Let $\widehat F_p, \widehat \gamma, \widehat F \in E$ be
	the elements that act on ${\tilde \G}^{F_p^{2hm}}$ by $F_p$, $\gamma$ and $F$,
	respectively.
\end{notation}

By construction, the element $\hat{F}$ centralizes $\Gtilde^F$. We can therefore identify the subgroup $\mathcal{B} \subset \mathrm{Aut}(\Gtilde^F)$ from before with the quotient group $E/ \langle \hat{F} \rangle $. The following lemma is a consequence of the proof of \cite[Proposition 3.4]{MS}.

\begin{lemma}\label{TwistIsom}
	Let $v \in V$ be as before and $g \in \G$ with $v=g F(g)^{-1}$. Then the map 
	$$\iota: \tilde\G^{F_q^{2e}} \rtimes E \to \tilde\G^{F_q^{2e}} \rtimes E, \quad x \mapsto {}^g x,$$
	is an isomorphism which maps $\G^F \rtimes E$ onto $\G^{vF} \rtimes E$.
\end{lemma}

It is a natural question to ask whether the isomorphism from Lemma \ref{TwistIsom} preserves Lusztig series.

\begin{lemma}\label{Lusztig twist}
	Let $s$ be one of the the semisimple elements in Corollary \ref{observation} and $v \in G^\circ(s)$ as before.
 Then the isomorphism $\iota:\G^F \to \G^{vF}, x \mapsto {}^g x,$ induces a bijection $\mathcal{E}(\G^{vF},s) \to \mathcal{E}(\G^F,s)$.
\end{lemma}

\begin{proof}
	Recall that $s \in \T^\ast$ and $W^\circ(s)$ is defined as the Weyl group of $\C^\circ_{\G^\ast}(s)$ with respect to its maximal torus $\T^\ast$. Assume that for $z \in \{1,v\}$ the duality between $(\G,zF)$ and $(\G^\ast,Fz^\ast)$ is defined around the maximal tori $\T$ and $\T^\ast$ with the same choices, see \cite[Definition 13.10]{DM} and the remarks following it. The characters in $\mathcal{E}(\G^{vF},s)$ are the constituents of $R_{\mathbf{S}}^\G(\theta')$, where $(\mathbf{S},\theta')$ is in duality with $(\mathbf{S}^\ast,s')$ and $s'$ is rationally conjugate to $s$. In other words, by \cite[3.25]{DM} there exists $x \in \G$ such that $\T={}^{x^{-1}} \mathbf{S}$ such that the image $w$ of $x^{-1} (vF)(x)$ in $W^\ast$ lies in $W^\circ(s)$ and the character $\theta:={}^{x^{-1}} \theta' \in \Irr(\T^{w vF})$ corresponds to $s \in (\T^{\ast})^{w^\ast v^\ast F^\ast}$ under the isomorphism $\Irr(\T^{wvF}) \to \T^{ F^\ast v^\ast w^\ast}$ induced by duality, see \cite[Proposition 13.11]{DM}. It therefore suffices to show that the irreducible constituents of ${}^g R_{\mathbf{S}}^{\G}(\theta')$ lie in $\mathcal{E}(\G^F,s)$. We have ${}^g R_{\mathbf{S}}^\G(\theta')=R_{{}^g \mathbf{S}}^{\G}({}^g \theta')$ by the proof of \cite[Proposition 7.13]{MarcBook}. It therefore suffices to show that the pair $({}^g \mathbf{S},{}^g \theta')$ corresponds to $(\mathbf{S}^\ast,s)$ under duality. We have ${}^{(gx)^{-1}} ({}^g \mathbf{S},{}^g \theta')=(\T,\theta)$ with $wv=(gx)^{-1} F(gx) \in W^\circ(s)$. By the consistent choice of duality, the character $\theta$ corresponds to the element $s \in (\T^{\ast})^{w^\ast v^\ast F^\ast}$ under the duality isomorphism $\Irr(\T^{wvF}) \to (\T^\ast)^{F v^\ast w^\ast }$. Thus,  $({}^g \mathbf{S},{}^g \theta')$ corresponds to $(\mathbf{S}^\ast,s)$ under duality.
\end{proof}

The proof of Lemma \ref{Lusztig twist} also shows the following:

\begin{corollary}\label{Lusztig twist levi}
Let ${}^g \Levi_I$ be the $2_\varepsilon$-split Levi subgroup associated to $s$ as in Lemma \ref{split constr}. The isomorphism $\iota:\Levi_I^F \to \Levi_I^{vF}, x \mapsto {}^g x,$ induces a bijection $\mathcal{E}(\Levi_I^{vF},s) \to \mathcal{E}(\Levi_I^F,s)$.
\end{corollary}

Recall that $W_{(e)}=N_{(e)}/L_{(e)}$. We have $W_{(1_\varepsilon)}=W(s)$ and $W_{(2_\varepsilon)} = \C_{W(s)}(v)$ by Lemma \ref{stabilizer}. The groups $W_{(1)}$ and $W_{(e)}$ have a common Sylow $2$-subgroup $W_2$ with $v_{(1)},v_{(e)} \in \Z(W_2)$.

\begin{lemma}\label{alpha}
	There exists a natural bijection $\Xi: \mathcal{E}_2(L_{(1_\varepsilon)},s)^{W_2} \to \mathcal{E}_2(L_{(2_\varepsilon)},s)^{W_2}$.
	
\end{lemma}

\begin{proof}
	Let $z \in \{v_{(1)},v_{(e)} \}$. By duality we have a bijection $\Irr(\T^{zF}) \to (\T^\ast)^{F^\ast z^\ast}$, which restricts to a $\C_{W(s)}(z)$-equivariant bijection 
	$$\mathcal{E}_2(\T^{zF},s) \to A_z:=\{ t \in (\T^\ast)^{F^\ast z^\ast} \mid t_{2'}=s \}.$$ 
As $z \in W_2$ it follows that $A_{v_{(1)}}^{W_2}=A_{v_{(e)}}^{W_2}$. Since $W_2 \leq \C_{W(s)}(z)$ we therefore obtain a bijection $$ \mathcal{E}_2(\T^F,s)^{W_2} \to \mathcal{E}_2(\T^{vF},s)^{W_2}.$$

Moreover, by definition $\T^\ast=\C_{\Levi_I^\ast}(s)$ and so Deligne--Lusztig induction yields a  bijection $R_{\T^{vF}}^{\Levi_I^{vF}}:  \mathcal{E}_2(\T^{vF},s) \to \mathcal{E}_2(\Levi_I^{vF},s)$, which is equivariant for all automorphisms of $\Levi_I$ which commute with $vF$ and stabilize $\T$.
The composition of these bijections yields a 
bijection $\mathcal{E}_2(\T^F,s)^{W_2} \to \mathcal{E}_2(\Levi_I^{vF},s)^{W_2}$ as desired. 
\end{proof}

%

\subsection{Extension maps}
Whenever $\Levi_{(d)}$ is a maximal torus for $d \in \{1,2\}$ an extension map $L_{(d)} \lhd N_{(d)}$ an automorphism-equivariant extension map is known by the work of Malle--Späth.

By Corollary \ref{observation} we therefore only need to consider the case when the centralizer of $s$ is of type $A_2(\varepsilon q)^3.3$ and $d=2_\varepsilon$.
The aim of this section is therefore to give an automorphism-equivariant extension map for the extension $L_{(2_\varepsilon)} \lhd N_{(2_\varepsilon)}$. For this, first observe that the semisimple element $s$ is $E':=E$-stable if $p \equiv 1 \mod 3$ and $E':=\langle \gamma,w_0F_p\rangle$-stable if $p \equiv 2 \mod 3$. Hence, for $d \in \{1,2\}$ and $G_{(d)}:=\G^{v_{(d)} F}$ we have
	$$\N_{G_{(d)} E}(L_{(d)},e_s^{L_{(d)}})=N_{(d)} E'.$$

To prove the existence of such an extension map, we need the following lemma.

\begin{lemma}\label{ext map}
	Let $s$ be the semisimple element with centralizer $A_2( \varepsilon q)^3.3$.
Then there exists an $E'$-equivariant extension map $\Gamma$ for $Z:=V_{(2_\varepsilon)} \cap H \lhd V_{(2_\varepsilon)}$.
\end{lemma}

\begin{proof}
	Observe that $w_0 F_p$ acts in the same way as $\gamma$ on $V$. We identify $V^\circ_{(2_\varepsilon)}=\langle v_1,v_2,v_3 \rangle$ with $C_4^3$. For $\nu=\nu_1 \times \nu_2 \times \nu_3$ we let $\theta=\theta_1 \times \theta_2 \times \theta_3 \in \Irr(V^\circ_{(2_\varepsilon)} \mid \nu)$ such that $\theta_i=\theta_j$ whenever $\nu_i=\nu_j$. Hence, $(V_{(2_\varepsilon)} E')_\theta=(V_{(2_\varepsilon)} E')_\nu$. If $v' \notin V_{(2_\varepsilon)}$ we let $\Gamma(\nu):=\theta$. Otherwise, since $V_{(2_\varepsilon)} \cong C_4 \wr C_3$, there exists a character $\Gamma(\nu) \in \Irr(V_{(2_\varepsilon)} \mid \theta)$ with $v'$ in its kernel. We show that there exists a choice of $\theta$ such that the so-obtained character $\Gamma(\nu)$ is $E'_\nu$-stable. For this we can assume that $\nu$ is $\langle \gamma,v' \rangle$ since the claim is clear otherwise.
	The character $\theta:=\theta_1 \times \theta_2 \times \theta_3$ in $\Irr(V^\circ_{(2_\varepsilon)} \mid \nu)$ is $\gamma$-stable if and only if $\theta_1=\theta_3$. Moreover, $\theta$ can only be $v'$-stable if all $\theta_i$ have order $4$. Now, we know that for $i=1,2,3$ there are signs $\varepsilon_i \in \{\pm 1 \}$ such that ${}^{v'} v_i=v_{i+1}^{\varepsilon_i}$. Since $v'$ has order $3$ this forces $\varepsilon_1 \varepsilon_2 \varepsilon_3=1$.  Moreover, $\gamma(v')=v'^{-1}$. In particular, ${}^{\gamma v' \gamma} v_2={}^{v'^{-1}} v_2$. Since $\gamma(v_2)=v_2$ and $\gamma(v_1)=v_3$, we obtain $v_1^{\varepsilon_2}=v_1^{\varepsilon_1}$. We deduce $\varepsilon_1 \varepsilon_2=1$ which forces $\varepsilon_3=1$. In particular, for any character $\theta_1 \in \Irr(C_4)$ of order $4$, the character $\theta=\theta_1 \times \theta_1^{\varepsilon_1} \times \theta_1$ is $\langle \gamma,v' \rangle$-stable and covers the character $\nu$.
\end{proof}

For constructing extensions we use the following lemma, which is a consequence of \cite[Lemma 2.11]{S12}.

\begin{lemma}\label{extension}
Let $\tilde{X}$ be a normal subgroup of a finite group $\tilde{Y}$ and $Y$ a subgroup of $\tilde{Y}$ with $\tilde{Y}=\tilde{X} Y$. Assume that $\vartheta \in \Irr(\tilde{X})$ is a $\tilde{Y}$-stable character which restricts irreducibly to the intersection $X:=\tilde{X} \cap Y$. If the character $\mathrm{Res}_{X}^{\tilde X}(\vartheta)$ extends to a character $\mu \in \Irr(Y)$, then $\vartheta$ extends to a character of $\Irr(\tilde{Y} \mid \mu)$.
\end{lemma}


\begin{lemma}\label{extension map}
Let $s$ be the semisimple element with centralizer $A_2( \varepsilon q)^3.3$.
There exists an extension map $\Lambda_{(2_\varepsilon)}$ for $\mathcal{E}_2(L_{(2_\varepsilon)},s)$ with respect to $L_{(2_\varepsilon)} \lhd N_{(2_\varepsilon)}$ such that
$\Lambda_{(2_\varepsilon)}$ is $\N_{G_{(2_\varepsilon)}E}(L_{(2_\varepsilon)},e_s^{L_{(2_\varepsilon)}})$-equivariant.
\end{lemma}

\begin{proof}
Let $V_{(2_\varepsilon)}$ be the section of $\C_{W(s)}(v)$ constructed in Proposition \ref{Chevalley2}.
We can therefore write $N_{(2_\varepsilon)}$ as $N_{(2_\varepsilon)}=L_{(2_\varepsilon)}  V_{(2_\varepsilon)}$.
The intersection $$Z=L_{(2_\varepsilon)} \cap V_{(2_\varepsilon)}=T_{(2_\varepsilon)} \cap V_{(2_\varepsilon)}=H \cap V_{(2_\varepsilon)}=\langle h_{\alpha_0}(-1),h_{\alpha_1}(-1),h_{\alpha_6}(-1) \rangle$$ is a central subgroup of $\Levi_I$ by Proposition \ref{Chevalley2}. Note that $\Levi_I=\Z^\circ(\Levi_I) [\Levi_I,\Levi_I]$ and $[\Levi_I,\Levi_I] \cong \mathrm{SL}_2(\overline{\mathbb{F}}_p)$.
%
%
By Bonnafé--Rouquier, see the main theorem of \cite{Dat}, the block $e_s^{L_{(2_\varepsilon)}}$ is Morita equivalent to the principal block of the torus $T_{(2_\varepsilon)}$. In particular, this block has $|(T_{(2_\varepsilon)})|_2$ many irreducible characters.

Let $s_0$ be the image of $s$ under the map $\Levi_I^\ast \to [\Levi_I,\Levi_I]^\ast$ with central kernel obtained by duality from the inclusion map $[\Levi_I,\Levi_I] \subset \Levi_I$, see \cite[Section 15.1]{MarcBook}. Since $1 \neq s_0$ is a $2'$-element it follows that its centralizer in $[\Levi_I,\Levi_I]^\ast \cong \mathrm{PGL}_2(\overline{\mathbb{F}}_p)$ is again a torus. In particular, $e_{s_0}^{[L_{(2_\varepsilon)},L_{(2_\varepsilon)}]}$ is again a block with $|(T_{(2_\varepsilon)}) \cap [L_{(2_\varepsilon)},L_{(2_\varepsilon)}]|_2$ many irreducible characters. Since $e_{s_0}^{[L_{(2_\varepsilon)},L_{(2_\varepsilon)}]}$ is the unique block below $e_s^{L_{(2_\varepsilon)}}$ we deduce by Clifford theory that every character in $\Irr([L_{(2_\varepsilon)} ,L_{(2_\varepsilon)} ],e_{s_0}^{[L_{(2_\varepsilon)} ,L_{(2_\varepsilon)} ]})$ is $L_{(2_\varepsilon)}$-stable.

Let us now first construct an extension map $\Lambda_0$ for $[L_{(2_\varepsilon)} ,L_{(2_\varepsilon)} ] Z \lhd [L_{(2_\varepsilon)} ,L_{(2_\varepsilon)} ] V_{(2_\varepsilon)}$.
The group $V_{(2_\varepsilon)}$ centralizes $[\Levi_I,\Levi_I]=\langle \X_{\alpha_4},\X_{-\alpha_4} \rangle$ by Proposition \ref{Chevalley2}. Hence, $[L_{(2_\varepsilon)},L_{(2_\varepsilon)}]V_{(2_\varepsilon)}$ is a central product. Recall the extension map $\Gamma$ for $Z \lhd V_{(2_\varepsilon)}$ from Lemma \ref{ext map}. We can therefore define an extension map $\Lambda_0$ by sending a character $\psi_0 \in \Irr([L_{(2_\varepsilon)},L_{(2_\varepsilon)}]Z)$ with $\nu \in \Irr(Z \mid \psi_0)$ below $\psi_0$ to the character $\psi_0 . \Gamma(\nu) \in \Irr([L_{(2_\varepsilon)},L_{(2_\varepsilon)}] V_{(2_\varepsilon)})$. By the equivariance properties of $\Gamma$, it follows that $\Lambda_0$ is $\N_{G_{(2_\varepsilon)}E}(L_{(2_\varepsilon)},e_s^{L_{(2_\varepsilon)}})$-equivariant as well.
 We now define an extension map $\Lambda_{(2_\varepsilon)}$ for $\mathcal{E}_2(L_{(2_\varepsilon)} s)$ with respect to $L_{(2_\varepsilon)}  \lhd N_{\Lambda_{(2_\varepsilon)}} =L_{(2_\varepsilon)}  V_{(2_\varepsilon)}$ as follows. For $\psi \in \mathcal{E}_2(L_{(2_\varepsilon)} ,s)$ its restriction $\psi_0:=\mathrm{Res}_{[L_{(2_\varepsilon)} ,L_{(2_\varepsilon)} ] Z}^{L_{(2_\varepsilon)} }(\psi)$ to $[L_{(2_\varepsilon)} ,L_{(2_\varepsilon)} ]$ is irreducible. By Lemma \ref{extension} we can therefore define $\Lambda_{(2_\varepsilon)}(\psi) \in \Irr(L_{(2_\varepsilon)}  (V_{(2_\varepsilon)})_\psi)$ as the unique character which restricts to both $\psi \in \Irr(L_{(2_\varepsilon)} )$ and $\Lambda_0(\psi_0)$. By construction the map $\Lambda_{(2_\varepsilon)}$ is $\N_{GE}(L_{(2_\varepsilon)},e_s^{L_{(2_\varepsilon)}})$-equivariant.
\end{proof}

\subsection{Construction of the AM-bijection}

To show the automorphism equivariance of the map constructed in Theorem \ref{AM bijection} we need the following statements which are analogues of \cite[Proposition 4.4, Proposition 4.5]{BroughRuhstorfer}. For $d \in \{1,2\}$ we define $\tilde{T}_{(d)}:=\Ttilde^{v_{(d)}F}$ and let $\tilde{L}_{(d)}:=L_{(d)} \tilde{T}_{(d)}$ and $\tilde{N}_{(d)}:=N_{(d)} \tilde{L}_{(d)}$.

\begin{proposition}\label{equiv1}
Let $d \in \{1,2\}$ and $\Lambda_{(d)}$ be the extension map for $\mathcal{E}_2(L_{(d)},s)$ from \cite[Corollary 3.13]{MS} or Corollary \ref{extension map} with respect to $L_{(d)} \lhd N_{(d)}$. Then the map
	$$\Pi: \mathcal{P}:=\{(\mu,\eta) \mid \mu \in \Irr(L_{(d)},e_s^{L_{(d)}}), \eta \in \Irr(W_{(d)}(\mu)) \} \to \Irr(N_{(d)},e_s^{L_{(d)}})$$ with $\Pi(\mu, \eta) = \Ind_{N_{(d)}(\mu)}^{N_{(d)}}(\Lambda_{(d)}(\mu) \eta)$
	is surjective and satisfies
	\begin{enumerate}[label=(\alph*)]
		\item  $\Pi(\mu, \eta)=\Pi(\mu',\eta')$ if and only if there exists some $n \in N_{(d)}$ such that ${}^n \mu= \mu'$ and ${}^n \eta= \eta'$.
		\item ${}^\sigma \Pi(\mu,\eta)=\Pi({}^\sigma \mu, {}^\sigma \eta)$ for all $\sigma \in E'$.
		\item Let $t \in \tilde T_{(d)}$, $\tilde \mu \in \Irr(\tilde L_{(d)} \mid \mu)$ and $\nu_t \in \Irr(N_{(d)}(\mu)/N_{(d)}({\tilde \mu}))$ is the faithful linear character given by ${}^t \Lambda(\mu)=\Lambda(\mu) \nu_t$. Then we have ${}^t \Pi(\mu,\eta)=\Pi(\mu,\eta \nu_t)$.
	\end{enumerate}
\end{proposition}

\begin{proof}
The proof of part (a) and (b) works as in \cite[Proposition 3.15]{MS}. For part (c) observe that our assumption on the semisimple element $s$ implies that $\C_{\Levi^\ast}(s)$ is a torus. Hence, for all elements $t \in \C_{\G^\ast}(s)_2$ the centralizer $\C_{\Levi^\ast}(st)$ is connected. Thus, every character of $\mathcal{E}_2(L_{(d)},s)$ is $\tilde T_{(d)}$. Hence, also part (c) follows as in \cite[Proposition 3.15]{MS}.
\end{proof}

For the following proposition recall that $W_{(2_\varepsilon)}$ can be viewed as a subgroup of $W_{(1_\varepsilon)}$. The formulation of the following proposition is slightly different from \cite[Proposition 4.5]{BroughRuhstorfer} but the proof is identical.

\begin{proposition}\label{equiv2}
Let $\mu_0 \in \mathcal{E}_2(L_{(1_\varepsilon)},s)^{W_2}$
with image $\mu = \Xi(\mu_0) \in \mathcal{E}_2(L_{(2_\varepsilon)},s)^{W_2}$.
Suppose that $\nu_0 \in \Irr(W_{(1_\varepsilon)}(\mu_0))$ and $t_0 \in \tilde{T}_{(1_\varepsilon)}$ satisfy ${}^{t_0} \Lambda_{(1_\varepsilon)}(\mu_0) = \Lambda_{(1_\varepsilon)}(\mu_0) \nu_0$. Then we have
$${}^{t} \Lambda_{(2_\varepsilon)}(\mu) = \Lambda_{(2_\varepsilon)}(\mu) \nu,$$
where $t := \iota(t_0)$ and $\nu:=\Res^{W_{(1_\varepsilon)}(\mu_0)}_{W_{(2_\varepsilon)}(\mu_0)}(\nu_0)$.
\end{proposition}



\begin{theorem}\label{AM bijection}
Let $b=e_s^{G}$ be one of the blocks in Corollary \ref{observation} and $B$ the unique block of $\N_G(\Levi)$ covering $e_s^{L}$. Then there exists an $\N_{\tilde{G} \mathcal{B}}(\Levi)_B$-equivariant bijection $f:\Irr_0(\N_{G}(\Levi),B) \to \Irr_0(G,e_s^{G})$.
\end{theorem}

\begin{proof}
	We first parametrize the local characters. For this we use the $e$-split Levi subgroup $\Levi_{(e)}$.
	Let $\Lambda^{(e)}$ be the extension map associated to $L_{(e)} \lhd N_{(e)}$ from Lemma \ref{extension map}.
Recall that we denote $W_{(e)}=N_{(e)}/L_{(e)}$. By Proposition \ref{equiv1} we obtain a surjective map
$$\Pi : \mathcal{P}=\{(\mu, \eta) \mid \mu \in \Irr(L_{(e)},e_s^{L_{(e)}}), \eta \in \Irr(W_{(e)}(\mu)) \} \to  \Irr(N_{(e)},e_s^{L_{(e)}}), (\mu, \eta) \to \Ind_{N_{(e)}(\mu)}^{N_{(e)}}(\Lambda_{(e)}(\mu)\eta).$$
The characters of height zero in the unique block $N_{(e)}$ covering the block $e_s^{L_{(e)}}$ are precisely those which correspond to tuples $(\mu,\eta) \in \mathcal{P}$ with $2 \nmid |W_{(e)}:W_{(e)}(\mu)|$ and $\eta \in \Irr_{2'}(W_{(e)}(\mu))$. By Proposition \ref{equiv1}(c) we can assume that the pair $(\mu,\eta)$ satisfies additionally $W_2 \leq W_{(e)}(\mu)$.

Let us now consider the global characters. We abbreviate $G_{(1)}:=\G^{v_{(1)} F}$. Every height zero character $\chi$ of $\mathcal{E}_2(G_{(1)},s)$ lies in the principal series, i.e. it is of the form $\chi=R_{L_{(1)}}^{G_{(1)}}(\mu_0)_{\eta_0}$ for some $\mu_0 \in \mathcal{E}_2(L_{(1)},s)$ and $\eta_0 \in \Irr(W_{(1)}(\mu_0))$, see Corollary \ref{HarishChandra}. For such a character to have height zero we must by Corollary \ref{necessary height zero} necessarily have that $\mu_0$ is of height zero with $2 \nmid |W_{(1)}:W_{(1)}(\mu_0)|$ and $\eta_0 \in \Irr_{2'}(W_{(1)}(\mu_0))$. 

We now construct our bijection. We can choose the Sylow $2$-subgroup $W_2$ of $W_{(1)}$ and $W_{(e)}$ to be $E'$-stable and we observe that $ \N_{W_{(e)}}(W_2)E'=\N_{W_{(1)}}(W_2)E'$. Moreover, recall that by Lemma \ref{alpha} and its proof we have an $\N_{W_{(e)}}(W_2)E'$-equivariant bijection
$$\Xi:\mathcal{E}_2(L_{(e)},s)^{W_2} \to \mathcal{E}_2(L_{(1)},s)^{W_2}.$$
Fix an $\N_{W_{(e)}}(W_2)E'$-transversal $\mathcal{T}$ of $\mathcal{E}_2(L_{(e)},s)^{W_2}$.
For $\mu \in \mathcal{T}$, the group $W_2$ is a common Sylow $2$-subgroup of $W_{(e)}(\mu)$ and $W_{(1)}(\mu_0)$ where $\mu_0:=\Xi(\mu)$.
As in \cite[Lemma 4.3]{BroughRuhstorfer} we construct for every $\mu \in \mathcal{T}$ an $(\N_{W_{(e)}}(W_2)E')_\mu$-equivariant McKay-bijection 
$$f_\mu:\Irr_{2'}(W_{(e)}(\mu)) \to \Irr_{2'}(W_{(1)}(\mu_0)),$$
compatible with the multiplication of linear characters.
For $\mu \in \mathcal{E}_2(L_{(e)},s)^{W_2}$ with ${}^x \mu \in \mathcal{T}$ for $x \in \N_{W_{(e)}}(W_2)E'$ we define ${}^x f_\mu:=f_{{}^x \mu}$.
Denote by $B_{(e)}$ the unique block of $N_{(e)}$ covering $e_s^{L_{(e)}}$ and set $b_{(1)}=e_s^{G_{(1)}}$. By Proposition \ref{equiv1} we obtain a bijection
	$$\Irr_0(N_{(e)},B_{(e)}) \to \Irr_0(G_{(1)},b_{(1)}), \quad  \Ind_{N_{(e)}(\mu)}^{N_{(e)}}(\Lambda^{(e)}(\mu)\eta) \mapsto R_{L_{(1)}}^{G_{(1)}}(\Xi(\mu))_{f_\mu(\eta)}.$$
	 Note that it is actually not clear that the so-obtained map is well-defined as the character $R_{L_{(1)}}^{G_{(1)}}(\Xi(\mu))_{f_\mu(\eta)}$ is not necessarily of height zero.
	 However, by Lemma \ref{Lusztig twist} and Corollary \ref{Lusztig twist levi}, using the map $\iota$, we obtain bijections $\Irr_0(N_{(e)},B_{(e)}) \to \Irr_0(\N_G(\Levi),B) $ and $\Irr_0(G_{(1)},b_{(1)}) \to \Irr_0(G,b)$ and by Remark \ref{br application} we have $|\Irr_0(B)|=|\Irr_0(b)|$. In addition, the map is injective by the properties of Harish-Chandra series and its image contains all height zero characters by the arguments from above. Therefore, the map is necessarily a well-defined bijection and yields by applying $\iota$ a bijection
	 $f:\Irr_0(\N_{G}(\Levi),B) \to \Irr_0(G,b)$.

It is therefore left to show that $f$ is $\N_{\tilde{G}E}(\Levi)_B$-equivariant. This can however by checked exactly as in the proof of \cite[Theorem 4.8]{BroughRuhstorfer}, where we replace
\cite[Proposition 4,4, Proposition 4.5]{BroughRuhstorfer} by
Propositions \ref{equiv1} and \ref{equiv2}.
\end{proof}

\section{Verifying the inductive condition}\label{section 11}


In this section we show the bijection $f$ from Theorem \ref{AM bijection} is a strong iAM-bijection. For this we need the following auxiliary result: 

\begin{proposition}\label{assump}
In the situation of Theorem \ref{AM bijection},	a character $\chi \in \Irr_0(G,b)$ satisfies assumption (i) of Theorem \ref{12} if and only if its AM-correspondent $\psi:=f(\chi)$ satisfies assumption (ii).
\end{proposition}

\begin{proof}
By the equivariance of $f$ it is clear that $(\tilde{G} \mathcal{B})_\chi = \tilde{G}_\chi \mathcal{B}_\chi$ if and only if $( \N_{\tilde{G}}(L)  \mathrm{N}_{G \mathcal{B}}(L) )_{\chi'}= \N_{\tilde{G}}(L)_{\chi'} \mathrm{N}_{G \mathcal{B}}(L)_{\chi'}$. To finish the proof of the proposition we show that every character in $\Irr_0(G,b)$ (resp. $\Irr_0(\N_{G}(\Levi),B)$) extends to its inertia group in $G \mathcal{B}_\chi$ (resp. $\N_{G\mathcal{B}}(\Levi)$). 

Let $\chi \in \Irr_0(G,b) \cup \Irr_0(\N_{G}(\Levi),B)$ and observe that $\mathcal{B}_\chi$ is only non-cyclic if $\G$ is of type $E_6$ and $\gamma \in \mathcal{B}_\chi$. However if $\gamma \in \mathcal{B}_\chi$ then for $g \in \tilde{G}$ we have $\gamma \in \mathcal{B}_{{}^g \chi}$ if and only if $\gamma(g)g^{-1} \in \tilde{G}_\chi$. Since $H^1(F,\mathrm{Z}(G))$ has odd order and $\gamma$ acts by inversion on it this is equivalent to $g \in \tilde{G}_\chi$. Hence it suffices to show that there exists a character in the $\tilde{L}$-orbit of $\chi$ which satisfies condition (i) resp. (ii) of Theorem \ref{12}.

In particular for $\chi \in \Irr_0(G,b)$ this follows from Theorem \ref{star}. For the characters of $\N_{G}(\Levi)$, the claim is a consequence of \cite[Proposition 3.18]{MS} whenever the associated $e$-split Levi subgroup $\Levi$ is a maximal torus. On the other hand, if $\varepsilon=-1$ then $\mathcal{B}$ is cyclic. If $\varepsilon=1$, then we only need to consider the case $e=2$, i.e. when $q \equiv  3  \mod  4$, since $\Levi$ is a maximal torus otherwise. In that case $q$ is not a square and thus the group of field automorphisms has odd order. It thus follows again that $\mathcal{B}$ is a cyclic group.
%
%
%
\end{proof}

We now discuss the Clifford theory for the normal inclusion $G \lhd \tilde{G}$. For this, we fix a character $\psi=\Ind_{N_{(e)}(\mu_{(e)})}^{N_{(e)}}(\Lambda^{(e)}(\mu_{(e)})\eta_{(e)})$ and we let $\chi:=R_{L_{(1)}}^{G_{(1)}}(\mu_{(1)})_{\eta_{(1)}}$ be its image under the map $f$. Moreover, we fix a semisimple element $\tilde{s} \in (\Gtilde^\ast)^{F}$ of $2'$-order with $\iota^\ast(\tilde{s})=s$. By the proof of Proposition \ref{equiv1} we find for every character $\mu_{(d)} \in \mathcal{E}_2(L_{(d)},s)$ an extension $\tilde{\mu}_{(d)} \in \mathcal{E}_2(\tilde{L}_{(d)},\tilde{s})$. We also fix characters $\tilde{\eta}_{(d)} \in \Irr(W_{(d)}(\tilde \mu_{(d)}) \mid \eta_{(d)})$ below $\eta_{(d)}$.
%
%
%
%

\begin{corollary}
For $d \in \{1,2\}$, there exists an $N_{(d)} E$-equivariant extension map $\tilde{\Lambda}_{(d)}$ for the set $\mathcal{E}_2(\tilde L_{(d)},\tilde s)$ with respect to $\tilde{L}_{(d)} \lhd \tilde{N}_{(d)}$ by sending $\tilde{\mu} \in \mathcal{E}_2(\tilde L_{(d)},\tilde s)$ to the unique common extension of $\tilde{\mu}$ and $\Res_{N_{\tilde{\mu}}}^{N_{\mu}}(\Lambda(\mu))$ where $\mu \in \mathcal{E}_2(L_{(d)},s)$ is the restriction of $\tilde{\mu}$ to $L_{(d)}$.
\end{corollary}

\begin{proof}
This follows exactly as in the proof \cite[Proposition 3.20]{MS}.
\end{proof}

We can now complete the proof of the inductive conditions:

\begin{theorem}\label{maximal block}
	Let $b$ be one of the blocks from Theorem \ref{AM bijection}.
Then $b$ is AM-good relative to the subgroup $\N_G(\Levi)$.
\end{theorem}

\begin{proof}
	Let $B$ be the unique block of $\N_G(\Levi)$ covering $e_s^{L}$. We show that the bijection $f: \Irr_0(b) \to \Irr_0(B)$ from Theorem \ref{AM bijection} is a strong AM-bijection, see Definition \ref{IAMsuitable}. To show this we want to use Theorem \ref{12}. First observe that the subgroup $M:=\N_{G}(\Levi)$ satisfies the assumptions of Theorem \ref{12} by Lemma \ref{L is suitable}. Let $\chi \in \Irr_0(b)$ and $\chi':=f(\chi)$. By possibly conjugating $\chi$ by an element of $\tilde G$ we can assume by Theorem \ref{star} that the character $\chi$ satisfies condition (i) of Theorem \ref{12}. Using the Butterfly Theorem \cite[Theorem 1.10]{Jordan2} we see that it's enough to show that $\chi$ and $\chi'$ satisfy the remaining conditions in Theorem \ref{12}. Since $f$ is equivariant we deduce that conditions (ii) and (iii) hold, see Proposition \ref{assump}.
	
	As in \cite[Lemma 5.3]{BroughRuhstorfer} one sees that $\tilde{\chi}:=R_{\tilde{L}_{(1)}}^{\tilde{G}_{(1)}}(\tilde \mu_{(1)})_{\tilde \eta_{(1)}} \in \Irr(\tilde{G},e_{\tilde{s}}^{\tilde{G}})$ and $\tilde \psi:=\tilde{\Lambda}(\tilde{\mu}) \tilde{\eta} \in \Irr(\tilde{N}_{(e)},\tilde{B}_e)$. These characters satisfy condition (iv) of Theorem \ref{12} by Proposition \ref{equiv1}. Finally for condition (v) by the remarks after Lemma \ref{block theory} it suffices to show that $\mathrm{bl}(\tilde \psi)^{\tilde G}=\mathrm{bl}(\tilde{\chi})$. This is however guaranteed by Proposition \ref{Brauer morphism}.
\end{proof}

\section{Centralizers of type $A_2(\varepsilon q^3).3$}\label{section 12}

\subsection{Setup and notation}
After having introduced the notation for dealing with the semisimple element $s$ with centralizer of type $A_2(\varepsilon q)^3.3$ in $E_6(\varepsilon q)$, we use this subsection to complete the proof of the iAM-condition in cases (2) and (8) of Table \ref{table}, i.e. of the blocks associated to the semisimple element of $E_6(\varepsilon q)$ with centralizer of type $A_2(\varepsilon q^3).3$. As before, we let $\G$ be a simple, simply connected algebraic group of type $E_6$ with Frobenius endomorphism $F$ as in Section \ref{section 10}. In addition, we let $s \in \T^\ast$ be the semisimple element with centralizer $A_2^3.3$ as in Section \ref{section 10}, but we will consider $\G$ with a different Frobenius endomorphism in order to simplify the calculations for these cases. Throughout this section, we will consider the Frobenius endomorphism
$$F:=F_q v' \text{ if }\varepsilon=1 \text{ and } F:=F_q \gamma_0 v' \text{ if } \varepsilon=-1,$$
where $v' \in V$ is the Weyl group element constructed in the proof of Proposition \ref{Chevalley2}. We observe that $\G(s)^{F}$ is of type $A_2(\varepsilon q^3).3$ and $(\T,F)$ has polynomial order $\Phi_1^3 \Phi_2^3$.

We define
$$m:=n_{\alpha_0}(1) \, {}^{v'} n_{\alpha_0}(1) \, {}^{v'^2} n_{\alpha_0}(1)$$
and
$$n:=n_{\alpha_2}(1) \, {}^{v'} n_{\alpha_2}(1) \, {}^{v'^2} n_{\alpha_2}(1).$$
Our construction together with \cite[Bemerkung 2.1.7(b)]{BS2006} implies that $m,n \in V^{F}=\C_{V}(v')$.
We fix an element $g \in \G$ such that $m=g^{-1} F(g)$. Moreover, $m_{(d_\varepsilon)}:=1$ if $d=1$ and $m_{(d_\varepsilon)}:=m$ if $d=2$.

\begin{notation}\label{notation}
	Set $h:=\mathrm{o}(\gamma_0) \mathrm{o}(m) \mathrm{ord}(v')$. Let $E:= \C_{hm} \times \langle \gamma_0 \rangle$ if $\varepsilon=1$ and $E:=\C_{hm}$ if $\varepsilon=-1$ such that the first factor of $E$ acts on
	${\tilde \G}^{F_p^{2hm}}$ by $\langle F_p \rangle$ and the second by the group $\langle \gamma_0 \rangle$ generated by graph automorphisms.
\end{notation}

Note that we changed the graph automorphism $\gamma$ to $\gamma_0$ in order to ensure that $E$ stabilizes $\G^{F}$, where $F$ is defined as above.
%
Again we denote by
	$$\iota: \tilde\G^{F_q^{2e}} \rtimes E \to \tilde\G^{F_q^{2e}} \rtimes E, \quad x \mapsto {}^g x,$$
the isomorphism induces by conjugation with the element $g$. Observe, that this	is an isomorphism which maps $\G^F \rtimes E$ onto $\G^{mF} \rtimes E$.

\subsection{Construction of Levi subgroups}
%
We now construct the Levi subgroups associated to the semisimple element $s \in (\G^\ast)^F$.
The Levi subgroup $\Levi_{(1_\varepsilon)}$ of $\G$ associated to $s \in (\T^\ast)^{F}$ is the centralizer of the $\Phi_{1_\varepsilon}$-torus $$\mathbf{S}_{(1_\varepsilon)}:=\{h_{\alpha_2+\alpha_3+\alpha_5}(t_1) h_{\alpha_0+\alpha_1+\alpha_6}(t_2) \mid t_1,t_2 \in \overline{\mathbb{F}}_q^\times  \}.$$ 
Calculations in the root system of $E_6$ show that
$$\Levi_{(1_\varepsilon)}=\G_1 \G_2 \mathbf{S}_{(1_\varepsilon)},$$
where $\G_1$ has base $$A_2^{(1)}=\{\alpha_2+\alpha_4+\alpha_5+ \alpha_6, \alpha_1+\alpha_3+\alpha_4 +\alpha_5\}$$
and $\G_2$ has base
$$A_2^{(2)}=\{\alpha_3+\alpha_4+\alpha_5+ \alpha_6, \alpha_1+ \alpha_2 +\alpha_3 + \alpha_4\}.$$
Observe that the roots in the sets $A_1^{(1)},A_2^{(2)}$ are orthogonal to $\alpha_2+\alpha_3+\alpha_5$ and $\alpha_1+ \alpha_0+ \alpha_6$. Moreover, the second root in each of theses sets is the image under $v'$ of the first. In particular, $\Levi_{(1_\varepsilon)}^{F}$ has rational type $A_2(\varepsilon q)^2 \Phi_{1_\varepsilon}^2$, as already predicted by Table \ref{table}.

The $2_\varepsilon$-split Levi subgroup $\mathbf{L}_{(2_\varepsilon)}$ associated to $s \in (\T^\ast)^{mF}$ is the $2_\varepsilon$-split Levi subgroup associated to the $2_\varepsilon$-split torus $\mathbf{S}_{(2_\varepsilon)}:=\{h_{\alpha_0+ \alpha_1 + \alpha_6}(t) \mid t \in \overline{\mathbb{F}}_q^\times \}$ of $(\G,mF)$. In particular, we have $$\Levi_{(2_\varepsilon)}=\G_1 \G_2 \G_3 \mathbf{S}_{(2_\varepsilon)},$$
where $\G_1,\G_2$ are defined as above and $\G_3$ is of type $A_1$ with base $\{\alpha_4\}$. Here, the roots of the root system of $[\mathbf{L}_{(2_\varepsilon)},\mathbf{L}_{(2_\varepsilon)}]$ are orthogonal to the root $\alpha_1+\alpha_0+\alpha_6$. Observe that the element $m$ interchanges $\G_1$ and $\G_2$ so that the Levi subgroup $\Levi_{(2_\varepsilon)}^{mF}$ has rational type $A_2(q^2) A_1(\varepsilon q) \Phi_{(2_\varepsilon)}$ in accordance with Table \ref{table}.

Summarizing our computations we have thus proved:

\begin{lemma}
	
	Assume the notation from above. Then the following hold:
	\begin{enumerate}[label=(\alph*)]
		\item The Levi subgroup $\Levi_{(1_\varepsilon)}$ is the $1_\varepsilon$-split Levi subgroup associated to the semisimple  element $s \in (\G^\ast)^{F'}$.
		\item The Levi subgroup $({}^g\Levi_{(2_\varepsilon)})$ is the $2_\varepsilon$-split Levi subgroup associated to the semisimple element $s \in (\G^\ast)^{F}$.
	\end{enumerate}
\end{lemma}

As before for $d \in \{1,2\}$ we define $L_{(d)}:=\Levi_{(d)}^{m_{(d)} F}$ and $N_{(d)}:=\N_{\G^{v_{(d)} F}}(\Levi_{(d)},e_s^{L_{(d)}})$ and $W_{(d)}:=N_{(d)}/L_{(d)}$.

\begin{proposition}\label{Chevalley3}
	The subgroup $V_{(d_\varepsilon)}=\langle m,n^{d^2} \rangle$ is a section of $W_{(d_\varepsilon)}$.
In particular, $V_{(d_\varepsilon)}$ has central intersection $$Z_{(d_\varepsilon)}:=\Levi_{(d_\varepsilon)}\cap V_{d_\varepsilon}=\langle h_{\alpha_0}(-1) h_{\alpha_1}(-1) h_{\alpha_6}(-1), (h_{\alpha_2}(-1) h_{\alpha_3}(-1) h_{\alpha_5}(-1))^d \rangle \subset \mathbf{S}_{(d_\varepsilon)}$$ with $\Levi_{(d_\varepsilon)}$.
\end{proposition}

\begin{proof}
From Lemma \ref{stabilizer} it follows that $V_{(d_\varepsilon)}$ is a section of $W_{(d_\varepsilon)}$. Moreover, the Chevalley relations \cite[Bemerkung 2.1.7(b)]{BS2006} imply that the unique map $\langle n_{\alpha_0}(1), n_{\alpha_2}(1) \rangle \to \langle m,n \rangle$ with $n_{\alpha_0}(1) \mapsto m$ and $n_{\alpha_2}(1) \mapsto n$ is an isomorphism. Observe that $H\cap \langle n_{\alpha_0}(1), n_{\alpha_2}(1) \rangle= \langle h_{\alpha_0}(-1), h_{\alpha_2}(-1) \rangle$. The claim follows from this.
\end{proof}

As before, for $d \in \{1,2\}$ we define $G_{(d)}:=\G^{v_{(d)} F}$.

\begin{remark}\label{Howlett}
	By the tables in \cite{Howlett}, we have $$\N_{G_{(1_\varepsilon)}}(\Levi_{(1_\varepsilon)})/L_{(1_\varepsilon)} \cong S_3 \times C_2,$$
	where the $S_3$ component centralizes $[\Levi_{(1_\varepsilon)},\Levi_{(1_\varepsilon)}]$ and the central $C_2$ copy permutes the two simple components of $\Levi_{(1_\varepsilon)}$.
As $W_{(1_\varepsilon)} \cong S_3$ does not centralize $[\Levi_{(1_\varepsilon)},\Levi_{(1_\varepsilon)}]$, we observe that $W_{(1_\varepsilon)}$ embedds diagonally in that decomposition of the normalizer.
\end{remark}

We now consider the action of the automorphisms in $E$ on the semisimple element $s$. Recall that ${}^{w_0 F_p} s =s$ if $p \equiv 2 \mod 3$ while ${}^{F_p} s=s$ if $p \equiv 1 \mod 3$. Moreover, $w_0 \gamma_0$ stabilizes $s$. However, $w_0 \notin \C_W(v')$ and so either $\C_{WE}(v',s)=\C_{W(s)}(v')\langle F_p \gamma_0 \rangle$ or $\C_{WE}(v',s)=\C_{W(s)}(v')\langle F_p  \rangle$. This agrees with the computation in the proof of Lemma \ref{block stabilizer}. We set $E':=\C_{WE}(v',s)$ and deduce that $\N_{G_{(d)} E}(L_{(d)},e_s^{L_{(d)}})=N_{(d)} E'$.

\begin{lemma}\label{ext map2}
There exists an $E'$-equivariant extension map $\Gamma$ for $Z_{(d)} \lhd V_{(d)}$.
\end{lemma}

\begin{proof}
	Observe that $E'$ centralizes $V$. In particular, every extension map for $Z_{(d)} \lhd V_{(d)}$ is $E'$-equivariant and therefore an equivariant extension map $\Gamma$ always exists by \cite[Corollary 11.13]{Isaacs} since $V_{(d)}/Z_{(d)} \cong W_{(d)}$ has only cyclic Sylow subgroups.
\end{proof}

The groups $W_{(1)}$ and $W_{(e)}$ have a common ($E'$-stable) Sylow $2$-subgroup $W_2$ with $v_{(1)},v_{(e)} \in \Z(W_2)$. Note that in contrast to the situation of Section \ref{section 10} the group $W_2$ is selfnormalizing in $W_{(1)}$ and $W_{(e)}$.

\begin{lemma}\label{extension map2}
For $d \in \{1,e\}$ there exists an extension map $\Lambda_{(d)}$ for $\mathcal{E}_2(L_{(d)},s)^{W_2}$ with respect to $L_{(d)} \lhd N_{(d)}$ such that
	$\Lambda_{(d)}$ is $\N_{G_{(d)} E}(L_{(d)},e_s^{L_{(d)}})$-equivariant.
\end{lemma}

\begin{proof}
	Observe that it suffices to construct $\Lambda_{(d)}$ on an $E'$-transversal $\mathcal{T}$ of $\mathcal{E}_2(L_{(d)},s)^{W_2}$ and extend it $E'$-equivariantly. Hence, for $\psi \in \mathcal{T}$ it is enough to construct an extension $\Lambda(\psi) \in \Irr((N_{(d)})_\psi \mid \psi)$ with the same stabilizer in $E'$. Observe furthermore that the Levi subgroup $\Levi_{(2_\varepsilon)}$ is only relevant, when $q \equiv 3 \mod 4$ or $\varepsilon=-1$ and therefore $q \equiv 2 \mod 3$. In both cases $q$ is not a square. Hence, the group $E'/\langle m_{(2_\varepsilon)} \hat{F} \rangle$ has always odd order. Since $N_{(2_\varepsilon)}/L_{(2_\varepsilon)} \cong C_2$, an extension map with the desired properties always exists.
	We therefore only need to consider the case $d=1_\varepsilon$. Let $V_{(d)}$ be the section of $W_{(d)}$ constructed in Proposition \ref{Chevalley3}.
	We can therefore write $N_{(d)}$ as $N_{(d)}=L_{(d)} V_{(d)}$. Since $W_{(d)} \cong S_3$, any character of $\mathcal{T}$ has $2$ extensions to its inertia group in $N_{(d)}$. Hence, we may replace $E'$ without loss of generality by its Sylow $2$-subgroup.
	The intersection $Z_{(d)}=L_{(d)} \cap V_{(d)}$ is a central subgroup of $L_{(d)} $ by Proposition \ref{Chevalley3}. Denote $L_0:=[L_{(d)},L_{(d)}] \cong \SL_3(q) \times \SL_3(q)$. We let $\psi \in \mathcal{E}_2(L_{(d)},s)$ and fix $\psi_0 \in \Irr_0(L_0 \mid \psi)$ and denote by $\psi' \in \Irr((L_0)_{\psi_0} \mid \psi_0)$ the Clifford correspondent of $\psi$. Note that the involutions in $W_{(d)}$ act by permuting the two components of $L_0$, while the Sylow $3$-subgroup of $W_{(d)}$ centralizes $L_0$, see also Remark \ref{Howlett}. Since $L_{(d)}/\Z(L_0)L_{0}$ is a $2'$-group and $E'$ is a $2$-group it follows that we can choose $\psi_0$ with the additional property that $$(W_{(d)}E')_{\psi} \leq (W_{(d)}E')_{\psi_0}.$$
	
	Let $\nu \in \Irr(Z_{(d)} \mid \psi)$ be the unique character below $\psi$.
%
	We consider the semidirect product $L_0 \rtimes W_{(d)}$. We have a surjective group homomorphism $L_0 \rtimes W_{(d)} \to L_0 \rtimes W_{(d)}/\C_{W_{(d)}}(L_0)$. Since $W_{(d)}/\C_{W_{(d)}}(L_0)$ acts by permuting the two components of $L_0$, we can extend $\psi_0$ to a representation of $L_0 \rtimes W_{(d)}$ with $W_{(d)}$ in its kernel. By \cite[Proposition 4.2]{Brough}, the character $\psi_0 \Gamma(\nu) \in \Irr(L_0 V_{(d)}\mid \psi_0)$ is therefore a well-defined character which extends $\psi_0$. By Lemma \ref{extension} there exists a unique character $\hat{\psi}' \in \Irr((L_0 V_{(d)})_{\psi'})$ extending both $\psi'$ and $\psi_0 \Gamma( \nu )$.
	
	We then define $\Lambda(\psi):= \Ind_{(L_0 V_{(d)})_{\psi'}}^{L (V_{(d)})_\psi}(\hat{\psi}')$. Note that $(V_{(d)})_\psi=(V_{(d)})_{\psi'}$ since $\psi_0$ is $V_{(d)}$-stable. Hence, the so-constructed character is indeed irreducible. By Mackey's formula $\Lambda(\psi)$ is indeed an extension of $\psi$. We must check that $\Lambda(\psi)$ is $E'_\psi$-invariant. This follows however from the fact that $\psi'$,$\nu$ are both $E'_{\psi}$-invariant and $\Gamma$ is an $E'_\psi$-equivariant extension map.
\end{proof}

\subsection{Construction of the AM-bijection}
Let $b$ be one of the $2$-blocks of $G=\G^F$ associated to the semisimple element $s\in G^\ast$. We let $b_{(d)}$ be the block of $G_{(d)}$ corresponding to the block $b$ of $\G^F$ under the isomorphism induced by $\iota$. Furthermore, we let $(\Levi_{(e)},\lambda_{(e)})$ be the $e$-cuspidal pair of $(\G,m_{(e)}F)$ associated to the block $b_{(e)}$ of $\G^{m_{(e)}F}$. By the proof of \cite[Proposition 4.3]{KessarMalle}, there exists a unique character $\lambda_{(1)} \in \mathcal{E}(L_{(1)},s)$ such that all constituents of $R_{L_{(1)}}^{G_{(1)}}(\lambda_{(1)})$ lie in the $2$-block $b_{(1)}$.

\begin{lemma}\label{alpha2}
There exists a natural bijection $$\Xi: \Irr(L_{(1)},b_{L_{(1)}}(\lambda_{(1)}))^{W_2} \to \Irr(L_{(e)},b_{L_{(e)}}(\lambda_{(e)}))^{W_2}.$$
	
\end{lemma}

\begin{proof}
	We follow the proof of Lemma \ref{alpha}. For $z \in \{1,m \}$, we obtain again $\C_{W(s)}(z)$-equivariant bijections 
	$$\mathcal{E}_2(\T^{zF},s) \to A_z:=\{ t \in (\T^\ast)^{F^\ast z^\ast} \mid t_{2'}=s \}.$$ 
and from this we get again a bijection $$ \mathcal{E}_2(\T^F,s)^{W_2} \to \mathcal{E}_2(\T^{zF},s)^{W_2}.$$

Moreover, Deligne--Lusztig induction yields for $d \in \{1,2\}$ a bijection $b_L(\lambda_{(d)}) R_{T_{(d)}(s)}^{L_{(d)}}:  \mathcal{E}_2(T_{(d)},s) \to  \Irr(L_{(d)},b_{L_{(d)}}(\lambda_{(d)}))$, see the proof of Lemma \ref{local2}. The composition is again the sought bijection. 
\end{proof}

\begin{theorem}\label{AM bijection2}
	Let $b$ be one of the blocks of Table \ref{table} associated to the $e$-cuspidal pair $(\Levi,\lambda)$ and $B$ the unique block of $\N_G(\Levi)$ covering $b_L(\lambda)$. Then there exists an $\N_{\tilde{G} \mathcal{B}}(\Levi)_B$-equivariant bijection $f:\Irr_0(\N_{G}(\Levi),B) \to \Irr_0(G,b)$.
	In particular, the block $b$ is AM-good.
\end{theorem}

\begin{proof}
	The proof is similar to the proof of Theorem \ref{AM bijection}. We will therefore only sketch the proof with emphasis on the differences to the proof of Theorem \ref{AM bijection}.
	We first parametrize the local characters.
	Let $\Lambda^{(e)}$ be the extension map associated to $L_{(e)} \lhd N_{(e)}$ from Lemma \ref{extension map2}.
Again we have a parametrization
	$$\Pi : \mathcal{P}=\{(\mu, \eta) \mid \mu \in \Irr(L_{(e)},b_{L_{(e)}}(\lambda{(e)})), \eta \in \Irr(W_{(e)}(\mu)) \} \to  \Irr(N_{(e)},b_{L_{(e)}}(\lambda_{e}))$$
with $\Pi(\mu, \eta) = \Ind_{N_{(e)}(\mu)}^{N_{(e)}}(\Lambda_{(e)}(\mu)\eta)$.		
Observe that $W_2$ is self-normalizing in $W_{(e)}$ and $\mathcal{E}_2(T_{(e)},s)^{W_2}$ is in bijection with $\Irr(L_{(e)},b_{L_{(e)}}(\lambda_{(e)}))^{W_2}$. A count of the height zero characters therefore gives $|\Irr_0(N_{(e)},b_{L_{(e)}}(\lambda)) |=(q-\varepsilon)_2 2$.
	
 Hence, every height zero character is of the form $\chi=R_{L_{(1)}}^{G_{(1)}}(\mu_0)_{\eta_0}$ for some $\mu_0 \in \mathcal{E}_2(L_{(1)},b_{L_{(1)}}(\lambda))$ and $\eta_0 \in \Irr(W_{(1)}(\mu_0))$, see Corollary \ref{HarishChandra}. For such a character to have height zero we must by Corollary \ref{necessary height zero} necessarily have that $\mu_0$ is of height zero with $2 \nmid |W_{(1)}:W_{(1)}(\mu_0)|$ and $\eta_0 \in \Irr_{2'}(W_{(1)}(\mu_0))$. 
	
	By Lemma \ref{alpha2} we have a bijection
	$$\Xi:\mathcal{E}_2(L_{(e)},s)^{W_2} \to \mathcal{E}_2(L_{(1)},s)^{W_2}.$$
	For  $(\mu,\eta) \in \mathcal{P}$ a pair with $W_2 \leq W_{(e)}(\mu)$ and $\mu_0:=\Xi(\mu)$,
we have again an $E'_\mu$-equivariant McKay-bijection 
	$$f_\mu:\Irr_{2'}(W_{(e)}(\mu)) \to \Irr_{2'}(W_{(1)}(\mu_0)).$$
	Denote by $B_{(e)}$ the unique block of $N_{(e)}$ covering $b_L(\lambda_{(e)})$. By Proposition \ref{equiv1} we obtain a bijection
	$$\Irr_0(N_{(e)},B_{(e)}) \to \Irr_0(G_{(1)},b_{(1)}), \quad \Ind_{N_{(e)}(\mu)}^{N_{(e)}}(\Lambda^{(e)}(\mu)\eta) \mapsto R_{L_{(1)}}^{G_{(1)}}(\Xi(\mu))_{f_\mu(\eta)}.$$
Again, we need to check that the so-obtained map is well-defined.
For this, observe that the number of height zero characters in the block $b$ is equal to the number of height zero characters in the principal block of $A_2(\varepsilon q^3)$, see Lemma \ref{isomorphic blocks}. By \cite[Lemma 3.8]{BroughRuhstorfer}, the latter number is equal to $(q-\varepsilon)_2 2$. Comparing this with the number of height zero characters of the block $B$ obtained above, we deduce that $|\Irr_0(B)|=|\Irr_0(b)|$. In addition, the map is injective by the properties of Harish-Chandra series and its image contains all height zero characters by the arguments from above. Therefore, the map is necessarily a well-defined bijection and yields by applying $\iota$ a bijection
	$f:\Irr_0(\N_{G}(\Levi),B) \to \Irr_0(G,b)$. It is left to check that $f$ is $\N_{\tilde{G} \mathcal{B}}(\Levi,B)=\N_{G \mathcal{B}} (\Levi,B)$-equivariant. But this is again checked as in Theorem \ref{AM bijection}. Hence, the block $b$ is AM-good by Proposition \ref{cyclic AM}.
\end{proof}


\subsection{Some more information on defect groups}\label{more info}

We now turn to the description of defect groups in case (2) or (8) of Table \ref{table}. We keep the notation from the beginning of this section. We denote $F':=F m_{(e)}$, where the Frobenius $F$ and $m_{(e)} \in V^F$ are defined as before. Let $s \in (\T^\ast)^{F'}$ the semisimple element with centralizer $A_2(\varepsilon q^3).3$ from before and $b$ one of the blocks of $\G^{F'}$ associated to $s$. By Theorem \ref{Malle}, the group $P'=\T_2^{F'} \langle m \rangle$ is therefore a Sylow $2$-subgroup of $\G(s)^{F'}$.
Let $D':=\T_2^{F'}$ and $\mathbf{M}:=\C_{\Levi}(D')$. Using \cite[Lemma 13.17]{MarcBook} we obtain $\M=\G_1 \G_2 \mathbf{T}$ in both cases. We use the description of Lemma \ref{Brauer morphism} to see that one of the blocks $b_{D'}$ of $e_s^{\mathbf{M}^{F'}}$ associated to $s$ is a $b$-Brauer pair. Finally, observe that the unique character $\lambda' \in \mathcal{E}({\bM}^{F'},s) \cap \Irr(\bM^{F'},b_{D'})$ is the canonical character associated to the self-centralizing Brauer pair $(D',b_{D'})$ and that $D'=\Z(\bM)_2^{F'}$.

We need the following representation theoretic lemma:

\begin{lemma}\label{aux}
	Let $H_0$ be a normal subgroup of a finite group $H$ and $U \leq H_0$ with $\Z(U) \cong C_3$ and $\Z(U) \leq \Z(H)$. Let $Z \subset \Z(H)$ and $\chi \in \Irr(H)$ with $\Res_{U}^{H}(\chi)=\sum_{\phi \in \Irr(U)} a_\phi \phi$. Then
	\begin{enumerate}[label=(\alph*)]
		\item 

	$\Res^U_{\Z(U)}(\phi) \phi(1)=\Res^H_{\Z(U)}(\chi) \chi(1)$ for all $\phi \in \Irr(U)$ with $a_\phi \neq 0$.
	\item $\Res_{U Z}^{H}(\chi)=\sum_{\phi \in \Irr(U)} a_\phi \hat{\phi}$, where $\hat{\phi}$ is the unique extension of $\phi$ with 	$\Res^U_{Z \Z(U)}(\phi) \hat \phi(1)=\Res^H_{Z\Z(U)}(\chi) \hat \chi(1)$
\end{enumerate}
\end{lemma} 

\begin{proof}
	Let $1\neq z \in \Z(U)$. Assume first that $\chi(z)=n \omega$ for $n=\chi(1)$ and a primitive third root of unity $\omega$. For $\xi \in \{1, \omega^{\pm 1}\}$ denote by $a_\xi$ the sum of degrees of constituents $\phi$ of the restriction  $\Res^{H}_U(\chi)$ with $\phi(z)=\phi(1)\xi$. We obtain $n=\chi(1)=a_1+a_{\omega}+a_{\omega^2}$ and $\omega n=a_1+ a_{\omega} \omega + a_{\omega^2} \omega^2$. Since $-1-\omega=\omega^2$ the second equation rewrites as $0=(a_1-a_{\omega^2}) + (a_\omega -a_{\omega^2}-n) \omega$. This forces $a_1-a_{\omega^2}=0$ and $a_\omega -a_{\omega^2}-n=0$ which together with the first equation implies $a_1=a_{\omega^2}=0$. Hence, in the restriction of $\chi$ only characters with the same character on the center appear. 
	
	If on the other hand $\Z(U)$ is in the kernel of $\chi$, then we obtain the two equations $n=a_1+a_{\omega}+a_{\omega^2}$ and $n=a_1+ a_{\omega} \omega + a_{\omega^2} \omega^2$. This forces again $a_{\omega}=a_{\omega^2}=0$. This shows part (a) also in this case.
	
Part (b) is now obtained by comparing the character values on $Z$ on both sides.
\end{proof}

%

\begin{proposition}\label{compute}
	There exists a character $\zeta' \in \Irr(\C_{G}(P') \mid \zeta)$ which appears with odd multiplicity in $\zeta$ and such that $ \zeta'(1)_2=|\C_G(P'): \Z(D') \cap \C_G(P')|_2$.
\end{proposition}

\begin{proof}
We start by examining the right hand side of the equality in the proposition. We note that $\Z(D') \cap \C_{G}(P')=\Z(P')$, as $\Z(P') \leq D'$. Denote $M:=\bM^{F'}$. We observe that $\C_{[M,M]}(P') \Z(P')$ has odd index in $\C_{M}(P')$. This is because $M/[M,M] \Z(M)$, $[M,M] \cap \Z(M)$ are both of odd order and $\Z(M)_2=D'$.

Moreover, we have  $\C_{[M,M]}(P') \Z(P')/\Z(P') \cong \C_{[M,M]}(P')/\C_{[M,M]}(P') \cap \Z(P')$
and $\C_{[M,M]}(P') \cap \Z(P') \leq \Z(\C_{[M,M]}(P'))$ which is a $2'$-group. Hence, we need to find a constituent with degree $\zeta'(1)_2=|\C_{[M,M]}(P')|_2$.

We first consider case (2).
	Let $\zeta_0=\chi \times \chi \in \Irr([M,M] \mid \zeta)$. Our aim is to construct a constituent $\zeta_0' \in \Irr(\C_{[M,M]}(P') \mid \zeta_0)$ with odd multiplicity. The subgroup $\C_{[M,M]}(P') \cong \SL_3(q)$ embedds diagonally into $[M,M] \cong \SL_3(q) \times \SL_3(q)$. Hence, the constituents of the restriction of a character $\zeta=\chi \times \chi \in \Irr([M,M])$ are precisely the constituents $\zeta' \in \Irr(\mathrm{SL}_3(q))$ of the character $\chi^2 \in \Irr(\mathrm{SL}_3(q))$. Here, the character $\chi$ is a character of degree $\frac{1}{3} \Phi_1(q)^2 \Phi_2(q)$ of $\SL_3(q)$, lying in a Lusztig series with centralizer of type $\Phi_3(q).3 $. Also observe that the diagonal automorphisms induced by $M$ on $[M,M] \cong \SL_3(q) \times \SL_3(q)$ are the ones induced by the diagonal action of $\Delta \GL_3(q)$. It can now be checked with CHEVIE \cite{CH}, that there exists a unique semisimple character $\zeta_0'$ of degree $\frac{1}{3} \Phi_1(q)^2 \Phi_2(q)$ which appears as a constituent (of odd multiplicity) in $\zeta_0$.
	
	Denote $C:=\C_{M}(P')$ and observe that $C$ induces all diagonal automorphisms of $M$ on $[M,M]$. We let $\hat{\zeta} \in \Irr(M_{\zeta_0} \mid \zeta_0)$ be the Clifford correspondent of $\zeta$ and we denote by $\hat{\zeta}_0' \in \Irr(C_{\zeta_0} \mid \zeta_0')$ the character with the same values as $\zeta_0$ on $\Z(M) \cap C $ as in Lemma \ref{aux}. Denote $\zeta':=\Ind_{C_{\zeta_0}}^{C}(\hat{\zeta}_0')$.  Observe that $M=M_{\zeta_0}C$ and $C_{\zeta_0}=C_{\zeta_0'}$. Hence, for the scalar product of $\zeta$ and $\zeta'$, we have by Mackey's formula
	$$\langle \zeta, \zeta' \rangle=\langle \Res_{C}^M \Ind_{M_{\zeta_0}}^{M}(\hat{\zeta}_0), \Ind_{C_{\zeta_0'}}^{C}(\hat{\zeta}_0') \rangle = \langle \Ind_{C_{\zeta_0'}}^{C} (\Res_{C_{\zeta_0}}^{M_{\zeta_0}}(\hat{\zeta}_0)), \Ind_{C_{\zeta_0'}}^{C}(\hat{\zeta}_0') \rangle.$$
By construction, the character $\hat{\zeta}_0'$ is the unique character in its $C$-orbit that appears with non-zero multiplicity in $\Res_{C_{\zeta_0}}^{M_{\zeta_0}}(\hat{\zeta}_0)$. Hence, by Clifford correspondence, we have $\langle \zeta, \zeta' \rangle=\langle \Res_{C_{\zeta_0}}^{M_{\zeta_0}}(\hat{\zeta}_0),\zeta'_0 \rangle$, which is odd.
%
	
	Let us now consider case (8). Note that here $M=[M,M] \Z(M)$. Let $\zeta_0=\Res_{[M,M]}^{M}(\zeta)$. This character of $[M,M] \cong \SL_3(q^2)$ lies in a Lusztig series with centralizer of type $\Phi_3(q^2).3$. Moreover, it is non-trivial on the center and has degree $\frac{1}{3}\Phi_1(q^2)^2 \Phi_2(q^2)$. We now consider its restriction to the centralizer $\C_{[M,M]}(P') \cong \SL_3(q)$.
	
We let $x\in \SL_3(q)$ such that $x$ is conjugate to the diagonal matrix $\mathrm{diag}(\tau,\tau^q,\tau^{q^2})$ with $\tau \in \mathbb{F}_{q^3}^\times$ of order $q^2+q+1$. By \cite{chartable}, $\zeta_0(x)$ is a third root of unity. Moreover in the restriction of $\zeta_0$ to $\SL_3(q)$ by \cite{chartable} and Lemma \ref{aux} only characters $\psi$ of degree $\psi(1)_2=(\Phi_1^2(q) \Phi_2(q))_2$ contribute to $\zeta(x)$ (as all the other characters either vanish on $x$ or are trivial on the center). If all these characters were to appear with even multiplicity in the restriction of $\zeta_0$, then $\zeta_0(x)/2$ would be an algebraic integer. This is however not the cases as $\zeta_0(x)$ is a third root of unity. Hence, there exists a constituent with odd multiplicity as required. We conclude with Lemma \ref{aux}.
\end{proof}

\begin{corollary}
	Let $b$ be one of the blocks of $G$ associated to case (2) or (8) in \ref{table} and $P'$ a Sylow $2$-subgroup of $G(s)$. Then there exists a maximal $b$-Brauer pair $(P',b_{P'})$ with $(D',b_{D'}) \leq (P',b_{P'})$. Moreover, the pair $(D',P')$ is conjugate to the pair $(D,P)$ from \cite[Theorem 1.2]{KessarMalle}. 
\end{corollary}	

\begin{proof}
Let $b$ be associated to the quasi-central $e$-cuspidal pair $(\Levi,\lambda)$ with $\lambda \in \mathcal{E}(\Levi^{F'},s)$ and $b_P'$ be the block associated to the character $\zeta'$ constructed in Proposition \ref{compute}. 
Hence, \cite[Proposition 22.14]{MarcBook} is applicable and we obtain that $(P',b_{P'})$ is a $b$-Brauer pair. Since $|P'|=|\G(s)^{F'}|_2$ it follows that $(P',b_{P'})$ is a maximal $b$-Brauer pair. 

Comparing $(D',P')$ with the description of the tuple $(D,P)$ from \cite[Proposition 2.7]{KessarMalle} shows again that these are $\G^{F'}$-conjugate. 
\end{proof}

\section{Proof of the inductive condition for non-maximal unipotent blocks}\label{section 13}

In this section, we consider the remaining non-maximal unipotent blocks. In particular, $b=b_G(\Levi,\lambda)$ is a unipotent block of $G$ associated to a unipotent $e$-cuspidal pair $(\Levi,\lambda)$ of central defect.

\begin{lemma}\label{unipotent}
Let	$b=b_G(\Levi,\lambda)$ be a unipotent non-maximal block of $G$ associated to a unipotent $e$-cuspidal pair $(\Levi,\lambda)$ of central defect. Moreover, let $\mathbf{H} \in \{\Levi,\G\}$ and $\mu \in \mathcal{E}(\mathbf{H},1)$ a unipotent character. Then $\mu$ extends to $\mathrm{N}_{\tilde{G} \mathcal{B}}(\mathbf{H})$.
\end{lemma}

\begin{proof}
By \cite[Proposition 15.9]{MarcBook}, the character $\mu \in \mathcal{E}(H,1)$ restricts irreducibly to a character $\mu_0$ of $H_0$, where $\mathbf{H}_0=[\mathbf{H},\mathbf{H}]$. Note that by the classification of Lemma \ref{class nonmax}, the group $\mathbf{H}_0$ is always simple of simply connected type and we have an injective map $H_0/\Z(H_0) \hookrightarrow \bH_{\ad}^F$, which factors through $H/ \Z(H)$.
Since $\mu_0$ is unipotent, we can consider it as a character of $S:=H_0/\Z(H_0)$, which is a non-abelian simple group. Now, the group $N:=\N_{\tilde{G} \mathcal{B}}(\mathbf{H})$ acts by automorphisms on $S$. Hence, we obtain $N/\C_N(S) \cong A$, where $S \leq A \leq \mathrm{Aut}(S)$. By \cite[Theorem 2.4, Theorem 2.5]{MalleUnip} and Lemma \ref{class nonmax} the unipotent character $\mu_0 \in \Irr(S)$ extends to $\mathrm{Aut}(S)$. Therefore, the character $\mu_0$ extends to a character $\hat{\mu} \in \Irr(N)$, which is trivial on $\C_N(S)$. Moreover, by the proof of \cite[Theorem 2.4, Theorem 2.5]{MalleUnip}, the so-obtained extension $\hat{\mu}$ is an extension of $\mu$. 
\end{proof}

Let $\T$ be an $F$-stable maximal torus of $\Levi$.  As $\Levi^F=\T^F [\Levi,\Levi]^F$ we can consider $\Irr(\Levi^F/[\Levi,\Levi]^F)$ as a subgroup of $\Irr(\T^F)$. By duality we obtain a bijection $(\T^\ast)^{F^\ast} \to \Irr(\T^F)$ which by \cite[Equation 8.19]{MarcBook} restricts to a bijection
$$\Irr(\Z(\Levi^\ast)^{F^\ast}) \to \Irr(\Levi^F/[\Levi,\Levi]^F), \quad z \mapsto \hat{z}.$$
With this in mind, we can formulate the next lemma.

\begin{lemma}\label{linear}
	Let $\mu=\hat{z} \in \Irr_(L)$ be a linear character.
	\begin{enumerate}[label=(\alph*)]
		\item 
		The character $\mu$ extends to its stabilizer in $\N_{G\mathcal{B}}(\Levi)$.
		\item Any extension of $\mu$ to $\N_G(\Levi)_\mu$ is $\tilde{L}$-stable if and only if $\mathrm{N}_{\C_{G^\ast}(z)}(\Levi^\ast)=\C^\circ_{G^\ast}(z)$.
	\end{enumerate}
\end{lemma}

\begin{proof}
As $\Levi$ is an $e$-split Levi subgroup it follows that $\Levi$ is the centralizer of an $e$-split torus $\mathbf{S}_0$ of $\G$. Consequently, $\mathbf{S}_0$ is contained in a Sylow $e$-torus $\mathbf{S}$ of $\G$. Its centralizer $\T:=\C_{\G}(\mathbf{S})$ is by \cite[Lemma 3.17]{Cabanesgroup} a maximal torus of $\G$ and by construction it centralizes $\mathbf{S}_0$. In particular, $\T$ is a maximal torus of $\Levi$. Consider the restriction $\mu_0:=\Res_{\T^F}^{\Levi^F}(\mu) \in \Irr(\T^F)$ of $\mu$. By \cite[Theorem 3.1]{MS}, the character $\mu_0$ extends to a character $\mu_1 \in \Irr(\N_{G \mathcal{B}}(\T)_{\mu_0})$. Since all Sylow $e$-tori of $\Levi$ are $L$-conjugate by \cite[Theorem 25.11]{MT}, it follows that $\N_{G \mathcal{B}}(\Levi) = \Levi^F \N_{G \mathcal{B}}(\Levi,\T)$.
 From this we conclude that $\mu$ extends to $\N_{G \mathcal{B}}(\Levi)_\mu$ by Lemma \ref{extension}.

For part (b), let $\hat{z}'$ be an extension of $\hat{z}$ to its inertia group in $\N_G(\Levi)$ and fix $t \in \tilde{L}$ with $\tilde{L}=\langle L, t \rangle$. Observe that by \cite[Lemma 5.8 (b)]{Spaeth} there exists a character $\nu \in \Irr(\N_G(\Levi)_{\hat{z}})$ with ${}^t \hat{z}'= \hat{z}' \nu$, whose kernel is equal to $\mathrm{ker}(\nu)=\N_G(\Levi)_{\widehat{\tilde{z}}}$, where $\tilde{z} \in \Z(\tilde{L}^\ast)$ satisfies $\iota^\ast(\tilde{z})=z$. By Clifford theory, the extension $\hat{z}'$ is $\tilde{L}$-stable if and only if $\N_G(\Levi)_{\hat{z}}= \N_G(\Levi)_{\widehat{\tilde{z}}}$. By Proposition \ref{CE}, the subgroup $W_G(\Levi)_{\hat{z}}$ corresponds via duality to the subgroup $\mathrm{N}_{\C_{G^\ast}(z)}(\Levi^\ast)/L^\ast$ and for $w \in W_G(\Levi)_{\hat{z}}$ we have $w \in W_G(\Levi)_{\widehat{\tilde{z}}}$ if and only if $w^\ast \in \C^\circ_{G^\ast}(z)=\iota^\ast(\C^\circ_{\tilde{G}^\ast}(\tilde{z}))$.
%
\end{proof}

Using the previous two statements we can now parametrize the height zero characters of a local block associated to $b$. The character $\lambda$ is of central defect and so the 
block $b_L(\lambda)$ is nilpotent with defect group $\Z(L)_2$. We deduce that $$\Irr_0(L,b_L(\lambda))=\{\lambda \hat{z} \mid z \in \Z(\Levi^\ast)_2^F \}.$$
For each $\mu= \lambda \hat{z}$ choose a character $\Lambda(\mu)$ of $\N_G(L,{\mu})$ extending it (such a character always exists whenever the Sylow subgroups of $W_G(L,\lambda)$ are cyclic; for the remaining cases this follows from Lemma \ref{unipotent} and Lemma \ref{linear})
According to (the proof of) Lemma \ref{L is suitable}(ii) there exists a unique block $B$ of $\N_G(\Levi)$ covering $b_L(\lambda)$.
By Clifford theory, we obtain
$$\Irr(\N_G(\Levi),B)= \{ \Ind_{{\N_G(\Levi,\lambda)_\eta}}^{\N_G(\Levi)}(\Lambda(\mu) \eta) \mid \eta \in \Irr_0(W_G(\Levi,\lambda)_\mu) \text{ and } \mu \in \Irr(L,b_L(\lambda))\}.$$
We deduce that a character $\Ind_{{\N_G(\Levi,\lambda)_\eta}}^{\N_G(\Levi)}(\Lambda(\mu) \eta)$ has height zero if and only if the inertia group of $\mu$ in $W_G(\Levi,\lambda)$ contains a Sylow $2$-subgroup of $W_G(\Levi,\lambda)$. Fix a Sylow $2$-subgroup $W_2$ of $W_G(\Levi,\lambda)$ and observe that $W_2$ is self-normalizing in $W_G(\Levi,\lambda)$. Therefore, we obtain 
$$\Irr_0(\N_G(\Levi),B)= \{ \Ind_{{\N_G(\Levi,\lambda)_\eta}}^{\N_G(\Levi)}(\Lambda(\mu) \eta) \mid \eta \in \Irr_0(W_G(\Levi,\lambda)_\mu) \text{ and } \mu \in \Irr(L,b_L(\lambda))^{W_2}\},$$
where every character in the set on the right-hand side appears exactly once.
%
%

\begin{lemma}\label{classifaction local unip}
	Let $B$ be the unique block of $\N_G(\Levi)$ covering $b_L(\lambda)$. Then the number of height zero characters in $B$ for the unipotent blocks in Lemma \ref{class nonmax} is as follows:
	\begin{enumerate}[label=(\roman*)]
		\item $|\Irr_0(B)|=|P/P'|=4$; 
		\item $|\Irr_0(B)|=|P/P'|=8$.
	\end{enumerate}
\end{lemma}

\begin{proof}
As already observed in Lemma \ref{nonmaximal} the defect groups in case (i) are dihedral and the equality follows again from \cite[Theorem 8.1]{Sambale} or by a similar argument as in case (ii).
	Suppose therefore that we are in case (ii): The action of $W_{G^\ast}(\Levi^\ast) \cong S_3 \times C_2$ on $\Z(L^\ast)=\Phi_1^2$ is described in the proof of \cite[Proposition 6.4]{KessarMalle}. We may therefore choose a Sylow $2$-subgroup $W_2 \cong C_2 \times C_2$ of $W_{G^\ast}(\Levi^\ast)$ such that $C_2 \leq S_3$ permutes the two copies of $\Phi_1$ and the central subgroup $C_2$ of $W_G(\Levi)$ acts on $\Z(\Levi)_2^F$ by inversion. The only characters $\mu=\lambda \hat{z} \in \Irr(L,b_L(\lambda))$ which are therefore fixed under $W_2$ (isomorphic to $C_2 \times C_2$) are the ones with $z=\pm (1,1) \in \Z(L^\ast) \cong \mathbb{F}_q^\times \times \mathbb{F}_q^\times$. Both of these characters contribute four height zero characters to the block $B$. Hence, there are $8$ height zero characters in total.
\end{proof}

\begin{proposition}\label{unipotent AM}
Assume that we are either in case (i) or (ii) of Lemma \ref{class nonmax}. Then there exists an $\N_{G \mathcal{B}}(\Levi)_B$-equivariant bijection $\Irr_0(\N_G(\Levi),B) \to \Irr_0(G,b)$. Moreover, the block $b$ is AM-good.
\end{proposition}

\begin{proof}
Let us first describe the action of automorphisms on the height zero characters of the block $b$. In case (i) the same reasoning applies as in \cite{Jordan3}. In particular, the two unipotent characters are invariant under $\mathrm{Out}(G)$ and the two non-unipotent characters are permuted by the diagonal automorphisms, see also Lemma \ref{linear}(b). In case (ii) we note that by Example \ref{trivial action} all 8 characters are necessarily stable under all automorphisms of $G$. 
	
Let us now consider the action of automorphisms on the height zero characters of the block $B$. We claim that every height zero character $\psi \in \Irr_0(\N_G(\Levi),B)$ extends to its inertia group in $\N_{\tilde{G} \mathcal{B}}(\Levi)$.
We first observe that $\lambda$ has an extension $\hat{\lambda}$ to $\N_{\tilde{G} \mathcal{B}}(\Levi)$ by Lemma \ref{unipotent}. It therefore remains to describe extensions of the characters $\psi \in \Irr_0(\N_G(\Levi),B)$ which cover the character $\mu=\lambda\hat{z}$, with $z \in \Z(\Levi)_2^F$ of order $2$ as described in the proof of Lemma \ref{classifaction local unip}. First note that by Lemma \ref{linear}, the character $\hat{z}$ extend to a character $\hat{z}'$ in its inertia group in $\N_{G \mathcal{B}}(\Levi)$. In case (i) note that we have $\C_{\G^\ast}(z)=(E_6(q) \Phi_1).2$.
Therefore, by Lemma \ref{linear} any such extension $\hat{z}' \in \Irr(\N_G(\Levi))$ of the character $\hat{z}$ to $\N_G(\Levi)$ is not invariant under diagonal automorphisms. In particular, we have $\N_{\tilde{G} \mathcal{B}}(\Levi)_\psi=\N_{G \mathcal{B}}(\Levi)_\psi$. Since $\N_{G \mathcal{B}}(\Levi)_\psi/\N_G(\Levi)$ is cyclic, the character $\psi$ thus extends to its inertia group,

We now discuss the underlying central characters. Note that in case (ii), $\Z(G)=1$ so there is nothing to do. In case (i) the two unipotent characters in $\Irr_0(G,b)$ (resp. in $\Irr_0(\N_G(\Levi),B)$) have $\Z(G)$ in their kernel. Moreover, for $1 \neq z \in G^\ast$ with $\mathcal{E}(G,z) \cap \Irr_0(G,b) \neq \emptyset$ we have $z \in [G^\ast,G^\ast]$ by \cite{Luebeck} and thus, the characters in $\mathcal{E}(G,z)$ and $\mathcal{E}(L,z)$ have $\Z(G)$ in their kernel as well.

By the description of automorphisms it is now clear how to construct an $\N_{\tilde{G}\mathcal{B}}(\Levi)_B$-equivariant bijection $\Irr_0(\N_G(\Levi),B) \to \Irr_0(G,b)$. The block $b$ is therefore AM-good by Lemma \ref{cyclic AM}.
%
\end{proof}

%
%

%
%
%
\section{Proof of the main theorem}\label{section 14}

\subsection{Navarro--Späth's Theorem}

 Recall that a group $X$ is said to be involved in a finite group $Y$, if there exist $H \lhd K \leq Y$ such that $X \cong K/H$.

We recall \cite[Theorem 7.1]{JEMS}. In the statement of the theorem, $\Irr_0(Y \mid P)$ denotes the set of height zero characters in a block of $Y$ with defect group $P$.

\begin{theorem}\label{NavarroSpaeth}
	Let $\mathcal{S}$ be a collection of finite non-abelian simple groups of order divisible
	by $\ell$ that satisfy the inductive AM-condition for the prime $\ell$. Assume that $N$ is a finite group such that all the non-abelian simple groups involved in $N$
	of order divisible by $\ell$ are in $\mathcal{S}$. Suppose that $N \lhd H$ for some finite group $H$, and let $P$
	be any $\ell$-subgroup of $N$. Then there exists an $\N_H(P)$-equivariant bijection
	$\Pi : \Irr_0(N \mid P) \to \Irr_0(\N_N(P) \mid P)$
	with
	$$(H_\tau , N , \tau ) \geq_b (\N_H(P)_{\Pi(\tau)}
	, \N_N (P), \Pi(\tau))$$
	for every $\tau \in \Irr_0(N \mid P)$.
\end{theorem}

\begin{definition}
	Let $S$ be a finite non-abelian simple group. We say that $S$ is a minimal counterexample if $S$ is not AM-good (with respect to the prime $\ell$) but all finite simple groups $S'$ with $|S'| < |S|$ are AM-good (with respect to the prime $\ell$). 
\end{definition}

We use the following consequence of Proposition \ref{NavarroSpaeth}. This will become important in the proof of Theorem \ref{main proof}.

\begin{corollary}\label{equivalent}
	Let $b$ be a block of the universal covering group $G$ of a finite simple group $S$ of Lie type and assume that all finite simple groups $S'$ with $|S'| < |S|$ are AM-good. Then the following are equivalent:
	\begin{enumerate}[label=(\roman*)]
		\item The block $b$ is AM-good relative to the normalizer of its defect group.
		\item The block $b$ is AM-good.
	\end{enumerate}
\end{corollary}

\begin{proof}
	Clearly (i) implies (ii). So we need to show that if $b$ is AM-good, then it is also AM-good relative to the normalizer of its defect group $P$. By the assumption of (ii), there exists a strong iAM-bijection $\Irr_0(G,b) \to \Irr_0(M,B)$ for some suitable $\mathrm{Aut}(G)_b$-stable subgroup $M$ with $\N_G(P) \leq M$. Now by assumption all simple non-abelian groups involved in $M$ are AM-good. Therefore, by Theorem \ref{NavarroSpaeth} there exist an iAM-bijection $\Irr_0(M,B) \to \Irr_0(\N_G(P),B_P)$. In particular, the composition $\Irr_0(G,b) \to \Irr_0(\N_G(P),B_P)$ is again an iAM-bijection, see \cite[Lemma 3.8(a)]{JEMS}. By \cite[Remark 3.15]{Jordan2}, the block $b$ is therefore AM-good relative to the normalizer of its defect group.
\end{proof}

\subsection{Proof of the inductive conditions}

%
%
\begin{proposition}\label{cyclic}
	Let $b$ be a strictly quasi-isolated $2$-block of an exceptional group of Lie type $G$ as in Section \ref{section 2} which is not associated to a semisimple element with centralizer $A_2(\pm q^3).3$ or ${}^3 D_4(\pm q)$ in $E_6(\pm q)$. Assume that $\mathrm{Out}(G)_b$ is cyclic and assume that $S:=G/\mathrm{Z}(G)$ is a minimal counterexample. Then the block $b$ is AM-good.
\end{proposition}

\begin{proof}
	We use the criterion of Lemma \ref{cyclic AM}. By Proposition \ref{equivariant Jordan} there exists a block $c$ of $G^\circ(s)$ and an
	$\N_{\tilde{G} \mathcal{B}}(G^\circ(s),c)$-equivariant bijection $\Irr_0(G,b) \to \Irr_0(G^\circ(s),c)$, which preserves central characters. 
	
	Let $(D,b_D)$ be the $b$-Brauer pair constructed in Lemma \ref{local} resp. Lemma \ref{properties} such that $\N_G(D,b_D)=\N_{G(s)}(D)$. We obtain that $C_D=b_D$.

By assumption, $S$ is a minimal counterexample.	Hence, by Theorem \ref{NavarroSpaeth} there exists an $\mathrm{Aut}(G^\circ(s))_{c,P}$-equivariant bijection $\Irr_0(G^\circ(s),c) \to \Irr_0(\N_{G^\circ(s)}(D),C_P)$, which preserves central character. In addition, Theorem \ref{NavarroSpaeth} also yields an $\mathrm{Aut}(G^\circ(s))_{c,D}$-equivariant bijection $\Irr_0(\N_{G^\circ(s)}(D),C_D) \to \Irr_0(\N_{G^\circ(s)}(D),C_P)$ preserving central characters. Combining these bijections yields an $\mathrm{Aut}(G^\circ(s))_{c,D}$-equivariant bijection $\Irr_0(G^\circ(s),c) \to \Irr_0(\N_{G^\circ(s)}(D),C_D)$, which preserves central character.

Now induction of characters yields a natural bijection $\Ind_{\N_G(D,b_D)}^{\N_G(D)}:\Irr_0(\N_G(D,b_D),b_D) \to \Irr_0(\N_G(D), B_D)$. We conclude that the composition of these bijections yields an $\N_{\tilde G \mathcal{B}}(D,B_D)$-equivariant bijection $\Psi:\Irr_0(G,b) \to \Irr_0(\N_G(D), B_D)$, which preserves central characters. Let $\chi \in \Irr_0(G,b)$ and $\chi'=\Psi(\chi)$. If $G$ is not of type $E_7(q)$, then $\chi$ and $\chi'$ satisfy the conditions of Lemma \ref{cyclic AM} by Lemma \ref{block bijection}.
	On the other hand, if $G$ is of type $E_7(q)$ then the condition of Lemma \ref{cyclic AM} are satisfied by Lemma \ref{trivial on center}. Hence, by Lemma \ref{cyclic AM} the block $b$ is AM-good.
\end{proof}

\begin{lemma}\label{cyclic2}
	Let $b$ be one of the blocks of $G=E_6(\pm q)$ associated to a semisimple element with centralizer ${}^3 D_4(q)$. Then $b$ is AM-good relative to the normalizer of its defect group.
\end{lemma}

\begin{proof}
By the main result of \cite{Dat} there exists a Morita equivalence between $\mathcal{O} G(s) e_s^{G^\circ(s)}$ and $\mathcal{O} G e_s^{G}$ which is obtained by lifting the Deligne--Lusztig induction functor $R_{G^\circ(s)}^{G}$. In particular, there exists a maximal $b$-Brauer pair $(P,b_P)$ with $P \leq G(s)$. By Lemma \ref{dade} and Lemma \ref{block bijection}, the blocks of $\mathcal{O} G(s) e_s^{G^\circ(s)}$ are Morita equivalent by restriction to the unique block $c$ of $\mathcal{O} G^\circ(s) e_s^{G^\circ(s)}$. We deduce that Deligne--Lusztig induction followed by truncation with the idempotent $b$ induces a Morita equivalence between $\mathcal{O}G^\circ(s) c$ and $\mathcal{O} G b$. This induces a character bijection $b R_{G^\circ(s)}^{G}:\Irr_0(G^\circ(s),c) \to \Irr_0(G,b)$ which is $\N_{\tilde{G} \mathcal{B}}(\G^\circ(s),b)$-equivariant. Hence, as in the proof of \cite[Theorem 10.6]{Jordan3}, there exists a maximal $c$-Brauer pair $(P,c_P)$ together with an $\N_{\tilde{G} \mathcal{B}}(\G^\circ(s),P,C_P)$-equivariant bijection
$$R_{\N_{G^\circ(s)}(P)}^{\N_{G}(P)}:\Irr_0(\N_{G^\circ(s)}(P),C_P) \to \Irr_0(\N_G(P),B_P).$$
We are thus exactly in the situation of Lemma \ref{cyclic} and the proof of said lemma shows the statement.
\end{proof}

The blocks in Lemma \ref{cyclic2} could have been treated alternatively with the same methods as the ones developed in Section \ref{12}.

We are now ready to prove our main theorem.

\begin{theorem}\label{main proof}
	The Alperin--McKay conjecture holds for the prime $2$.
\end{theorem}

\begin{proof}
	According to Späth's reduction theorem \cite[Theorem C]{IAM}, it suffices to show that all finite simple non-abelian groups $S$ are AM-good for the prime $2$.
	
	We assume that $S$ is a minimal counterexample. In particular, by \cite[Section 7,Section 8]{BroughRuhstorfer} we can assume that its universal covering group $G$ is a group of Lie type defined over a field of odd characteristic with exceptional root system. Assume first that the block $b$ of $G$ is not quasi-isolated and let $\bM^\ast$ be the minimal Levi subgroup of $\G^\ast$ containing $\C_{\G^\ast}(s)$. In particular, by assumption on the minimality of $S$, the $2$-blocks of the simple components of $[\bM,\bM]^F$ are AM-good relative to the normalizer of their defect group, see Corollary \ref{equivalent}. By the main theorem of \cite{Jordan2}, the block $b$ is therefore AM-good relative to the normalizer of its defect group. We may therefore assume that $b$ is a quasi-isolated block. Suppose first that $b$ is a unipotent block. If $b$ is the principal block, then the result follows from \cite[Theorem 8.3]{BroughRuhstorfer} and if $b$ is not the principal block, the result follows from Proposition \ref{unipotent AM}. Therefore, we can assume that $b$ is not unipotent. If $\mathrm{Out}(G)_b$ is cyclic, the result follows from Proposition \ref{cyclic}, while if $\mathrm{Out}(G)_b$ is non-cyclic, the result follows from Corollary \ref{equivalent} and Proposition \ref{maximal block}.
\end{proof}


\end{document}